\documentclass[11pt]{article}
\usepackage[titletoc,title]{appendix}
\usepackage{amsmath,amsfonts,amssymb,bm,amsthm}
\usepackage[colorlinks,citecolor=blue,linkcolor=blue,pagebackref]{hyperref}
\usepackage{tikz}
\usetikzlibrary{shapes,arrows}
\usetikzlibrary{intersections,shapes.arrows}
\usetikzlibrary{decorations.pathreplacing}
\usetikzlibrary{arrows.meta}
\numberwithin{equation}{section}
\newtheorem{theorem}{Theorem}[section]
\newtheorem{lemma}[theorem]{Lemma}
\newtheorem{proposition}[theorem]{Proposition}
\newtheorem{corollary}[theorem]{Corollary}

\theoremstyle{definition}
\newtheorem{example}[theorem]{Example}
\newtheorem{definition}[theorem]{Definition}

\theoremstyle{remark}
\newtheorem{remark}[theorem]{Remark}

\usepackage{multicol}
\usepackage{longtable}

\usepackage[T1]{fontenc}
\usepackage[top=1in, bottom=1in, left=1in, right=1in]{geometry}
\usepackage{fancyhdr}
\newcommand{\dr}{\mathrm{d}}
\newcommand{\dist}{\operatorname{dist}}
\newcommand{\Ric}{\operatorname{Ric}}
\newcommand{\Riem}{\operatorname{Riem}}
\newcommand{\Sc}{\operatorname{Sc}}
\newcommand{\curl}{\operatorname{curl}}
\newcommand{\prin}{\mathrm{prin}}
\newcommand{\sub}{\mathrm{sub}}
\newcommand{\loc}{\mathrm{loc}}
\usepackage{enumerate}
\usepackage{xcolor,cancel}
\usepackage{relsize}

\renewcommand{\tilde}{\widetilde}

\usepackage{orcidlink}

\allowdisplaybreaks

\usepackage{pifont}
\renewcommand*{\backrefalt}[4]{%
\ifcase #1 %
No citations%
\or
\ding{43}~p.~#2%
\else
\ding{43}~pp.~#2%
\fi}

\begin{document}

\title{Beyond the Hodge Theorem:\\ curl and asymmetric pseudodifferential projections}
\author{Matteo Capoferri\,\orcidlink{0000-0001-6226-1407}\thanks{MC:
Maxwell Institute for Mathematical Sciences
\&
Department of Mathematics,
Heriot-Watt University,
Edinburgh EH14 4AS, UK; m.capoferri@hw.ac.uk
\textit{and}
Dipartimento di Matematica ``Federigo Enriques'', Università degli Studi di Milano, Via C.~Saldini 50, 20133 Milano, Italy;
 matteo.capoferri@unimi.it,
\url{https://mcapoferri.com}.
}
\and
Dmitri Vassiliev\,\orcidlink{0000-0001-5150-9083}\thanks{DV:
Department of Mathematics,
University College London,
Gower Street,
London WC1E~6BT,
UK;
D.Vassiliev@ucl.ac.uk,
\url{http://www.homepages.ucl.ac.uk/\~ucahdva/}.
}}

\renewcommand\footnotemark{}

\date{21 September 2026}

\maketitle

\vspace{-.6cm}

\begin{abstract}
We develop a new approach to the study of spectral asymmetry.
Working with the operator $\operatorname{curl}:=*\mathrm{d}$
on a connected oriented closed Riemannian 3-manifold,
we construct, by means of microlocal analysis,
the asymmetry operator ---
a scalar pseudodifferential operator of order $-3$.
The latter is completely determined by the Riemannian manifold and its orientation,
and encodes information about spectral asymmetry.
The asymmetry operator generalises and contains the classical eta invariant traditionally associated with the asymmetry of the spectrum, which can be recovered by computing its regularised operator trace. 
Remarkably, the whole construction is direct and explicit.

\


\

{\bf 2020 MSC classes: }
primary
58J50; 
secondary
35P20, 
35Q61, 
47B93, 
47F99; 
58J28, 
58J40. 

\end{abstract}

\tableofcontents

\allowdisplaybreaks

\section{Main results}
\label{Main results}

Let $(M,g)$ be a connected closed Riemannian manifold of dimension $d\ge2$ and let $\Omega^k(M)$ be the Hilbert space of real-valued $k$-forms, $1\le k\le d-1$. Hodge's Theorem \cite[Corollary~3.4.2]{jost} tells us that $\Omega^k(M)$ decomposes into a direct sum of three orthogonal closed subspaces
\begin{equation}
\label{Hodge decompostion}
\Omega^k(M)
=
\dr\Omega^{k-1}(M)
\oplus
\delta\Omega^{k+1}(M)
\oplus
\mathcal{H}^k(M),
\end{equation}
where $\dr\Omega^{k-1}(M)$, $\delta\Omega^{k+1}(M)$ and $\mathcal{H}^k(M)$ are the Hilbert subspaces of exact, coexact and harmonic $k$-forms, respectively. The overarching idea of our paper is that in the special case when $d=3$, $k=1$ and $M$ is oriented the space $\delta\Omega^{k+1}(M)$ admits a further decomposition into a distinguished pair of orthogonal Hilbert subspaces $\delta\Omega_\pm^{k+1}(M)$. Moreover, this further decomposition is effectively described in terms of pseudodifferential projections for which the symbols can be written down explicitly in terms of curvature and its covariant derivatives. Remarkably, this leads to a new approach to the study of \emph{spectral asymmetry}.

The study of spectral asymmetry, that is, the difference in the distribution of positive and negative eigenvalues of (pseudo)differential operators, has a long and noble history, initiated by the seminal series of papers by Atiyah, Patodi and Singer \cite{asymm1,asymm2,asymm3,asymm4}. In a nutshell, the classical approach goes as follows. One considers the \emph{eta function} of the operator at hand --- say, curl or Dirac --- defined as
\begin{equation}
\label{eta function intro}
\eta(s):=\sum_{\lambda_k\ne 0}\operatorname{sgn}(\lambda_k)\, |\lambda_k|^{-s}, \qquad s\in \mathbb{C},
\end{equation}
where $\lambda_k$
are the nonzero eigenvalues of the operator. The complex function $\eta(s)$ can be easily shown to be holomorphic for $\operatorname{Re} s>d$; one then tries to give a meaning to the quantity $\eta(0)$, called \emph{eta invariant}, by examining the meromorphic extension of \eqref{eta function intro} for $\operatorname{Re}s<d$ and showing that there is no pole at $s=0$. The motivation for considering $\eta(0)$ as a measure of spectral asymmetry is that, when the operator in question is simply a Hermitian matrix, $\eta(0)$ is precisely the number of positive eigenvalues minus the number of negative eigenvalues. Classically, it has been shown that $\eta(0)$ is a geometric invariant which can be computed resorting to complex analysis and algebraic topology.

\

Perhaps surprisingly, there is almost no literature on the spectrum of curl on a 3-manifold with or without boundary. Indeed, if one looks carefully, the overwhelming majority of existing papers deal, effectively, with $\curl^2$ as opposed to $\curl$ itself. Many authors 
\cite{safarov_curl,birman_curl, birman_curl2, gureev, filonov_curl1, filonov_curl2} have studied Maxwell's equations on domains in $\mathbb{R}^3$ subject to appropriate boundary conditions, which leads to the analysis of spectral problems of the form
\begin{equation}
\label{Maxwell resonator}
\begin{pmatrix}
0&i\curl\\
-i\curl&0
\end{pmatrix}
\begin{pmatrix}
E\\
B
\end{pmatrix}
=
\lambda
\begin{pmatrix}
E\\
B
\end{pmatrix}.
\end{equation}
One would think that the spectral problem \eqref{Maxwell resonator} reduces to the spectral problem for $\curl$ by means of a unitary transformation
\begin{equation}
\label{unitary transformation}
\begin{pmatrix}
E\\
B
\end{pmatrix}
\mapsto
\begin{pmatrix}
v_+\\
v_-
\end{pmatrix}
:=
\frac{1}{\sqrt{2}}
\begin{pmatrix}
-i&1\\
i&1
\end{pmatrix}
\begin{pmatrix}
E\\
B
\end{pmatrix}
\end{equation}
but this argument fails because physically meaningful boundary conditions do not agree with \eqref{unitary transformation}. The presence or absence of a boundary is a major issue in the subject.

To the best of our knowledge, \textcolor{black}{one of the few papers} dedicated to the study of the spectrum of $\curl$ on a closed oriented Riemannian manifold \textcolor{black}{--- at least, in the way it is understood in the current paper ---} is \cite{baer_curl}. In \cite{baer_curl} the author examined the qualitative properties of the spectrum of $\curl$ depending on the dimension $d$ and established one-term asymptotic formulae (Weyl law) for the (global) eigenvalue counting functions, alongside a number of explicit examples. Although not \textcolor{black}{entirely aligned} in spirit \textcolor{black}{with} the aims of the current paper, we refer the reader to \textcolor{black}{\cite{colbois06,jammes,peralta1,peralta2,peralta3,gerner,kepplinger,lin1,lin2} and references therein, examining the spectrum of $\curl$ (or $\curl^2$) from different perspectives.}

\

Let us now formulate our problem and state our main results, postponing proofs and detailed arguments until later sections. In the rest of this paper, unless otherwise stated, the dimension $d$ is 3 and the manifold is orientable and oriented.
Let $*$ be the Hodge dual (see Appendix~\ref{Exterior calculus} for our sign convention), and define $\curl$ to be the differential expression 
\begin{equation}
\label{curl differential experssion}
\curl:=\ast \dr
\end{equation}
acting in $\Omega^1(M)$.

We will show in Section~\ref{The operator curl} --- more precisely, in Theorem~\ref{theorem properties of curl} --- that the differential expression \eqref{curl differential experssion} gives rise to a self-adjoint operator (denoted by the same symbol) in the space of coexact 1-forms $\delta\Omega^2(M)$. The latter has discrete spectrum, accumulating to both $+\infty$ and $-\infty$, and not necessarily symmetric about zero.
As a way of illustrating that spectral asymmetry actually occurs, in Appendices~\ref{The spectrum of the Laplacian on a Berger sphere} and~\ref{The spectrum of curl on a Berger sphere} we write down explicitly the spectrum of $\curl$ on a Berger sphere and compute the corresponding (classical) eta invariant.

\

Note that the definition of $\operatorname{curl}$ does not involve the concept of connection. In this respect, $\operatorname{curl}$ is one of the most fundamental operators of mathematical physics, like the Laplace--Beltrami operator.
Furthermore, $\curl$ lies at the heart of (homogeneous) Maxwell's equations. Namely, seeking solutions harmonic in time reduces Maxwell's equations on $M\times\mathbb{R}$ to the spectral problem for $\curl$. And the solution of the Cauchy problem for Maxwell's equations
can be expressed in terms of eigenvalues and eigenforms of $\curl$, see formula \eqref{solution to Maxwell as a series}.
Finally, it is worth emphasising that the sign of the eigenvalues of curl has a physical meaning: it coincides with the sign of electromagnetic chirality of time-harmonic polarised solutions of Maxwell's equations generated by the corresponding $\curl$ eigenpair (see \eqref{harmonic solution 1}, \eqref{harmonic solution 2}). We refer the reader to Appendix~\ref{Maxwell's equations and electromagnetic chirality} for further details and a self-contained exposition of this fact. These are but few of the reasons why a more detailed examination of the spectrum of curl --- and in particular of its spectral asymmetry --- is very timely, if not overdue.

\

As a first step, we introduce the following definitions.

\begin{definition}
\label{definition of the operator P0}
The operator $P_0$ is defined as the extension to $\Omega^1(M)$ of the identity operator on $\dr\Omega^0(M)$. Here the extension is specified by the requirement that the orthogonal complement of $\dr\Omega^0(M)$ in $\Omega^1(M)$ maps to zero.
\end{definition}

\begin{definition}
\label{definition of the operators Ppm}
Let
Let $\theta:\mathbb{R}\to\mathbb{R}$,
\begin{equation}
\label{Heaviside function}
\theta(z):=
\begin{cases}
0&\text{if}\quad z\le0,
\\
1&\text{if}\quad z>0
\end{cases}
\end{equation}
be the Heaviside function.
The operators $P_\pm$ are defined as the extensions to $\Omega^1(M)$ of the operators $\theta(\pm\curl)$ on $\delta\Omega^2(M)$. Here the extensions are specified by the requirement that the orthogonal complement of $\delta\Omega^2(M)$ in $\Omega^1(M)$ maps to zero.
\end{definition}

The operator $P_0$ is the orthogonal projection onto the kernel of $\curl$ with harmonic 1-forms removed, whereas the operators $P_\pm$ are the positive (+) and negative (-) spectral projections of $\curl$.

Further on we use the notation
\begin{equation}
\label{norm of xi}
\|\xi\|
:=\sqrt{g^{\mu\nu}(x)\,\xi_\mu\xi_\nu}
\end{equation}
and we denote by $|\,\cdot\,|$ the Euclidean norm of vectors. Throughout this paper we use Greek letters for tensor indices.

\begin{theorem}
\label{main theorem 1}

\phantom{?}
\begin{enumerate}[(a)]
\item
The operators $P_0$, $P_+$ and $P_-$ are pseudodifferential of order zero.
\item
Their principal symbols read
\begin{equation}
\label{main theorem 1 equation 1}
[(P_0)_\mathrm{prin}]_\alpha{}^\beta(x,\xi)
=
\|\xi\|^{-2}\,\xi_\alpha \,g^{\beta\gamma}(x)\,\xi_\gamma\,,
\end{equation}
\begin{equation}
\label{main theorem 1 equation 2}
[(P_\pm)_\mathrm{prin}]_\alpha{}^\beta(x,\xi)
=
\frac12
\left[
\delta_\alpha{}^\beta
-
[(P_0)_\mathrm{prin}]_\alpha{}^\beta(x,\xi)
\pm i\,
\|\xi\|^{-1}\,
E_\alpha{}^{\gamma\beta}(x)\,
\xi_\gamma
\right],
\end{equation}
where
\begin{equation}
\label{main theorem 1 equation 3}
E_{\alpha\beta\gamma}(x):=\rho(x)\,\varepsilon_{\alpha\beta\gamma}\,,
\end{equation}
$\rho$ is the Riemannian density and $\varepsilon$ is the totally antisymmetric symbol,
$\varepsilon_{123}:=+1$.
\item
Their subprincipal symbols, defined in accordance with~\eqref{subprincipal symbol operators 1-forms},  are zero:
\begin{equation}
\label{main theorem 1 equation 4}
(P_0)_\mathrm{sub}=0, \qquad (P_\pm)_\mathrm{sub}=0.
\end{equation}
\end{enumerate}
\end{theorem}

Euclidean versions of formulae
\eqref{main theorem 1 equation 1} and \eqref{main theorem 1 equation 2}
appeared in
\cite{lerner}.
On the other hand, the notion of subprincipal symbol for operators acting on 1-forms was
never previously defined in full generality.
We address this matter in detail, including a discussion of known results,
in Section~\ref{Pseudodifferential operators acting on 1-forms}.

\

Remarkably, the pseudodifferential projections 
$P_\pm$ can be constructed explicitly. In Proposition~\ref{proposition algorighm Pj} we will provide an algorithm for the calculation of their full symbols. The latter proposition and Theorem~\ref{main theorem 1} give a detailed description of the structure of the projection operators $P_0$ and $P_\pm$, which is our first main result.

\

The projections $P_+$ and $P_-$ are the key ingredients of our new approach to spectral asymmetry. One starts by observing that, at a formal level, we have
\begin{equation}
\label{at a formal level}
\#\{\text{positive eigenvalues}\}
\,-\,
\#\{\text{negative eigenvalues}\}
=\operatorname{Tr}(P_+-P_-).
\end{equation}
Here and further on $\operatorname{Tr}$ stands for operator trace, whereas $\operatorname{tr}$ stands for matrix trace. Unfortunately, one is immediately presented with two major issues: (i) the LHS of \eqref{at a formal level} is a difference of two infinities and (ii) the (matrix) operator $P_+-P_-$ in the RHS of \eqref{at a formal level} is not of trace class. Our strategy is to assign a meaning to the RHS of \eqref{at a formal level} by decomposing the procedure of taking operator trace into two steps:
\begin{enumerate}[1)]
\item
take the pointwise matrix trace of the matrix pseudodifferential operator $P_+-P_-\,$,
\item
compute the operator trace of the resulting scalar pseudodifferential operator.
\end{enumerate}
Step 1 above defines a scalar self-adjoint pseudodifferential operator which we call the \emph{asymmetry operator} and denote by $A$ --- see Definition~\ref{definition asymmetry operator}. This operator is defined uniquely, up to the addition of an integral operator with infinitely smooth kernel vanishing in a neighbourhood of the diagonal $M\times M$, and it is determined by the Riemannian 3-manifold $(M,g)$ and its orientation.

The operator $A$ is \emph{prima facie} a pseudodifferential operator of order 0. However, taking the matrix trace brings about unexpected cancellations, so that the operator $A$ turns out, in fact, to be of order $-3$. This is our second main result.

\begin{theorem}
\label{main theorem 2}
The asymmetry operator $A$ is a pseudodifferential operator of order $-3$ and its principal symbol reads
\begin{equation}
\label{main theorem 2 equation 1}
A_\mathrm{prin}(x,\xi)=-\frac1{2\|\xi\|^5}E^{\alpha\beta \gamma}(x)\, \nabla_\alpha \operatorname{Ric}_{\beta}{}^\rho(x)\, \xi_\gamma\xi_\rho\,,
\end{equation}
where $\Ric$ is the Ricci tensor\footnote{The Riemann curvature tensor $\Riem$ has components ${\Riem^\kappa}_{\lambda\mu\nu}$ defined in accordance with 
\[
{\Riem^\kappa}_{\lambda\mu\nu}:=
\dr x^\kappa(\Riem(\partial_\mu\,,\partial_\nu)\,\partial_\lambda)
=
\partial_\mu{\Gamma^\kappa}_{\nu\lambda}
-\partial_\nu{\Gamma^\kappa}_{\mu\lambda}
+{\Gamma^\kappa}_{\mu\eta}{\Gamma^\eta}_{\nu\lambda}
-{\Gamma^\kappa}_{\nu\eta}{\Gamma^\eta}_{\mu\lambda}\,,
\]
the $\Gamma$'s being Christoffel symbols. The Ricci tensor is defined as $\Ric_{\mu\nu}:=R^\alpha{}_{\mu\alpha\nu}$ and $\Sc:=g^{\mu\nu}\Ric_{\mu\nu}$ is scalar curvature.}.
\end{theorem}

\begin{remark}
\label{principal symbol does not feel scalar curvature}
Note that formula \eqref{main theorem 2 equation 1} can be equivalently rewritten as
\begin{equation*}
\label{main theorem 2 equation 1 trace-free}
A_\mathrm{prin}(x,\xi)=-\frac1{2\|\xi\|^5}E^{\alpha\beta \gamma}(x)\, \nabla_\alpha 
\left(
\operatorname{Ric}_{\beta}{}^\rho(x)
-
\frac{1}{3}\,
\delta_{\beta}{}^\rho\,\Sc(x)
\right)
\xi_\gamma\xi_\rho\,,
\end{equation*}
i.e.~one can replace the Ricci tensor with its trace-free part. In other words, the principal symbol of the asymmetry operator does not feel scalar curvature.
\end{remark}

A sufficient condition for a self-adjoint pseudodifferential operator to be of trace class and to have continuous integral kernel is for its order to be strictly less than $-d$, where $d$ is the dimension of the manifold. See also \cite[\S 12.1]{shubin}.
In the case at hand, the operator $A$ has order $-3$ and the dimension of the manifold is 3, so we are looking at a borderline situation. One would hope that with a mild regularisation one could cross the finish line and be able to define a notion of regularised trace.

Indeed, we will establish in Theorem~\ref{singularity theorem} that the integral kernel $\mathfrak{a}(x,y)$ of $A$ decomposes as $\mathfrak{a}(x,y)=\mathfrak{a}_d(x,y)+\mathfrak{a}_c(x,y)$, where $\mathfrak{a}_c$ is continuous, whereas $\mathfrak{a}_d$ is possibly discontinuous but bounded, and integrates to zero over small spheres centred at either $x$ or $y$.

This leads us to our third main result.

\begin{theorem}
\label{main theorem 3}
The integral kernel $\mathfrak{a}(x,y)$ of $A$ is a bounded function, smooth outside the diagonal. Furthermore, for any $x\in M$ the limit
\[
\psi_{\operatorname{curl}}^\loc(x)
:=
\lim_{r\to0^+}
\frac{1}{4\pi r^2}
\int_{\mathbb{S}_r(x)}
\mathfrak{a}(x,y)\,\dr S_y
\]
exists and defines a continuous scalar function $\,\psi_{\operatorname{curl}}^\loc:M\to\mathbb{R}\,$. Here $\mathbb{S}_r(x)=\{y\in M|\dist(x,y)=r\}$ is the sphere of radius $r$ centred at $x$
and $\dr S_y$ is the surface area element on this sphere.
\end{theorem}

We call the function $\psi_{\operatorname{curl}}^\loc(x)$ and the number $\,\psi_{\operatorname{curl}}:=\int_M\psi_{\operatorname{curl}}^\loc(x)
\,\rho(x)\,\dr x\,$ the local and global regularised trace of $A$, see Definitions~\ref{local regularised trace of the asymmetry operator} and~\ref{global regularised trace of the asymmetry operator}. These two quantities are geometric invariants, in that they are determined by the Riemannian 3-manifold and its orientation. In particular, the global regularised trace is a real number measuring the asymmetry of the spectrum of $\curl$.

The natural question is how are our geometric invariants $\psi_{\operatorname{curl}}(x)$ and $\psi_{\operatorname{curl}}$ related to classical eta invariants.
We comprehensively address this matter in our companion paper \cite{conjectures} where we prove that our $\psi_{\operatorname{curl}}(x)$ and $\psi_{\operatorname{curl}}$
are precisely the classical local and global eta invariants for $\curl$.

\begin{remark}
Pseudodifferential projections are very natural tools in the study of the spectral properties of systems. At first glance, one may be tempted to think they are unnecessarily complicated, especially for operators acting on trivial bundles, in that one could try and diagonalise the operator along the lines of \cite{diagonalisation} instead, thus reducing the problem at hand to the examination a collection of scalar operators. However, after examining the matter more carefully, diagonalisation turns out to be topologically obstructed in numerous examples of physical relevance. In particular, it was shown in \cite[Section~3.3]{obstructions} that for the case of the operator $\curl$ one cannot even diagonalise the principal symbol $\curl_\prin$ \eqref{principal symbol curl} in the whole cotangent fibre at a given point of $M$, let alone the operator itself.
\end{remark}

\

All in all, our results offer a new approach to the subject of spectral asymmetry, one that possesses the following elements of novelty.
\begin{enumerate}
\item
We characterise asymmetry in terms of a pseudodifferential operator of negative order, as opposed to a single number. Here the fact that the order is negative is key as it opens the way to defining the regularised trace.
\item
Our approach is not operator-specific. It is versatile and can be deployed in a variety of different scenarios. For example, our strategy can be applied to the Dirac operator\footnote{We also refer the reader to \cite{branson} for related results.}, as we will show in a separate paper.
\item
Our technique is explicit and based on direct computations. Moreover, it rigorously implements the intuitive notion of spectral asymmetry as difference between numbers of positive and negative eigenvalues, without relying on ``black boxes'' such as analytic continuation from the very beginning, or using heat-kernel-type arguments (see, e.g., \cite[Theorem~2.6 and Remark~3]{bismut}).
\end{enumerate}

\color{black}
There are, of course, many different approaches to the subject of spectral asymmetry. A notable one, for instance, relies on a local invariant known as \emph{Wodzicki residue} \cite{wodzicki}, see also~\cite{loya, bruning} for a broader discussion. Since the literature on spectral asymmetry is vast, we decided to refrain from including a detailed literature review here, as it would inevitably be incomplete and would deserve a separate paper (if not a monograph).
\color{black}

\begin{remark}
Throughout this paper, we focus our analysis on the space of real-valued 1-forms $\Omega^1(M)$. One could perform a completely analogous analysis working in the space of real-valued 2-forms $\Omega^2(M)$; indeed, on a 3-manifold $\Omega^1(M)$ and $\Omega^2(M)$ are isomorphic by Hodge duality.
\end{remark}

\subsection*{Structure of the paper}
\addcontentsline{toc}{subsection}{Structure of the paper}

Our paper is structured as follows.

\color{black}
In Section~\ref{The operator curl} we put $\curl$ on a rigorous operator-theoretic footing, and state its basic properties. Although most of the results from this section are known at least in some form, detailed proofs, often hard to find in the literature, are provided in Appendix~\ref{Basic properties of the operator curl} in full and in a self-contained fashion, for the convenience of the reader and future reference. In doing so, we shall rely on the auxiliary operator $\curl_E$, called \emph{extended curl}, which can be viewed as an elliptic ``extension'' of $\curl$.
\color{black}

In Section~\ref{Pseudodifferential operators acting on 1-forms} we temporarily move away from dimension 3 and develop an invariant calculus for pseudodifferential operators acting on 1-forms over a Riemannian closed $d$-manifold. In particular we define a notion of subprincipal symbol for these operators, and write down the composition formula.

In Section~\ref{Trace of pseudodifferential operators acting on 1-forms} we prepare the ground for the formulation of our main results and examine the issue of taking the trace of pseudodifferential operators on 1-forms. The highlight of this section is the notion of matrix trace $\operatorname{\mathfrak{tr}}\,$, defined by \eqref{Q acting on 1-forms 6} and \eqref{Q acting on 1-forms 7}, which addresses the issue of taking the pointwise matrix trace of integral operators whose integral kernel is a two-point tensor.

In Section~\ref{The operators P0 and Ppm} we study the pseudodifferential projections $P_0$, $P_+$ and $P_-$, and provide an explicit algorithm for the construction of their full symbols. This section contains the proof of Theorem~\ref{main theorem 1}.

In Section~\ref{The asymmetry operator} we define the asymmetry operator $A$ (subsection~\ref{Definition of the asymmetry operator subsection}), prove that it is a pseudodifferential operator of order $-3$ (subsection~\ref{The order of the asymmetry operator subsection}), and compute its principal symbol $A_\mathrm{prin}$ (subsection~\ref{The principal symbol of the asymmetry operator}). This section establishes Theorem~\ref{main theorem 2}.

Section~\ref{The regularised trace of $A$} is the centrepiece of our paper. After examining the singular structure of the integral kernel $\mathfrak{a}(x,y)$ of $A$ near the diagonal (subsection~\ref{Structure of Aprin near the diagonal}), we devise a regularisation procedure which allows one to define the local and global trace of the asymmetry operator, geometric invariants measuring spectral asymmetry (subsection~\ref{Local and global trace}). This yields Theorem~\ref{main theorem 3}.

In Section~\ref{Challenges in higher dimensions} we briefly elaborate on the challenges involved in extending our results to dimensions higher than 3, and why doing so could be interesting.

The paper is complemented by a list of notation given at the end of this section, and by \textcolor{black}{seven} appendices.
In Appendix~\ref{Exterior calculus} we specify our notation and conventions on exterior calculus. \textcolor{black}{In Appendix~\ref{Basic properties of the operator curl} we provide proofs for the basic properties of the operator curl, summarised in Section~\ref{The operator curl}.}
Appendices~\ref{The spectrum of the Laplacian on a Berger sphere} and~\ref{The spectrum of curl on a Berger sphere} are concerned with spectral theory on Berger 3-spheres: in Appendix~\ref{The spectrum of the Laplacian on a Berger sphere} we recall the definition of a Berger sphere and list the eigenvalues of the Laplace--Beltrami operator, whereas in Appendix~\ref{The spectrum of curl on a Berger sphere} we list the eigenvalues of $\curl$ on a Berger sphere as an illustration of spectral asymmetry, and compute the eta invariant by explicitly examining the eta function for curl.
In Appendix~\ref{Maxwell's equations and electromagnetic chirality} we examine the relation between the spectral problem for curl and solutions of Maxwell's equations, and discuss the physical meaning of the sign of the eigenvalues of $\curl$.
In Appendix~\ref{appendix parallel transport} we derive series expansions for the parallel transport maps. Finally, in Appendix~\ref{An alternative derivation of formula} we independently verify formula~\eqref{main theorem 2 equation 1}.


\subsection*{Notation}
\addcontentsline{toc}{subsection}{Notation}

\begin{longtable}{l l}
\hline
\\ [-1em]
\multicolumn{1}{c}{\textbf{Symbol}} & 
  \multicolumn{1}{c}{\textbf{Description}} \\ \\ [-1em]
 \hline \hline \\ [-1em]
$\sim$ & Asymptotic expansion \\ \\ [-1em]
$\ast$ & Hodge dual \eqref{definition of Hodge star} \\ \\ [-1em]
$\|\,\cdot\,\|$ & Riemannian norm \eqref{norm of xi} \\ \\ [-1em]
$|\,\cdot\,|$ & Euclidean norm \\ \\ [-1em]
$\langle\,\cdot\,\rangle$ & Japanese bracket \eqref{Japanese bracket}\\ \\ [-1em]
$A$ & Asymmetry operator, Definition~\ref{definition asymmetry operator} \\ \\ [-1em]
$A_\mathrm{diag}$, $A_\mathrm{pt}$ & Local decomposition of $A$ as per~\eqref{A diag} and~\eqref{A pt} \\ \\ [-1em]
$\mathfrak{a}(x,y)$ & Integral kernel of the asymmetry operator $A$ \\ \\ [-1em]
$\curl$ & Curl, as a differential expression \eqref{curl differential experssion} and as an operator \eqref{definition curl} \\ \\ [-1em]
$\curl_E$ & Extended curl, Definition~\ref{definition extended curl}  \\ \\ [-1em]
$\curl_{E,\dr},\curl_{E,\delta},\curl_{E,\mathcal{H}}$ & Orthogonal summands of $\curl_E$, Lemma~\ref{lemma invariant subspaces extended curl} \\ \\ [-1em]
$d$ & Dimension of the manifold $M$, $d\ge2$\\ \\ [-1em]
$\dr$ & Exterior derivative, Appendix~\ref{Exterior calculus} \\ \\ [-1em]
$\delta$ & Codifferential, Appendix~\ref{Exterior calculus} \\ \\ [-1em]
$\Delta:=-\delta\dr$ & (Nonpositive) Laplace--Beltrami operator \\ \\ [-1em]
$\boldsymbol{\Delta}:=-(\dr\delta+\delta\dr)$ & (Nonpositive) Hodge Laplacian\\ \\ [-1em]
$\dist$ & Geodesic distance \\ \\ [-1em]
$e_j{}^\alpha(x)$, $e^k{}_\beta(x)$ & Framing and dual framing \eqref{dual framing} \\ \\ [-1em]
$\tilde e_j{}^\alpha(x)$, $\tilde e^k{}_\beta(x)$ & Levi-Civita framing and dual Levi-Civita framing, Definition~\ref{definition Levi-Civita framing} \\ \\ [-1em]
$\varepsilon_{\alpha\beta\gamma}$ & Totally antisymmetric symbol, $\varepsilon_{123}=+1$ \\ \\ [-1em]
$E_{\alpha\beta\gamma}$ & Totally antisymmetric tensor \eqref{main theorem 1 equation 3} \\ \\ [-1em]
$f_{x^\alpha}$ & Partial derivative of $f$ with respect to $x^\alpha$ \\ \\ [-1em]
$g$ & Riemannian metric \\ \\ [-1em]
$G$ & Einstein tensor \\ \\ [-1em]
$\gamma(x,y;\tau)$ & Geodesic connecting $x$ to $y$, with $\gamma(x,y;0)=x$ and $\gamma(x,y;1)=y$ \\ \\ [-1em]
$\Gamma^\alpha{}_{\beta\gamma}$ & Christoffel symbols\\ \\ [-1em]
$\eta_Q(s)$ & Eta function of the operator $Q$ \\ \\ [-1em]
$H^s(M)$ & Generalisation of the usual Sobolev spaces $H^s$ to differential forms \\ \\ [-1em]
$\mathcal{H}^k(M)$ & Harmonic $k$-forms over $M$ \\ \\ [-1em]
$(\lambda_j, u_j)$, $j=\pm1, \pm2, \ldots$ & Eigensystem for $\curl$ \\ \\ [-1em]
$\theta$ & Heaviside theta function \eqref{Heaviside function} \\ \\ [-1em]
$\operatorname{I}$ & Identity matrix \\ \\ [-1em]
$\operatorname{Id}$ & Identity operator \\ \\ [-1em]
$(\mu_j, f_j)$, $j=0,1,2,\ldots$ & Eigensystem for $-\Delta$ \\ \\
[-1em]
$M$ & Connected oriented closed manifold\\ \\ [-1em]
$\operatorname{mod} \ \Psi^{-\infty}$ & Modulo an integral operator with infinitely smooth kernel \\ \\ [-1em]
$P_0$, $P_\pm$ & Definitions~\ref{definition of the operator P0} and~\ref{definition of the operators Ppm} \\ \\ [-1em]
$p_\pm(x,y)$ & Full symbol of $P_\pm$ \\ \\ [-1em]
$Q_\prin$ & Principal symbol of the pseudodifferential operator $Q$ \\ \\ [-1em]
$Q_{\prin,s}$ & Principal symbol of $Q$, a pseudodifferential operator of order $-s$ \\ \\ [-1em]
$Q_{\sub}$ & Subprincipal symbol of $Q$, for operators on 1-forms see Definition~\ref{definition subprincipal symbol operators 1 forms}  \\ \\ [-1em]
$\Riem$, $\Ric$, $\Sc$ & Riemann curvature tensor, Ricci tensor and scalar curvature \\ \\ [-1em]
$\rho(x)$ & Riemannian density \\ \\ [-1em]
$\mathbb{S}_r(x)$ & Geodesic sphere of radius $r$ centred at $x\in M$\\ \\ [-1em]
$\operatorname{tr}$ & Matrix trace \eqref{Q acting on column of scalars 3} \\ \\ [-1em]
$\operatorname{\mathfrak{tr}}$ & Matrix trace defined in accordance with~Definition~\ref{definition of matrix trace of a pdo acting on 1-forms} \\ \\ [-1em]
$\operatorname{Tr}$ & Operator trace \eqref{Q acting on column of scalars 2}  \\ \\ [-1em]
$TM$, $T^*M$ & Tangent and cotangent bundle \\ \\ [-1em]
$\Omega^k(M)$ & Differential $k$-forms over $M$ \\ \\ [-1em]
$\psi_{\curl}^\loc(x)$ & Regularised local trace of $A$, Definition~\ref{local regularised trace of the asymmetry operator} \\ \\ [-1em]
$\psi_{\curl}$ & Regularised global trace of $A$, Definition~\ref{global regularised trace of the asymmetry operator} \\ \\ [-1em]
$\Psi^s$ & Classical pseudodifferential operators of order $s$ \\ \\ [-1em]
$\Psi^{-\infty}$ & Infinitely smoothing operators $\eqref{Psi minus infinity}$ \\ \\ [-1em]
$\zeta_Q(s)$ & Zeta function of the operator $Q$ \\ \\ [-1em]
$Z$ & Parallel transport map \eqref{Q acting on 1-forms 3}, see also Appendix~\ref{appendix parallel transport} \\ \\ [-1em]
\hline
\end{longtable}

\section{The operator curl}
\label{The operator curl}

\color{black}

In this section we put $\operatorname{curl}$ on a rigorous mathematical foundation,  defining it within the framework of the theory of unbounded self-adjoint operators in Hilbert spaces. This would provide the reader with an account of the fundamental properties of $\operatorname{curl}$,  in a self-contained fashion, postponing proofs until Appendix~\ref{Basic properties of the operator curl}.

The main challenge that one faces when dealing with $\operatorname{curl}$ is that it is not elliptic\footnote{Recall that, by definition, a matrix (pseudo)differential operator is elliptic if the determinant of its principal symbol is nowhere vanishing on $T^*M\setminus\{0\}$.}. Indeed,  the formula for the principal symbol of $\operatorname{curl}$ (which happens to coincide with its full symbol) reads
\begin{equation}
\label{principal symbol curl}
[\operatorname{curl}_\mathrm{prin}]_\alpha{}^\beta(x,\xi)= -i\,E_\alpha{}^{\beta\gamma}(x)\,\xi_\gamma\,,
\end{equation}
where the tensor $E$ is defined in accordance with \eqref{main theorem 1 equation 3}. Although $\xi$ is not a 1-form in its own right, at a formal level the RHS of \eqref{principal symbol curl} can be rewritten, concisely, as $-i*\xi$. A straightforward calculation shows that
\begin{equation}
\label{determinant principal symbol curl}
\det (\operatorname{curl}_\mathrm{prin})=0.
\end{equation}
The eigenvalues of $\operatorname{curl}_\mathrm{prin}$ are simple and read
\begin{equation}
\label{eigenvalues principal symbol curl}
h^{(0)}(x,\xi)=0, \qquad h^{(\pm)}(x,\xi)=\pm\|\xi\|\,,
\end{equation}
for all $(x,\xi)\in T^*M\setminus \{0\}$. Consequently,  standard elliptic theory does not apply and care is required in defining the appropriate Hilbert space,  domain of the operator and establishing the mapping properties of $\operatorname{curl}$.

Let $H^s(M)$, $s>0$, be the space of differential forms that are square integrable together with their partial derivatives up to order $s$. We do not carry in our notation for Sobolev spaces the degree of differential forms: this will be clear from the context. Henceforth, to further simplify notation we drop the $M$ and write $\Omega^k$ for $\Omega^k(M)$ and $H^s$ for $H^s(M)$.

\

To begin with, let us observe that $\operatorname{curl}$ maps smooth coexact 1-forms to smooth coexact 1-forms. Therefore,  it is natural to define $\operatorname{curl}$ as an operator
\begin{equation}
\label{definition curl}
\operatorname{curl}=*\dr : \delta\Omega^2 \cap H^1\to \delta\Omega^2,
\end{equation}
where 
$\delta\Omega^2$ is the Hilbert space of real-valued coexact 1-forms with inner product
\begin{equation}
\label{inner product}
\langle u,v \rangle:=\int_M *u \wedge v=\int_M u \wedge * v.
\end{equation}

The basic properties of the operator \eqref{definition curl} are summarised by the following theorem.

\begin{theorem}
\label{theorem properties of curl}
Let $(M,g)$ be a  connected oriented closed Riemannian 3-manifold.
\begin{enumerate}[(a)]

\item The operator $\operatorname{curl}$ \eqref{definition curl} is self-adjoint.

\item The spectrum of $\operatorname{curl}$ is discrete and accumulates to $+\infty$ and $-\infty$.

\item Zero is not an eigenvalue of $\operatorname{curl}$.

\item The operator $\operatorname{curl}^{-1}$ is a bounded operator from $\delta\Omega^2 \cap H^s$ to $\delta\Omega^2 \cap H^{s+1}$ for all $s\ge0$.

\end{enumerate}
\end{theorem}

\begin{remark}
Let us point out, once again, that throughout this paper $\curl$ comes in two guises: as a differential expression \eqref{curl differential experssion} and as a self-adjoint operator \eqref{definition curl}. Both will be denoted by ``$\curl$''; which is which will be clear from the context.
\end{remark}

We would like to emphasise the significance of property (c) in the above theorem. If we start deforming the metric the eigenvalues of $\curl$, understood as those of a self-adjoint operator \eqref{definition curl}, will never turn to zero or cross zero (change sign). This is in stark contrast with the Dirac operator where zero may be an eigenvalue, depending on the choice of metric \cite{hitchin}. In fact, the study of zero modes of the Dirac operator (harmonic spinors) is an active area of research.

\color{black}

\section{Pseudodifferential operators acting on 1-forms}
\label{Pseudodifferential operators acting on 1-forms}

In this section we develop an invariant calculus for pseudodifferential operators acting on 1-forms over a closed Riemannian manifold $(M,g)$ of arbitrary dimension $d$.

Let $Q$ be a classical polyhomogeneous pseudodifferential operator of order $s$ acting on 1-forms,  and let
\begin{equation}
\label{11 June 2021 equation 1}
Q: u_\alpha(x)
\mapsto
v_\alpha(x)
=(2\pi)^{-d}
\int
e^{i(x-y)^\gamma\xi_\gamma}\,q_\alpha{}^\beta(x,\xi)\,
u_\beta(y)\,\dr y\,\dr\xi\,,
\end{equation}
\begin{equation}
\label{11 June 2021 equation 2}
q_\alpha{}^\beta(x,\xi)
\sim
[q_s]_\alpha{}^\beta(x,\xi)
+
[q_{s-1}]_\alpha{}^\beta(x,\xi)
+\dots
\end{equation}
be its representation in local coordinates (the same for $x$ and $y$). Here $\sim$ stands for asymptotic expansion \cite[\S~3.3]{shubin}, and `polyhomogeneous' means that
$[q_{s-k}]_\alpha{}^\beta(x, \lambda \,\xi)=\lambda^{s-k}\,[q_{s-k}]_\alpha{}^\beta(x, \xi)$ for any positive $\lambda$ and $k=0,1,2,\dots$.

Throughout the paper we denote pseudodifferential operators with upper case letters, and their symbols and homogeneous components of symbols --- with lower case. For the principal and subprincipal symbol we use upper case letters and denote them by $Q_\mathrm{prin}$ and $Q_\mathrm{sub}$ respectively, with the subscript indicating that we are looking at invariant (or covariant) objects `living' on the cotangent bundle as opposed to the operator $Q$ itself which `lives' (acts) on the base manifold.

Later on we will sometimes write a pseudodifferential operator as an integral operator with distributional integral kernel (Schwartz kernel), see, for example, formula \eqref{Q acting on 1-forms 1}. We will use lower case Fraktur font for the Schwartz kernel.

Returning to the pseudodifferential operator \eqref{11 June 2021 equation 1}, \eqref{11 June 2021 equation 2}, it is easy to see that
\begin{equation}
\label{formal definition of principal symbol}
[Q_\prin]_\alpha{}^\beta(x,\xi):=[q_s]_\alpha{}^\beta(x,\xi)
\end{equation}
provides an invariant definition for the principal symbol $Q_\prin$ of $Q$.

However, the subleading term, $[q_{s-1}]_\alpha{}^\beta\,$, is \emph{not} invariant under change of coordinates. The task at hand is to define and provide a formula for the subprincipal symbol of $Q$, by determining appropriate correction terms to $[q_{s-1}]_\alpha{}^\beta$.  

The original definition of subprincipal symbol goes back to J.J.~Duistermaat and L.~H\"ormander \cite[Eqn. (5.2.8)]{DuHo}. They defined the subprincipal symbol for operators acting on half-densities.

Duistermaat and H\"ormander's definition of subprincipal symbol extends to operators acting on scalar fields because scalars can be identified with half-densities: it is just a matter of multiplying or dividing by $\sqrt\rho\,$. This leads to the appearance of an additional correction term. Namely, given a pseudodifferential operator $Q$ of order $s$
\begin{equation*}
\label{11 June 2021 equation 1 scalar}
Q: f(x)
\mapsto
(2\pi)^{-d}
\int
e^{i(x-y)^\gamma\xi_\gamma}\,q(x,\xi)\,
f(y)\,\dr y\,\dr \xi\,,
\end{equation*}
\begin{equation*}
\label{11 June 2021 equation 2 scalar}
q(x,\xi)
\sim
q_s(x,\xi)
+
q_{s-1}(x,\xi)
+\dots
\end{equation*}
acting on scalar fields, its subprincipal symbol reads
\begin{equation}
\label{subprincipal symbol of operator acting on a scalar field}
Q_\mathrm{sub}
:=
q_{s-1}
+
\frac{i}2\dfrac{\partial^2 q_s}{\partial x^\gamma\partial\xi_\gamma}
+
\frac{i}2\,
\frac{\partial\ln\rho}{\partial x^\gamma}\,
\frac{\partial q_s}{\partial\xi_\gamma}\,,
\end{equation}
where $\rho$ is the Riemannian density. In formula \eqref{subprincipal symbol of operator acting on a scalar field} one can, in principle, instead of the Riemannian density use any other prescribed positive reference density. However, in this paper we always stick with the Riemannian density. In particular, this implies the identity
\begin{equation}
\label{relation between Riemannian density and Christoffel symbols}
\frac{\partial\ln\rho}{\partial x^\gamma}
=
\Gamma^{\alpha}{}_{\gamma\alpha}\,.
\end{equation}

The definition of subprincipal symbol \eqref{subprincipal symbol of operator acting on a scalar field} extends further, without change, to operators acting on $m$-columns of scalar fields (sections of the trivial $\mathbb{R}^m$- or $\mathbb{C}^m$-bundle over $M$). \textcolor{black}{By choosing a global framing (an orientable $3$-manifold is automatically parallelizable), one can then produce a coordinate-invariant reformulation of the classical construction of subprincipal symbol for operators on 1-forms. However, we should like to emphasise that the approach presented in this manuscript avoids the reliance on a specific global framing. For this reason, we believe it is still of interest and worth presenting.}

\textcolor{black}{Note that a} notion of subprincipal symbol for operators acting on 1-forms is available in the literature in the special case when the principal symbol is of the form $fI$, where $f$ is a scalar function on $T^*M\setminus\{0\}$ and $I$ is the $d\times d$ identity matrix, see, e.g., \cite{hintz,JS}.

\begin{theorem}
\label{theorem subprincipal symbol operators 1 forms}
Let $Q$ be a pseudodifferential operator of order $s$ acting on 1-forms with symbol \eqref{11 June 2021 equation 2}. Then the quantity
\begin{equation}
\label{subprincipal symbol operators 1-forms}
[Q_\mathrm{sub}]_\mu{}^\nu:=
\left([q_{s-1}]_\mu{}^\nu+\frac{i}2\dfrac{\partial^2 [q_s]_\mu{}^\nu}{\partial x^\gamma\partial\xi_\gamma} \right)
+\frac{i}2
\left(
\Gamma^{\alpha}{}_{\gamma\alpha} \dfrac{\partial [q_s]_\mu{}^\nu}{\partial\xi_\gamma}\
-
\Gamma^\alpha{}_{\gamma \mu}
\dfrac{\partial [q_s]_\alpha{}^\nu}{\partial\xi_\gamma}
-
\Gamma^\nu{}_{\gamma\alpha}
\dfrac{\partial [q_s]_\mu{}^\alpha}{\partial\xi_\gamma} 
\right)
\end{equation}
is covariant under change of local coordinates.
\end{theorem}

\begin{definition}
\label{definition subprincipal symbol operators 1 forms}
We call \eqref{subprincipal symbol operators 1-forms} the \emph{subprincipal symbol} of $Q$.
\end{definition}

\begin{remark}
\label{remark about meaning of covariance}
`Covariance' in Theorem \ref{theorem subprincipal symbol operators 1 forms} means that one doesn't get derivatives of Jacobians under change of local coordinates. More precisely, let $x$ and $\tilde x$ be two local coordinate systems, and let $Q_\mathrm{sub}$ and $\tilde Q_\mathrm{sub}$ be the quantity \eqref{subprincipal symbol operators 1-forms} evaluated in coordinates $x$ and $\tilde x$ respectively. Put $J^\alpha{}_\beta:=\partial\tilde x^\alpha/\partial x^\beta$, $[J^{-1}]^\mu{}_\nu:=\partial x^\mu/\partial\tilde x^\nu$,
$\tilde\xi_\alpha:=[J^{-1}]^\beta{}_\alpha\,\xi_\beta\,$.
Then the statement of the theorem is that
\begin{equation}
\label{covariance explained in detail}
[\tilde Q_\mathrm{sub}]_\alpha{}^\beta(\tilde x,\tilde\xi)=
[J^{-1}]^\mu{}_\alpha(\tilde x)
\ [Q_\mathrm{sub}]_\mu{}^\nu(x(\tilde x),\xi(\tilde x,\tilde\xi))
\ J^\beta{}_\nu(x(\tilde x))\ .
\end{equation}
The above transformation law also applies to the principal symbol \eqref{formal definition of principal symbol}.
\end{remark}
\begin{remark}
\label{remark minus sign}
Note that the RHS of formula \eqref{subprincipal symbol operators 1-forms} does not reduce to
\[
[q_{s-1}]_\mu{}^\nu+\frac{i}2
\nabla_\gamma
\dfrac{\partial[q_s]_\mu{}^\nu}{\partial\xi_\gamma}
\]
by formally treating $\frac{\partial[q_s]_\mu{}^\nu}{\partial\xi_\gamma}$ as a $(2,1)$-tensor on $M$. The issue here is that $\frac{\partial[q_s]_\mu{}^\nu}{\partial\xi_\gamma}$ is not a true tensor, and even if it were, the sign in front of the last Christoffel symbol is not what one would expect.
\end{remark}

We now proceed to the proof of Theorem~\ref{theorem subprincipal symbol operators 1 forms}. One way of proving the theorem is by a lengthy explicit calculation establishing \eqref{covariance explained in detail}.
However, we will present an alternative proof which provides insight into the origins of formula \eqref{subprincipal symbol operators 1-forms}. In order to do so we need to introduce first some geometric notions that will be useful here and later on in the paper.

By a \emph{frame} at a point $x\in M$ we mean a basis of orthonormal vectors in $T_xM$. By a \emph{framing} in an open set $\mathcal{U}\subset M$ we mean the choice of a frame at each point $x$ of $\mathcal{U}$,  depending smoothly on the base point $x$. 

Given a framing $\{e_j{}^\alpha\}_{j=1}^d$,  the corresponding \emph{dual framing} $\{e^k{}_\beta\}_{k=1}^d$ is defined as
\begin{equation}
\label{dual framing}
e^k{}_\beta(x):=\delta^{jk}\,g_{\alpha\beta}(x)\,e_j{}^\alpha(x),
\end{equation}
$\delta$ being the Kronecker symbol.
Here and further on,  $e_j{}^\alpha$ denotes the $\alpha$-th component of the $j$-th vector field and $e^k{}_\beta$ denotes the $\beta$-th component of the $k$-th covector field.

\begin{definition}
\label{definition Levi-Civita framing}
Let $z\in M$ be given and let $\{V_j\}_{j=1}^d$ be a frame at $z$.  Let $\mathcal{U}_z$ be a sufficiently small neighbourhood of $z$. We define the \emph{Levi-Civita framing} $\{\tilde e_j\}_{j=1}^d$ to be the framing in $\mathcal{U}_z$ obtained by parallel-transporting $\{V_j\}_{j=1}^d$ along geodesics emanating from $z$.
\end{definition}

Let us emphasise that the Levi-Civita framing is uniquely determined in the neighbourhood of every point $z\in M$ up to a rigid (constant) orthogonal transformation, which reflects the freedom in choosing the frame $\{V_j\}_{j=1}^d$ at the point $z$. By definition, the Levi-Civita framing and its dual satisfy
\begin{equation}
\label{14 June 2021 equation 1}
\left.
\left[
\partial_\beta \tilde e_j{}^\alpha+\Gamma^\alpha{}_{\beta\gamma}\,\tilde e_j{}^\gamma
\right]
\right|_{x=z}
=0\,,
\end{equation}
\begin{equation}
\label{14 June 2021 equation 2}
\left.
\left[
\partial_\beta \tilde e\,{}^j{}_\alpha-\Gamma^\gamma{}_{\beta\alpha}\,\tilde e\,{}^j{}_\gamma
\right]
\right|_{x=z}
=0,
\end{equation}
see also \cite[Lemma~7.2]{dirac}.

\begin{proof}[Proof of Theorem~\ref{theorem subprincipal symbol operators 1 forms}]
Let us fix a point $z\in M$.  Let $\mathcal{U}_z$ be a small open neighbourhood of $z$ and let $\{\tilde{e}_j\}_{j=1}^d$ be a Levi-Civita framing centred at $z$. 

Consider the operator $S$ mapping, locally, 1-forms on $\mathcal{U}_z$ to sections of the trivial $\mathbb{R}^d$-bundle over $\mathcal{U}_z$,  defined in accordance with
\begin{equation}
\label{proof subprincipal symbol operators 1 forms equation 1}
S: u_\alpha\mapsto \tilde e_j{}^\beta\, u_\beta.
\end{equation}

We claim that \eqref{subprincipal symbol operators 1-forms} is the map
\begin{equation}
\label{proof subprincipal symbol operators 1 forms equation 2}
(z ,\xi)\mapsto \left[ S^{-1}(\rho^{1/2}SQS^{-1} \rho^{-1/2})_\mathrm{sub} S\right]_\alpha{}^\beta(z,\xi),
\end{equation}
where 
\begin{equation*}
\label{proof subprincipal symbol operators 1 forms equation 3}
(\rho^{1/2}SQS^{-1} \rho^{-1/2})_\mathrm{sub}
\end{equation*} 
denotes the usual subprincipal symbol \cite[Eqn. (5.2.8)]{DuHo} of the operator $Q_{1/2}:=\rho^{1/2}SQS^{-1} \rho^{-1/2}$ acting on $d$-columns of half-densities.

The task at hand is to show, by means of an explicit computation, that 
\eqref{proof subprincipal symbol operators 1 forms equation 2}
and
\eqref{subprincipal symbol operators 1-forms} coincide.  As a by-product of the upcoming calculation and \eqref{proof subprincipal symbol operators 1 forms equation 2} it will follow that \eqref{subprincipal symbol operators 1-forms} is covariant under changes of local coordinates and independent of the choice of the Levi-Civita framing $\{\tilde e_j\}_{j=1}^3$.

The operator $Q_{1/2}$ acts as
\begin{equation*}
\label{proof subprincipal symbol operators 1 forms equation 4}
(Q_{1/2}f)_j(x)=\frac{1}{(2\pi)^d}\int e^{i(x-y)^\gamma\xi_\gamma}\,[\tilde q_{1/2}]_j{}^k(x,y,\xi)\,
f_k(y)\,\dr y\,\dr \xi\,,
\end{equation*}
where
\begin{equation*}
\label{proof subprincipal symbol operators 1 forms equation 5}
[\tilde q_{1/2}]_j{}^k(x,y,\xi):=\left[\frac{\rho(x)}{\rho(y)} \right]^{1/2} \tilde e_j{}^\alpha(x)\, q_{\alpha}{}^\beta(x,\xi) \,\tilde e\,{}^k{}_\beta(y)\,.
\end{equation*}
The left (polyhomogeneous) symbol $q_{1/2}\sim \sum_{k=0}^{+\infty} \,[q_{1/2}]_{s-k}$ of $Q_{1/2}$ is obtained by excluding the $y$-dependence from $\tilde q_{1/2}$ via the amplitude-to-symbol operator
\begin{equation}
\label{proof subprincipal symbol operators 1 forms equation 6}
\mathcal{S}_\mathrm{left}\sim \sum_{k=0}^{+\infty} \mathcal{S}_\mathrm{-k}, \qquad \mathcal{S}_0:=\left.(\,\cdot\,)\right|_{y=x}, \qquad \mathcal{S}_{-k}:=\frac{1}{k!}\left. \left[ \left(-i\dfrac{\partial^2}{\partial y^\gamma\partial\xi_\gamma} \right)^k(\,\cdot\,)\right] \right|_{y=x}.
\end{equation}

The terms positively homogeneous in $\xi$ of degree $s$ and $s-1$ in 
$$
[q_{1/2}]_j{}^k(x,\xi):=\mathcal{S}_\mathrm{left}([\tilde q_{1/2}]_j{}^k(x,y,\xi))
$$ 
read
\begin{equation}
\label{proof subprincipal symbol operators 1 forms equation 7}
[(q_{1/2})_s]_j{}^k(x,\xi)=\tilde e_j{}^\alpha(x)\, [q_s]_{\alpha}{}^\beta(x,\xi) \,\tilde e\,{}^k{}_\beta(x)
\end{equation}
and
\begin{equation}
\label{proof subprincipal symbol operators 1 forms equation 8}
[(q_{1/2})_{s-1}]_j{}^k(x,\xi)=\tilde e_j{}^\alpha\,[q_{s-1}]_\alpha{}^\beta \tilde e\,{}^{k}{}_\beta-i \tilde e_j{}^\alpha\dfrac{\partial [q_s]_\alpha{}^\beta}{\partial\xi_\gamma}\dfrac{\partial \tilde e\,{}^{k}{}_\beta}{\partial x^\gamma}
+\frac{i}2  (\ln \rho)_{x^\gamma} \,\tilde e_j{}^\alpha\dfrac{\partial [q_s]_\alpha{}^\beta}{\partial\xi_\gamma}\tilde e\,{}^{k}{}_\beta\,,
\end{equation}
respectively.
In the RHS of \eqref{proof subprincipal symbol operators 1 forms equation 8} framings are evaluated at $x$,  and homogeneous components of $q$ and their derivatives at $(x,\xi)$.

On account of \eqref{proof subprincipal symbol operators 1 forms equation 7} and \eqref{proof subprincipal symbol operators 1 forms equation 8} one obtains
\begin{multline}
\label{proof subprincipal symbol operators 1 forms equation 9}
[(Q_{1/2})_\sub]_j{}^k(z,\xi)=
\left.\left([(q_{1/2})_{s-1}]_j{}^k(x,\xi)+\frac{i}{2}\dfrac{\partial^2 [(q_{1/2})_s]_{\color{black}j}{}^{\color{black}k}}{\partial x^\gamma\partial\xi_\gamma}(x,\xi) \right) \right|_{x=z}
\\
=
\tilde e_j{}^\alpha\,\left([{\color{black}{q_{s-1}}}]_\alpha{}^\beta+\frac{i}2\dfrac{\partial^2 [{\color{black}{q_s}}]_\alpha{}^\beta}{\partial x^\gamma\partial\xi_\gamma} \right)\tilde e\,{}^{k}{}_\beta
+\frac{i}2
\dfrac{\partial \tilde e_j{}^\alpha}{\partial x^\gamma}\dfrac{\partial [{\color{black}{q_s}}]_\alpha{}^\beta}{\partial\xi_\gamma}\tilde e\,{}^{k}{}_\beta
\\
-
\frac{i}2
\tilde e_j{}^\alpha \dfrac{\partial [{\color{black}{q_s}}]_\alpha{}^\beta}{\partial\xi_\gamma} \dfrac{\partial \tilde e\,{}^{k}{}_\beta}{\partial x^\gamma}
+\frac{i}2  (\ln \rho)_{x^\gamma} \,\tilde e_j{}^\alpha\dfrac{\partial [{\color{black}{q_s}}]_\alpha{}^\beta}{\partial\xi_\gamma}\tilde e\,{}^{k}{}_\beta\,.
\end{multline}
In the RHS of \eqref{proof subprincipal symbol operators 1 forms equation 9} framings are evaluated at $z$,  and homogeneous components of $q$ and their derivatives at $(z,\xi)$.

In view of the identities
\begin{equation}
\label{proof subprincipal symbol operators 1 forms equation 10}
\left. \tilde e\,{}^j{}_\mu\dfrac{\partial \tilde e_j{}^\alpha}{\partial x^\gamma}\right|_{x=z}=-\Gamma^\alpha{}_{\gamma \mu}(z), \qquad 
\left.\dfrac{\partial \tilde e\,{}^{k}{}_\beta}{\partial x^\gamma}\,\tilde e_k{}^\nu\right|_{x=z}=\Gamma^\nu{}_{\gamma\beta}(z),
\end{equation}
which follow from \eqref{14 June 2021 equation 1}, \eqref{14 June 2021 equation 2},
and the identity \eqref{relation between Riemannian density and Christoffel symbols},
a straightforward calculation tells us that the quantity
\begin{equation}
\label{proof subprincipal symbol operators 1 forms equation 12}
\tilde e\,{}^j{}_\mu(z)\,[(Q_{1/2})_\sub]_j{}^k(z,\xi)\, \tilde e_k{}^\nu(z)
\end{equation}
coincides with \eqref{subprincipal symbol operators 1-forms}.

The quantity \eqref{proof subprincipal symbol operators 1 forms equation 12} is invariant under rigid orthogonal transformations of the Levi-Civita framing, hence, this quantity is a covariant object. As \eqref{proof subprincipal symbol operators 1 forms equation 12} coincides with \eqref{subprincipal symbol operators 1-forms}, this implies that \eqref{subprincipal symbol operators 1-forms} is covariant.
\end{proof}

Let us consider the special case when $M$ is an oriented 3-manifold. In this case we have the following result.

\begin{lemma}
\label{lemma subprincipal symbol curl}
The subprincipal symbol of $\operatorname{curl}$ is zero,
\begin{equation}
\label{lemma subprincipal symbol curl equation}
(\operatorname{curl})_\mathrm{sub}=0.
\end{equation}
\end{lemma}
\begin{proof}
It follows from \eqref{principal symbol curl} that
\begin{equation}
\label{proof subprincipal curl equation 1}
\frac{\partial [\curl_\prin]_\alpha{}^\beta}{\partial \xi_\gamma}=-i\,E_\alpha{}^{\beta\gamma}(x) 
=
\frac{i}{\rho(x)}\varepsilon^{\gamma\beta \mu} \,g_{\mu\alpha}(x),
\end{equation}
where the second equality is a straightforward consequence of \eqref{main theorem 1 equation 3} and the properties of the totally antisymmetric symbol.

In light of \eqref{relation between Riemannian density and Christoffel symbols},
the substitution of \eqref{proof subprincipal curl equation 1} and \eqref{principal symbol curl} into \eqref{subprincipal symbol operators 1-forms} gives us
\begin{equation*}
\label{proof subprincipal curl equation 2}
\begin{split}
2\rho\, [\curl_\sub]_\alpha{}^\beta
&
=
-\rho\frac{
\partial
[\rho^{-1}\,\varepsilon^{\nu\beta\mu}\,g_{\mu\alpha}]
}
{\partial x^\nu}
-
\Gamma^{\lambda}{}_{\nu\lambda}
[\varepsilon^{\nu\beta\mu}\,g_{\mu\alpha}]
+
\Gamma^\lambda{}_{\nu \alpha}
[\varepsilon^{\nu\beta\mu}\,g_{\mu\lambda}]
+
\Gamma^\beta{}_{\nu\lambda}
[\varepsilon^{\nu\lambda\mu}\,g_{\mu\alpha}]
\\
&
=
-
\varepsilon^{\nu\beta\mu}
\frac{\partial g_{\mu\alpha}}{\partial x^\nu}
+
\varepsilon^{\nu\beta\mu}\,\Gamma^\lambda{}_{\nu \alpha}\,g_{\mu\lambda}
\\
&
=
-
\varepsilon^{\nu\beta\mu}
\left(\Gamma^\lambda{}_{\nu \alpha}g_{\mu\lambda}+\Gamma^\lambda{}_{\nu \mu}g_{\alpha\lambda} \right)
+
\varepsilon^{\nu\beta\mu}\,\Gamma^\lambda{}_{\nu \alpha}\,g_{\mu\lambda}
\\
&
=
0.
\end{split}
\end{equation*}

Alternatively, one can prove \eqref{lemma subprincipal symbol curl equation} by fixing an arbitrary point on $M$ and carrying out the above calculations in normal coordinates, using the fact that the subprincipal symbol is covariant by Theorem~\ref{theorem subprincipal symbol operators 1 forms}.
\end{proof}

Returning to the analysis of the general case of a Riemannian manifold of arbitrary dimension, let us examine the second order differential operators $\delta d$ and $d\delta$ acting on 1-forms. Similarly to \eqref{lemma subprincipal symbol curl equation}, one can show that
\begin{equation}
\label{subprincipal symbol of delta d}
(\delta \dr)_\mathrm{sub}=0,
\end{equation}
\begin{equation}
\label{subprincipal symbol of d delta}
(\dr\delta)_\mathrm{sub}=0.
\end{equation}
Of course, formulae \eqref{subprincipal symbol of delta d} and \eqref{subprincipal symbol of d delta} imply that the Hodge Laplacian $\,\boldsymbol{\Delta}:=-(\delta \dr+\dr\delta)\,$ acting on 1-forms has zero subprincipal symbol.

The inner product
\begin{equation*}
\label{general inner product on 1-forms}
\langle u,v \rangle:=\int_M g^{\alpha\beta}(x)\,\overline{u_\alpha(x)}\,v_\beta(x)\,\rho(x)\,\dr x
\end{equation*}
allows us to define the formal adjoint $Q^*$ of a pseudodifferential operator $Q$ acting on 1-forms.

\begin{lemma}
\label{lemma about the formal adjoint of a pseudodifferential operator}
We have
\begin{eqnarray}
\label{lemma about the formal adjoint of a pseudodifferential operator equation principal}
[(Q^*)_\mathrm{prin}]_\mu{}^\nu
\!\!\!
&=
g_{\mu\beta}
\,
\overline{
[Q_\mathrm{prin}]_\alpha{}^\beta
}
\,
g^{\alpha\nu}\,,
\\
\label{lemma about the formal adjoint of a pseudodifferential operator equation subprincipal}
[(Q^*)_\mathrm{sub}]_\mu{}^\nu
\!\!\!
&=
g_{\mu\beta}
\,
\overline{
[Q_\mathrm{sub}]_\alpha{}^\beta
}
\,
g^{\alpha\nu}\,.
\end{eqnarray}
\end{lemma}
\begin{proof}
Formula \eqref{lemma about the formal adjoint of a pseudodifferential operator equation principal} is obvious.
As to formula \eqref{lemma about the formal adjoint of a pseudodifferential operator equation subprincipal}, it can be established by a lengthy explicit calculation, which shows that all additional terms involving derivatives of the metric and the Riemannian density cancel out.
However, a shorter way of proving it is to argue as in the alternative version of the proof of Lemma~\ref{lemma subprincipal symbol curl}: fix an arbitrary point on $M$ and carry out calculations in normal coordinates, using the fact that the subprincipal symbol is covariant by Theorem~\ref{theorem subprincipal symbol operators 1 forms}.
\end{proof}

\begin{theorem}
\label{theorem about subprincipal symbol of composition}
Let $Q$ and $R$ pseudodifferential operators acting on 1-forms. Then
\begin{equation}
\label{theorem about subprincipal symbol of composition equation 1}
(QR)_\mathrm{sub}
=
Q_\mathrm{prin}
R_\mathrm{sub}
+
Q_\mathrm{sub}
R_\mathrm{prin}
+\frac i2
\{\{Q_\mathrm{prin},R_\mathrm{prin}\}\},
\end{equation}
where 
\begin{multline}
\label{theorem about subprincipal symbol of composition equation 2}
\{\{Q_\mathrm{prin},R_\mathrm{prin}\}\}_\alpha{}^\beta
:=
\left(
\frac{\partial [Q_\mathrm{prin}]_\alpha{}^\kappa}{\partial x^\gamma}
-
\Gamma^{\alpha'}{}_{\gamma\alpha}[Q_\mathrm{prin}]_{\alpha'}{}^\kappa
+
\Gamma^\kappa{}_{\gamma\kappa'}[Q_\mathrm{prin}]_\alpha{}^{\kappa'}
\right)
\frac{\partial[R_\mathrm{prin}]_\kappa{}^\beta}{\partial\xi_\gamma}
\\
-
\frac{\partial[Q_\mathrm{prin}]_\alpha{}^\kappa}{\partial\xi_\gamma}
\left(
\frac{\partial [R_\mathrm{prin}]_\kappa{}^\beta}{\partial x^\gamma}
-
\Gamma^{\kappa'}{}_{\gamma\kappa}[R_\mathrm{prin}]_{\kappa'}{}^\beta
+
\Gamma^\beta{}_{\gamma\beta'}[R_\mathrm{prin}]_\kappa{}^{\beta'}
\right)
\end{multline}
is the generalised Poisson bracket.
\end{theorem}


\begin{remark}
Note that if we \emph{formally} treat principal symbols as $(1,1)$-tensors on $M$ and put
\begin{equation*}
\label{6 July 2021 equation 3}
\widetilde\nabla_\alpha[Q_\mathrm{prin}]_\beta{}^\gamma
:=
\frac{\partial [Q_\mathrm{prin}]_\beta{}^\gamma}{\partial x^\alpha}
-
\Gamma^{\beta'}{}_{\alpha\beta}[Q_\mathrm{prin}]_{\beta'}{}^\gamma
+
\Gamma^\gamma{}_{\alpha\gamma'}[Q_\mathrm{prin}]_\beta{}^{\gamma'},
\end{equation*}
then formula \eqref{theorem about subprincipal symbol of composition equation 2} simplifies and takes the form
\begin{equation*}
\label{6 July 2021 equation 4}
\{\{Q_\mathrm{prin},R_\mathrm{prin}\}\}_\alpha{}^\beta
=
\widetilde\nabla_\gamma[Q_\mathrm{prin}]_\alpha{}^\kappa
\,
\frac{\partial[R_\mathrm{prin}]_\kappa{}^\beta}{\partial\xi_\gamma}
-
\frac{\partial[Q_\mathrm{prin}]_\alpha{}^\kappa}{\partial\xi_\gamma}
\,
\widetilde\nabla_\gamma[R_\mathrm{prin}]_\kappa{}^\beta\,.
\end{equation*}
Compare with Remark~\ref{remark minus sign}.

It is possible to introduce genuine covariant derivatives on the cotangent bundle and reformulate formulae \eqref{subprincipal symbol operators 1-forms} and \eqref{theorem about subprincipal symbol of composition equation 2} in terms of the latter; see, for example, \cite[Remark~4.5]{dirac}. We refrain from doing this because it would take us away from the core subject of our paper.
\end{remark}

\begin{proof}[Proof of Theorem~\ref{theorem about subprincipal symbol of composition}]
Let us fix a point $z\in M$.  Let $\mathcal{U}_z$ be a small open neighbourhood of $z$ and let $\{\tilde{e}_j\}_{j=1}^3$ be a Levi-Civita framing centred at $z$. 
Arguing as in the proof of Theorem~\ref{theorem subprincipal symbol operators 1 forms}, we denote by $Q_{1/2}$ and $R_{1/2}$ the pseudodifferential operators acting on half-densities obtained from the operators $Q$ and $R$ (which act on 1-forms), see formula \eqref{proof subprincipal symbol operators 1 forms equation 1} and following text.

The task at hand is to show that, when $x=z$, the RHS of \eqref{theorem about subprincipal symbol of composition equation 1} coincides with the quantity
\begin{equation*}
\label{proof subprincipal symbol composition equation 1}
\left.\left(\tilde e\,{}^j{}_\alpha\,[(Q_{1/2}R_{1/2})_\sub]_j{}^k \,\tilde e_k{}^\beta\right)\right|_{x=z}\,,
\end{equation*}
where $(Q_{1/2}R_{1/2})_\sub$ denotes the usual subprincipal symbol for operators acting on half-densities and for which the composition formula is known
\begin{equation*}
\label{proof subprincipal symbol composition equation 2}
(Q_{1/2}R_{1/2})_\sub=(Q_{1/2})_\prin (R_{1/2})_\sub+(Q_{1/2})_\sub (R_{1/2})_\prin+\frac{i}2\{(Q_{1/2})_\prin ,(R_{1/2})_\prin \}\,.
\end{equation*}
Here
\begin{equation*}
\label{Poisson bracket on matrix-functions}
\{(Q_{1/2})_\prin ,(R_{1/2})_\prin\}
:=
\frac{\partial(Q_{1/2})_\prin}{\partial x^\gamma}
\,
\frac{\partial(R_{1/2})_\prin}{\partial\xi_\gamma}
\,-\,
\frac{\partial(Q_{1/2})_\prin}{\partial\xi_\gamma}
\,
\frac{\partial(R_{1/2})_\prin}{\partial x^\gamma}
\end{equation*}
is the Poisson bracket on matrix-functions.

We have
\begin{multline}
\label{dima composition proof equation 1}
\left.
\left(
\tilde e\,{}^j{}_\alpha\,[(Q_{1/2})_\prin(R_{1/2})_\sub]_j{}^k \,\tilde e_k{}^\beta
\right)
\right|_{x=z}
=
\left.
\left(
\tilde e\,{}^j{}_\alpha
\,[(Q_{1/2})_\prin]_j{}^l
\,[(R_{1/2})_\sub]_l{}^k \,\tilde e_k{}^\beta
\right)
\right|_{x=z}
\\
=
\left.
\left(
\tilde e\,{}^j{}_\alpha
\,[(Q_{1/2})_\prin]_j{}^l
\,\tilde e_l{}^\kappa
\,\tilde e\,{}^m{}_\kappa
\,[(R_{1/2})_\sub]_m{}^k 
\,\tilde e_k{}^\beta
\right)
\right|_{x=z}
\\
=
\left.
\left(
\tilde e\,{}^j{}_\alpha
\,[(Q_{1/2})_\prin]_j{}^l
\,\tilde e_l{}^\kappa
\right)
\right|_{x=z}
\left.
\left(
\tilde e\,{}^m{}_\kappa
\,[(R_{1/2})_\sub]_m{}^k 
\,\tilde e_k{}^\beta
\right)
\right|_{x=z}
\\
=
\left.
\left(
[Q_\prin]_\alpha{}^\kappa
\right)
\right|_{x=z}
\left.
\left(
[R_\sub]_\kappa{}^\beta
\right)
\right|_{x=z}
=
\left.
\left(
[Q_\prin\,R_\sub]_\alpha{}^\beta
\right)
\right|_{x=z}\,.
\end{multline}
Here we used the fact that, according to \eqref{proof subprincipal symbol operators 1 forms equation 2}, the expression
$
\left.
\left(
\tilde e\,{}^m{}_\kappa
\,[(R_{1/2})_\sub]_m{}^k 
\,\tilde e_k{}^\beta
\right)
\right|_{x=z}
$
is precisely
$\left.
\left(
[R_\sub]_\kappa{}^\beta
\right)
\right|_{x=z}\,$.

Similarly, we have
\begin{equation}
\label{dima composition proof equation 2}
\left.
\left(
\tilde e\,{}^j{}_\alpha\,[(Q_{1/2})_\sub(R_{1/2})_\prin]_j{}^k \,\tilde e_k{}^\beta
\right)
\right|_{x=z}
=
\left.
\left(
[Q_\sub\,R_\prin]_\alpha{}^\beta
\right)
\right|_{x=z}\,.
\end{equation}
In view of \eqref{dima composition proof equation 1} and \eqref{dima composition proof equation 2} the proof of the theorem reduces to establishing the identity
\begin{equation}
\label{dima composition proof equation 3}
\left.
\left(
\{\{Q_\mathrm{prin},R_\mathrm{prin}\}\}_\alpha{}^\beta
\right)
\right|_{x=z}
=
\left.
\left(
\tilde e\,{}^j{}_\alpha\,\{(Q_{1/2})_\prin ,(R_{1/2})_\prin \}_j{}^k \,\tilde e_k{}^\beta
\right)
\right|_{x=z}\,.
\end{equation}

Recalling that
\begin{equation*}
\label{proof subprincipal symbol composition equation 3}
[(Q_{1/2})_\prin]_j{}^l= \tilde e_j{}^\mu\,[Q_\prin]_\mu{}^\nu\,\tilde e\,{}^l{}_\nu\,,
\qquad
[(R_{1/2})_\prin]_l{}^k= \tilde e_l{}^\rho\,[R_\prin]_\rho{}^\sigma\,\tilde e\,{}^k{}_\sigma\,,
\end{equation*}
a straightforward calculation gives us
\begin{multline*}
\left.\{(Q_{1/2})_\prin ,(R_{1/2})_\prin \}_j{}^k\right|_{x=z}
\\
=
\Bigl[
\tilde e_j{}^\mu [Q_\prin]_\mu{}^\nu   \dfrac{\partial \tilde e\,{}^l{}_\nu}{\partial x^\gamma} \tilde e_l{}^\rho
\frac{\partial[R_\prin]_\rho{}^\sigma}{\partial\xi_\gamma}
\tilde e\,{}^k{}_\sigma
-
\tilde e_j{}^\mu
\frac{\partial[Q_\prin]_\mu{}^\nu}{\partial\xi_\gamma}
\tilde e\,{}^l{}_\nu
\dfrac{\partial \tilde e_l{}^\rho}{\partial x^\gamma}[R_\prin]_\rho{}^\sigma \tilde e\,{}^k{}_\sigma
\\
+
\frac{\partial \tilde e_j{}^\mu}{\partial x^\gamma}[Q_\prin]_\mu{}^\nu
\frac{\partial[R_\prin]_\nu{}^\sigma}{\partial\xi_\gamma}
\tilde e\,{}^k{}_\sigma
-
\tilde e_j{}^\mu
\frac{\partial[Q_\prin]_\mu{}^\nu}{\partial\xi_\gamma}
[R_\prin]_\nu{}^\sigma \dfrac{\partial \tilde e\,{}^k{}_\sigma}{\partial x^\gamma}
\\
\left.
+
\tilde e_j{}^\mu\,\{Q_\prin,R_\prin\}_\mu{}^\sigma\,\tilde e\,{}^k{}_\sigma\Bigr]\right|_{x=z}\,,
\end{multline*}
so that
\begin{multline}
\label{proof subprincipal symbol composition equation 5}
\left.
\left(
\tilde e\,{}^j{}_\alpha\,\{(Q_{1/2})_\prin ,(R_{1/2})_\prin \}_j{}^k\,\tilde e_k{}^\beta
\right)
\right|_{x=z}
\\
=
\Bigl[
[Q_\prin]_\alpha{}^\nu   \dfrac{\partial \tilde e\,{}^l{}_\nu}{\partial x^\gamma} \tilde e_l{}^\rho
\frac{\partial[R_\prin]_\rho{}^\beta}{\partial\xi_\gamma}
-
\frac{\partial[Q_\prin]_\alpha{}^\nu}{\partial\xi_\gamma}
\tilde e\,{}^l{}_\nu
\dfrac{\partial \tilde e_l{}^\rho}{\partial x^\gamma}[R_\prin]_\rho{}^\beta
\\
+
\tilde e\,{}^j{}_\alpha
\frac{\partial \tilde e_j{}^\mu}{\partial x^\gamma}[Q_\prin]_\mu{}^\nu
\frac{\partial[R_\prin]_\nu{}^\beta}{\partial\xi_\gamma}
-
\frac{\partial[Q_\prin]_\alpha{}^\nu}{\partial\xi_\gamma}
[R_\prin]_\nu{}^\sigma \dfrac{\partial \tilde e\,{}^k{}_\sigma}{\partial x^\gamma}\tilde e_k{}^\beta
\\
\left.
+
\{Q_\prin,R_\prin\}_\alpha{}^\beta\Bigr]\right|_{x=z}\,.
\end{multline}

On account of \eqref{proof subprincipal symbol operators 1 forms equation 10} and \eqref{theorem about subprincipal symbol of composition equation 2}, formula \eqref{proof subprincipal symbol composition equation 5} implies
\eqref{dima composition proof equation 3}.
\end{proof}

\begin{example}
Let $M$ be an oriented 3-manifold. Then Theorem~\ref{theorem about subprincipal symbol of composition} and Lemma~\ref{lemma subprincipal symbol curl} imply
\begin{equation}
\label{subprincipal symbol curl squared}
(\operatorname{curl}^2)_\mathrm{sub}=0.
\end{equation}
Alternatively, formula \eqref{subprincipal symbol curl squared} follows from \eqref{subprincipal symbol of delta d} and the fact that $\operatorname{curl}^2=\delta \dr$.
\end{example}

\section{Trace of pseudodifferential operators acting on 1-forms}
\label{Trace of pseudodifferential operators acting on 1-forms}

The aim of this section is to extend, in a natural way, the notion of trace of a pseudo\-differential operator acting on 1-forms. Given a pseudodifferential operator $Q$ acting on 1-forms, we will define a scalar operator $\operatorname{\mathfrak{tr}} Q$ is such a way that, if $Q$ is of trace class, then the operator-theoretic trace of $Q$ coincides with the operator-theoretic trace of $\operatorname{\mathfrak{tr}} Q$. Later in the paper we will use this definition without assuming that $Q$ is of trace class: it may happen that $Q$ itself is not of trace class, but $\operatorname{\mathfrak{tr}} Q$ is --- which is the case, up to an additional minor regularisation, for $Q=\theta(\operatorname{curl})-\theta(-\operatorname{curl})$.

In this section, as in the previous one, we work on a closed Riemannian manifold $(M,g)$ of arbitrary dimension $d$.

Let us start with the simpler case when $Q$ is a pseudodifferential operator of order $s$ acting on $m$-columns of scalar fields. Suppose that
\begin{equation}
\label{s less than d}
s<-d.
\end{equation}
Then \cite[\S 12.1]{shubin} $Q$ is an integral operator with continuous integral kernel,
\begin{equation}
\label{Q acting on column of scalars 1}
Q: u_j(x)\mapsto\int_M\mathfrak{q}_j{}^k(x,y)\,u_k(y)\,\rho(y)\,\dr y\,.
\end{equation}
We introduced the factor $\rho(y)$ in the above integral in order to make the integral kernel $\mathfrak{q}_j{}^k(x,y)$ a (matrix-valued) scalar function on $M\times M$.
Furthermore, if $Q$ is self-adjoint, then it is of trace class and we have
\begin{equation}
\label{Q acting on column of scalars 2}
\operatorname{Tr}Q=\int_M(\operatorname{tr}\mathfrak{q})(x,x)\,\rho(x)\,\dr x\,,
\end{equation}
where $(\operatorname{tr}\mathfrak{q})(x,y):=\mathfrak{q}_j{}^j(x,y)$ is the trace of the matrix-function $\mathfrak{q}_j{}^k(x,y)$. We define the scalar operator
\begin{equation}
\label{Q acting on column of scalars 3}
\operatorname{tr}Q: f(x)\mapsto\int_M (\operatorname{tr}\mathfrak{q})(x,y)\,f(y)\,\rho(y)\,\dr y
\end{equation}
and call it \emph{the matrix trace of the operator $Q$}. It is easy to see that
\begin{equation*}
\label{Q acting on column of scalars 4}
\operatorname{Tr}Q=\operatorname{Tr}(\operatorname{tr}Q).
\end{equation*}

Note that \eqref{Q acting on column of scalars 3} is well defined for any $s\in\mathbb{R}$, not necessarily satisfying condition \eqref{s less than d}. More precisely, if $Q$ is of order $s$, it is easy to see that
\begin{enumerate}[(i)]
\item
$\operatorname{tr}Q$ is also an operator of order $s$,
\item
$(\operatorname{tr}Q)^*=\operatorname{tr}(Q^*)$, where the star refers to formal adjoints with respect to the natural inner products,
\item
$(\operatorname{tr}Q)_\mathrm{prin}=\operatorname{tr}(Q_\mathrm{prin})$,
\item
$(\operatorname{tr}Q)_\mathrm{sub}=\operatorname{tr}(Q_\mathrm{sub})$.
\end{enumerate}
Of course, when condition \eqref{s less than d} is not satisfied the integral kernel $\mathfrak{q}_j{}^k(x,y)$ appearing in formula \eqref{Q acting on column of scalars 1} should be understood in the distributional sense (Schwartz kernel), but this does not prevent us from taking its matrix trace. The distribution $(\operatorname{tr}\mathfrak{q})(x,y)$ will be the Schwartz kernel of the scalar operator \eqref{Q acting on column of scalars 3}.

We will now adapt the above construction to the case of a pseudodifferential operator
\begin{equation}
\label{Q acting on 1-forms 1}
Q: u_\alpha(x)\mapsto\int_M \mathfrak{q}_\alpha{}^\beta(x,y)\,u_\beta(y)\,\rho(y)\,\dr y
\end{equation}
acting on 1-forms. Here we encounter a problem: the quantity $\mathfrak{q}_\alpha{}^\alpha(x,y)$ is not a scalar on $M\times M$ because the tensor indices $\alpha$ and $\beta$ in $q_\alpha{}^\beta(x,y)$ `live' at different points, $x$ and $y$ respectively. More precisely, the problem here is that the Schwartz kernel $\mathfrak{q}_\alpha{}^\beta(x,y)$ of the operator \eqref{Q acting on 1-forms 1} is a two-point tensor: it is a covector at $x$ in the tensor index $\alpha$ and a vector at $y$ in the tensor index $\beta$. We overcome this impediment as follows.

Let $x$ and $y$ be two points on the manifold $M$ which are sufficiently close. Then there is a unique shortest geodesic $\gamma(x,y)$ connecting $x$ and $y$. Given a point $z\in\gamma(x,y)$, put
\begin{equation*}
\label{Q acting on 1-forms 2}
t=\frac{\dist(z,x)}{\dist(x,y)}\,,
\end{equation*}
which is the variable arc length of the geodesic normalised by its total length. The variable~$t$ provides a convenient parameterisation of our geodesic $\gamma(x,y;\,\cdot\,):[0,1]\to M$,
so that $\gamma(x,y;0)=x$ and $\gamma(x,y;1)=y$.

We denote by
\begin{equation}
\label{Q acting on 1-forms 3}
Z: T_xM\ni u^\alpha\mapsto u^\alpha\,Z_\alpha{}^\beta(x,y)\in T_yM
\end{equation}
the linear map realising the parallel transport of vectors from $x$ to $y$ along the unique shortest geodesic connecting $x$ and $y$. Note that the result of parallel transport is independent of the choice of the particular parameterisation of the curve along which it takes place.

In what follows we will be raising and lowering indices in the 2-point tensor $Z_\alpha{}^\beta(x,y)$ using the Riemannian metric $g(x)$ in the first index and $g(y)$ in the second. Of course,
\begin{equation}
\label{Q acting on 1-forms 3a extra}
Z_\alpha{}^\kappa(x,y)\,Z_\kappa{}^\beta(y,x)
\,=\,
\delta_\alpha{}^\beta\,,
\qquad
Z^\alpha{}_\kappa(x,y)\,Z^\kappa{}_\beta(y,x)
\,=\,
\delta^\alpha{}_\beta\,,
\end{equation}
\begin{equation}
\label{Q acting on 1-forms 3a}
Z_\alpha{}^\kappa(x,y)\,Z^\beta{}_\kappa(x,y)
\,=\,
\delta_\alpha{}^\beta\,,
\qquad
Z_\kappa{}^\alpha(x,y)\,Z^\kappa{}_\beta(x,y)
\,=\,
\delta^\alpha{}_\beta\,.
\end{equation}
Formula \eqref{Q acting on 1-forms 3a extra} expresses the fact that if we parallel transport a vector/covector from $x$ to $y$ along the shortest geodesic connecting these two points, and then parallel transport it back from $y$ to $x$ along the same geodesic, we get the original vector/covector.
Formula \eqref{Q acting on 1-forms 3a} is a consequence of the fact that the Levi-Civita connection is metric compatible.

\begin{remark}
For later use, let us observe that formulae \eqref{Q acting on 1-forms 3a extra} and \eqref{Q acting on 1-forms 3a} imply
\begin{equation}
\label{aux4}
Z_\beta{}^\alpha(x,y)
=
Z^\alpha{}_\beta(y,x).
\end{equation}
Indeed, take the first identity \eqref{Q acting on 1-forms 3a extra}, multiply it by $Z^\alpha{}_\gamma(x,y)$ and then use the second identity \eqref{Q acting on 1-forms 3a}. This gives us
$\,\delta^\kappa{}_\gamma\,Z_\kappa{}^\beta(y,x)
\,=\,
\delta_\alpha{}^\beta
\,Z^\alpha{}_\gamma(x,y)
\,$,
which is equivalent to \eqref{aux4}.
\end{remark}

Let us introduce a one-parameter family of scalar distributions defined by
\begin{equation}
\label{Q acting on 1-forms 5}
Z^\alpha{}_\kappa(x,\gamma(x,y;\tau))
\,
\mathfrak{q}_{\alpha}{}^\beta(x,y)
\,
Z_\beta{}^\kappa(y,\gamma(x,y;\tau))\,,
\end{equation}
where $\mathfrak{q}_{\alpha}{}^\beta(x,y)$ is the Schwartz kernel from \eqref{Q acting on 1-forms 1} and $\tau\in[0,1]$. What happens here is that tensor indices are parallel transported to the common point $z=\gamma(x,y;\tau)$ on the geodesic $\gamma(x,y)$ connecting $x$ and $y$. Thus, the quantity \eqref{Q acting on 1-forms 5} is a genuine scalar.

As we are dealing with distributions we need to clarify the precise meaning of formula \eqref{Q acting on 1-forms 5}.

Consider a map from smooth scalar functions of two variables to smooth $(1,1)$ two-point tensors defined as
{\color{black}
\begin{equation}
\label{Q acting on 1-forms 5 bis}
C^\infty(M\times M)\ni f(x,y)
\mapsto
Z^\alpha{}_\kappa(x,\gamma(x,y;\tau))
\,
Z_\beta{}^\kappa(y,\gamma(x,y;\tau))
\,
f(x,y)\,.
\end{equation}
}
The quantity $\mathfrak{q}_{\alpha}{}^\beta(x,y)$ in \eqref{Q acting on 1-forms 5}
is a distribution (continuous complex-valued linear functional) acting on smooth $(1,1)$ two-point tensors.
Formula \eqref{Q acting on 1-forms 5} should be understood as a composition of \textcolor{black}{$\mathfrak{q}$ and \eqref{Q acting on 1-forms 5 bis}.}

\begin{lemma}
\label{no dependence on tau}
The quantity \eqref{Q acting on 1-forms 5} does not depend on $\tau$.
\end{lemma}

\begin{proof}
The statement of the lemma is an immediate consequence of the fact that the Levi-Civita connection is metric compatible.
\end{proof}

Since the definition \eqref{Q acting on 1-forms 5} is independent of the choice of $\tau$, henceforth we will set, for convenience, $\tau=0$, in which case formula \eqref{Q acting on 1-forms 5} simplifies and reads
\begin{equation}
\label{Q acting on 1-forms 5 with tau zero}
\mathfrak{q}_{\alpha}{}^\beta(x,y)
\,
Z_\beta{}^\alpha(y,x)\,.
\end{equation}

There is still one problem with formula \eqref{Q acting on 1-forms 5 with tau zero}: the linear operator $Z$ appearing in this formula is defined only for $x$ and $y$ sufficiently close. In order to view \eqref{Q acting on 1-forms 5 with tau zero} as a well-defined distribution we need a cut-off. Let $\chi:[0,+\infty)\to \mathbb{R}$ be a compactly supported infinitely smooth scalar function such that $\chi=1$ in a neighbourhood of zero. We modify formula \eqref{Q acting on 1-forms 5} to read
\begin{equation}
\label{Q acting on 1-forms 6}
\mathfrak{q}_{\alpha}{}^\beta(x,y)
\,
Z_\beta{}^\alpha(y,x)
\,\chi(\operatorname{dist}(x,y)/\epsilon)\,,
\end{equation}
where $\epsilon>0$ is a small parameter which ensures that the quantity \eqref{Q acting on 1-forms 6} vanishes when $x$ and $y$ are not sufficiently close.

\begin{definition}
\label{definition of matrix trace of a pdo acting on 1-forms}
The scalar operator
\begin{equation}
\label{Q acting on 1-forms 7}
\operatorname{\mathfrak{tr}} Q: f(x)\mapsto\int_M \mathfrak{q}_{\alpha}{}^\beta(x,y)
\,
Z_\beta{}^\alpha(y,x)
\,\chi(\operatorname{dist}(x,y)/\epsilon)\,\,f(y)\,\rho(y)\,\dr y
\end{equation}
is called \emph{the matrix trace} of the operator \eqref{Q acting on 1-forms 1}.
\end{definition}

Thus, given a pseudodifferential operator $Q$ acting on 1-forms we defined the scalar pseudo\-differential operator $\operatorname{\mathfrak{tr}} Q\,$. The latter depends on the small parameter $\epsilon>0$ and the cut-off $\chi$, but, as we will see, the choice of this parameter and the cut-off does not affect the main results of our paper. Let us emphasise that the matrix trace \eqref{Q acting on 1-forms 7} is defined uniquely, modulo the addition of a scalar integral operator whose integral kernel is infinitely smooth and vanishes in a neighbourhood of the diagonal.

\begin{proposition}
\label{properties of matrix trace of a pdo acting on 1-forms lemma}
\phantom{.}
\begin{enumerate}[(a)]
\item
For $Q$ self-adjoint and under the condition \eqref{s less than d} we have
\begin{equation}
\label{properties of matrix trace of a pdo acting on 1-forms lemma equation 3}
\operatorname{Tr}(\operatorname{\mathfrak{tr}} Q)=\operatorname{Tr}Q\,.
\end{equation}
\item
We have
\begin{equation}
\label{properties of matrix trace of a pdo acting on 1-forms lemma equation 4}
(\operatorname{\mathfrak{tr}} Q)^*=\operatorname{\mathfrak{tr}}(Q^*)\,,
\end{equation}
where the star refers to formal adjoints with respect to the natural inner products on scalar functions and 1-forms respectively.
\item
We have
\begin{equation}
\label{properties of matrix trace of a pdo acting on 1-forms lemma equation 1}
(\operatorname{\mathfrak{tr}} Q)_\mathrm{prin}=\operatorname{tr}Q_\mathrm{prin}\,,
\end{equation}
\begin{equation}
\label{properties of matrix trace of a pdo acting on 1-forms lemma equation 2}
(\operatorname{\mathfrak{tr}} Q)_\mathrm{sub}=\operatorname{tr}Q_\mathrm{sub}\,.
\end{equation}
\end{enumerate}
\end{proposition}

\begin{remark}
Let us emphasise that, crucially, the left-hand sides of 
\eqref{properties of matrix trace of a pdo acting on 1-forms lemma equation 3},
\eqref{properties of matrix trace of a pdo acting on 1-forms lemma equation 1}
and
\eqref{properties of matrix trace of a pdo acting on 1-forms lemma equation 2}
do not depend on $\epsilon$.
\end{remark}

\begin{proof}[Proof of Proposition \ref{properties of matrix trace of a pdo acting on 1-forms lemma}]
\phantom{.}

(a) 
Formula \eqref{properties of matrix trace of a pdo acting on 1-forms lemma equation 3} follows from the fact that
$\,Z_\alpha{}^\beta(x,x)=\delta_\alpha{}^\beta\,$
and $\chi(0)=1$.

\

(b)
We have
\begin{multline}
\label{aux1}
(\operatorname{\mathfrak{tr}} Q)^*: f(x)\mapsto\int_M
\overline
{
(\operatorname{\mathfrak{tr}}\mathfrak{q})(y,x)
}
\,f(y)\,\rho(y)\,\dr y
\\
=
\int_M
\overline
{
\mathfrak{q}_{\alpha}{}^\beta(y,x)
}
\,
Z_\beta{}^\alpha(x,y)
\,\chi(\operatorname{dist}(x,y)/\epsilon)
\,f(y)\,\rho(y)\,\dr y\,,
\end{multline}
\begin{equation*}
\label{aux2}
Q^*: u_\alpha(x)\mapsto\int_M
\overline
{
\mathfrak{q}^\beta{}_\alpha(y,x)
}
\,u_\beta(y)\,\rho(y)\,\dr y\,,
\end{equation*}
\begin{multline}
\label{aux3}
\operatorname{\mathfrak{tr}}(Q^*): f(x)\mapsto
\int_M
\overline
{
\mathfrak{q}^\beta{}_\alpha(y,x)
}
\,
Z_\beta{}^\alpha(y,x)
\,\chi(\operatorname{dist}(x,y)/\epsilon)
\,f(y)\,\rho(y)\,\dr y
\\
=
\int_M
\overline
{
\mathfrak{q}_\alpha{}^\beta(y,x)
}
\,
Z^\alpha{}_\beta(y,x)
\,\chi(\operatorname{dist}(x,y)/\epsilon)
\,f(y)\,\rho(y)\,\dr y\,.
\end{multline}
Formulae
\eqref{aux1},
\eqref{aux3}
and
\eqref{aux4}
imply \eqref{properties of matrix trace of a pdo acting on 1-forms lemma equation 4}.

\

(c)
Formula \eqref{properties of matrix trace of a pdo acting on 1-forms lemma equation 1} follows from the fact that
$\,Z_\alpha{}^\beta(x,x)=\delta_\alpha{}^\beta\,$
and $\chi(0)=1$.

The proof of \eqref{properties of matrix trace of a pdo acting on 1-forms lemma equation 2} is more subtle.
Formula \eqref{11 June 2021 equation 1} tells us that, modulo an infinitely smooth contribution, the Schwartz kernel $\mathfrak{q}_\alpha{}^\beta(x,y)$ of the operator \eqref{Q acting on 1-forms 1} reads
\begin{equation}
\label{properties of matrix trace of a pdo acting on 1-forms lemma proof equation 1}
\mathfrak{q}_\alpha{}^\beta(x,y)
=
\frac{1}{(2\pi)^d\,\rho(y)}
\int
e^{i(x-y)^\gamma\xi_\gamma}\,q_\alpha{}^\beta(x,\xi)
\,\dr\xi\,,
\end{equation}
where the symbol $q_\alpha{}^\beta(x,\xi)$ admits the asymptotic expansion \eqref{11 June 2021 equation 2}.
We also have
\begin{equation}
\label{properties of matrix trace of a pdo acting on 1-forms lemma proof equation 2}
Z_\alpha{}^\beta(x,y)=\delta_\alpha{}^\beta
-\Gamma^\beta{}_{\lambda\alpha}(x)\,(y-x)^\lambda
+O(|y-x|^2)\,.
\end{equation}
Substituting \eqref{properties of matrix trace of a pdo acting on 1-forms lemma proof equation 1} and \eqref{properties of matrix trace of a pdo acting on 1-forms lemma proof equation 2} into \eqref{Q acting on 1-forms 6} we arrive at a representation of the scalar pseudo\-differential operator \eqref{Q acting on 1-forms 7} with amplitude depending on $x$, $y$ and $\xi$. In order to write down the (left) symbol of the operator \eqref{Q acting on 1-forms 7} we need to exclude the dependence on $y$, which is achieved in the standard manner \textcolor{black}{--- recall \eqref{proof subprincipal symbol operators 1 forms equation 6}.}
This produces the following expressions for the leading and subleading terms of the symbol $(\operatorname{\mathfrak{tr}} q)(x,\xi)$ of the scalar pseudodifferential operator \eqref{Q acting on 1-forms 7}:
\begin{equation}
\label{properties of matrix trace of a pdo acting on 1-forms lemma proof equation 5}
(\operatorname{\mathfrak{tr}} q)_s
=
[q_s]_\alpha{}^\alpha\,,
\end{equation}
\begin{equation}
\label{properties of matrix trace of a pdo acting on 1-forms lemma proof equation 6}
(\operatorname{\mathfrak{tr}} q)_{s-1}
=
[q_{s-1}]_\alpha{}^\alpha
\,-\,
i
\,\Gamma^\alpha{}_{\gamma\beta}
\,\frac{\partial[q_s]_\alpha{}^\beta}{\partial\xi_\gamma}\,.
\end{equation}
Note that the right-hand sides of \eqref{properties of matrix trace of a pdo acting on 1-forms lemma proof equation 5} and \eqref{properties of matrix trace of a pdo acting on 1-forms lemma proof equation 6} do not depend on the parameter $\epsilon$.

Substituting \eqref{properties of matrix trace of a pdo acting on 1-forms lemma proof equation 5} and \eqref{properties of matrix trace of a pdo acting on 1-forms lemma proof equation 6} into \eqref{subprincipal symbol of operator acting on a scalar field} and using \eqref{relation between Riemannian density and Christoffel symbols}, we get
\begin{equation}
\label{properties of matrix trace of a pdo acting on 1-forms lemma proof equation 7}
(\operatorname{\mathfrak{tr}} Q)_\mathrm{sub}
=
[q_{s-1}]_\alpha{}^\alpha
\,+\,
\frac{i}2\dfrac{\partial^2 [q_s]_\alpha{}^\alpha}{\partial x^\gamma\partial\xi_\gamma}
\,+\,
\frac{i}2
\,\Gamma^\beta{}_{\gamma\beta}
\,\frac{\partial [q_s]_\alpha{}^\alpha}{\partial\xi_\gamma}
\,-\,
i
\,\Gamma^\alpha{}_{\gamma\beta}
\,\frac{\partial[q_s]_\alpha{}^\beta}{\partial\xi_\gamma}\,.
\end{equation}
But taking the trace of \eqref{subprincipal symbol operators 1-forms} gives the same expression as in the right-hand side of \eqref{properties of matrix trace of a pdo acting on 1-forms lemma proof equation 7}. This completes the proof of formula \eqref{properties of matrix trace of a pdo acting on 1-forms lemma equation 2}.

Alternatively, one can prove formula \eqref{properties of matrix trace of a pdo acting on 1-forms lemma equation 2} by fixing an arbitrary point on $M$ and carrying out calculations in normal coordinates, using the covariance of subprincipal symbols defined in accordance with formulae \eqref{subprincipal symbol of operator acting on a scalar field} and \eqref{subprincipal symbol operators 1-forms}. With this approach all the Christoffel symbols disappear.
\end{proof}

Proposition \ref{properties of matrix trace of a pdo acting on 1-forms lemma} part (b) immediately implies

\begin{corollary}
\label{self-adjointness of matrix trace}
If $Q$ is formally self-adjoint then so is $\operatorname{\mathfrak{tr}} Q\,$.
\end{corollary}

\section{The projection operators $P_0$ and $P_\pm$}
\label{The operators P0 and Ppm}

This section is devoted to the study of the operators $P_0$ and $P_\pm$. We will show that they are pseudo\-differential operators whose full symbols can be constructed via an explicit algorithm. Furthermore, we will compute their principal and subprincipal symbols, thus proving Theorem~\ref{main theorem 1}.

\

By $\Psi^s$ we denote the space of classical pseudodifferential operators of order $s$
with polyhomogeneous symbols acting on 1-forms, recall \eqref{11 June 2021 equation 1}, \eqref{11 June 2021 equation 2}. We define
\begin{equation}
\label{Psi minus infinity}
\Psi^{-\infty}:=\bigcap_s \Psi^s
\end{equation}
and we write $Q=R \mod \Psi^{-\infty}$ if $Q-R$ is an integral operator with (infinitely) smooth integral kernel.

\

Let us begin by proving parts (a) and (b) of Theorem~\ref{main theorem 1}. The matter of subprincipal symbols (part (c) of Theorem~\ref{main theorem 1}) and the construction of our operators will be addressed afterwards.

\begin{proof}[Proof of Theorem~\ref{main theorem 1}, (a) and (b)]

Let $f_j$ be the orthonormalised eigenfunctions corresponding to the eigenvalues \eqref{eigenvalues of Laplace--Beltrami operator} of $\,-\Delta\,$. Put
\begin{equation}
\label{8 November 2022 equation 2}
v_j:=\mu_j^{-1/2}\dr f_j,\qquad j=1,2,\dots.
\end{equation}
The $v_j$ form an orthonormal basis in the Hilbert space $\dr\Omega^0(M)$ and the operator $P_0$ can be written as
\begin{equation}
\label{8 November 2022 equation 3}
P_0:=\sum_{j=1}^{+\infty}v_j\langle v_j\,,\,\cdot\,\rangle.
\end{equation}

Let
\begin{equation}
\label{8 November 2022 equation 4}
\Delta^{-1}
:=
-\sum_{j=1}^{+\infty}\mu_j^{-1}f_j\langle f_j\,,\,\cdot\,\rangle
\end{equation}
be the pseudoinverse of the Laplace--Beltrami operator $\Delta$.

Combining formulae \eqref{8 November 2022 equation 2}--\eqref{8 November 2022 equation 4} we get an explicit formula for the projection operator $P_0$
\begin{equation}
\label{8 November 2022 equation 5}
P_0=-\dr\,\Delta^{-1}\delta\,.
\end{equation}
In particular, formula \eqref{8 November 2022 equation 5} tells us that $P_0$ is a pseudodifferential operator of order zero and its principal symbol is given by formula \eqref{main theorem 1 equation 1}.

Consider the operator $\operatorname{curl}$ in the Hilbert space $\delta\Omega^2(M)$. Let $\lambda_j$ be its eigenvalues and $u_j$ its orthonormalised eigenforms. Here we enumerate using positive integers $j$ for positive eigenvalues and negative integers $j$ for negative eigenvalues, so that
\begin{equation*}
\label{6 November 2022 equation 1}
\ldots\le\lambda_{-2}\le\lambda_{-1}<0<\lambda_1\le\lambda_2\le\ldots,
\end{equation*}
with account of multiplicities. The $u_j$ form an orthonormal basis in the Hilbert space $\delta\Omega^2(M)$ and the operators $P_\pm$ can be written as
\begin{equation*}
\label{7 November 2022 equation 1}
P_+=\sum_{j=1}^{+\infty}u_j\langle u_j\,,\,\cdot\,\rangle,
\end{equation*}
\begin{equation*}
\label{7 November 2022 equation 2}
P_-=\sum_{j=-1}^{-\infty}u_j\langle u_j\,,\,\cdot\,\rangle.
\end{equation*}

We have
\begin{equation}
\label{formula for Pplus plus Pminus}
P_++P_-=\operatorname{Id}-P_0-P_{\mathcal{H}^1}\,,
\end{equation}
where $\operatorname{Id}$ is the identity operator on 1-forms and $P_{\mathcal{H}^1}$ is the orthogonal projection onto the finite-dimensional space of harmonic 1-forms. Of course, $P_{\mathcal{H}^1}$ is an integral operator with infinitely smooth integral kernel,
\begin{equation}
\label{PH is smooth}
P_{\mathcal{H}^1}\in\Psi^{-\infty}.
\end{equation}
%
%
%

Consider the Hodge Laplacian
\begin{equation}
\label{Hodge Laplacian}
\boldsymbol{\Delta}:=-(\dr\delta+\delta \dr)
=
-
\sum_{j=1}^{+\infty}\mu_j\,v_j\langle v_j\,,\,\cdot\,\rangle
-
\sum_{j\in\mathbb{Z}\setminus\{0\}}\lambda_j^2\,u_j\langle u_j\,,\,\cdot\,\rangle
\end{equation}
acting on 1-forms. This is an elliptic self-adjoint nonpositive differential operator which commutes with the differential operator $\curl$. Given an $s\in\mathbb{R}$, the operator
\begin{equation}
\label{Hodge Laplacian to the power s}
(-\boldsymbol{\Delta})^{s}
:=
\sum_{j=1}^{+\infty}\mu_j^{s}\,v_j\langle v_j\,,\,\cdot\,\rangle
+
\sum_{j\in\mathbb{Z}\setminus\{0\}}|\lambda_j|^{2s}\,u_j\langle u_j\,,\,\cdot\,\rangle
\end{equation}
is a pseudodifferential operator of order $2s$. It is easy to see that
\begin{equation}
\label{formula for Pplus minus Pminus}
P_+-P_-
=
(-\boldsymbol{\Delta})^{-1/2}
\,
\curl
=
(-\boldsymbol{\Delta})^{-1/4}
\,
\curl
\,
(-\boldsymbol{\Delta})^{-1/4}
=
\curl
\,
(-\boldsymbol{\Delta})^{-1/2}
\,.
\end{equation}

Equations \eqref{formula for Pplus plus Pminus} and \eqref{formula for Pplus minus Pminus} are a system of two linear algebraic equations for the two unknowns $P_+$ and $P_-\,$, whose unique solution is
\begin{equation}
\label{whose unique solution is}
P_\pm
=
\frac{1}{2}
\left[
\operatorname{Id}-P_0-P_{\mathcal{H}^1}
\pm
(-\boldsymbol{\Delta})^{-1/4}
\,
\curl
\,
(-\boldsymbol{\Delta})^{-1/4}
\right].
\end{equation}

We have
\begin{equation}
\label{principal symbol of the identity operator}
[\operatorname{Id}_\prin]_\alpha{}^\beta
=
\delta_\alpha{}^\beta\,,
\end{equation}
\begin{equation}
\label{principal symbol of the Hodge Laplacian}
[-\boldsymbol{\Delta}_\prin]_\alpha{}^\beta
=
\|\xi\|^{2}\,
\delta_\alpha{}^\beta\,.
\end{equation}
Formulae
\eqref{whose unique solution is}--\eqref{principal symbol of the Hodge Laplacian},
\eqref{main theorem 1 equation 1},
\eqref{PH is smooth}
and
\eqref{principal symbol curl}
imply
\eqref{main theorem 1 equation 2}.
\end{proof}

Recall that, as explained in Section~\ref{The operator curl}, $\curl_\prin$ has three simple eigenvalues $h^{(\aleph)}$, $\aleph\in\{0,+,-\}$, given by \eqref{eigenvalues principal symbol curl}. Let us denote by $P^{(\aleph)}$, $\aleph\in\{0,+,-\}$, the corresponding eigenprojections. It is easy to see that these are exactly the principal symbols \eqref{main theorem 1 equation 1}, \eqref{main theorem 1 equation 2} of our pseudodifferential operators $P_\aleph$, $\aleph \in\{0,+,-\}$.

Observe that the pseudodifferential operators $P_\aleph$, $\aleph \in\{0,+,-\}$, satisfy
\begin{equation}
\label{theorem about projections equation 2}
P_\aleph^*= P_\aleph,
\qquad
 P_\aleph P_\beth= \delta_{\aleph\beth}  P_\aleph\,,
\qquad
\sum_\aleph \ P_\aleph=\mathrm{Id}\textcolor{black}{\mod  \Psi^{-\infty}}\,,
\qquad
[\curl, P_\aleph]=0\,,
\end{equation}
where $\delta$ is the Kronecker delta.

Next, let us examine what happens if we relax \eqref{theorem about projections equation 2}, asking that they are satisfied not exactly, but only modulo $\Psi^{-\infty}$, cf.~\eqref{formula for Pplus plus Pminus} and~\eqref{PH is smooth}. Namely, let us seek operators $\tilde P_\aleph\in\Psi^0$, $\aleph\in\{0,+,-\}$, satisfying
\begin{subequations}
\label{theorem about projections tilde equation all}
\begin{equation}
\label{theorem about projections tilde equation 2}
(\tilde P_\aleph)^*=\tilde P_\aleph \mod \Psi^{-\infty},
\end{equation}
\begin{equation}
\label{theorem about projections tilde equation 3}
\tilde P_\aleph \tilde P_\beth= \delta_{\aleph\beth} \tilde P_\aleph\mod \Psi^{-\infty}\,,
\end{equation}
\begin{equation}
\label{theorem about projections tilde equation 4}
\sum_\aleph \tilde P_\aleph=\mathrm{Id}\mod \Psi^{-\infty}\,,
\end{equation}
\begin{equation}
\label{theorem about projections tilde equation 5}
[\curl, \tilde P_\aleph]=0 \mod \Psi^{-\infty}\,.
\end{equation}
\end{subequations}


\begin{theorem}
\label{theorem about projections tilde}
Let $\tilde P_0$, $\tilde P_+$ and $\tilde P_-$  be pseudodifferential operators of order zero, acting on 1-forms, satisfying
\begin{equation}
\label{theorem about projections tilde equation 1}
(\tilde P_\aleph)_\prin=P^{(\aleph)}\,, \qquad \aleph\in\{0,+,-\},
\end{equation}
and \eqref{theorem about projections tilde equation all}. Then
\begin{equation}
\label{theorem about Pj and tilde Pj equation 1}
\tilde P_\aleph=P_\aleph \mod \Psi^{-\infty}\,,\qquad\aleph\in\{0,+,-\}\,.
\end{equation}
\end{theorem}

\begin{proof}
The claim follows from a straightforward adaptation of \cite[Theorem~2.2]{part1} to the setting of the current paper. The original result from \cite{part1} was formulated under the assumption of ellipticity but ellipticity was never used in establishing the existence of projections. What was used is simplicity of eigenvalues of the principal symbol, which holds for $\curl$. Furthermore, the original result from \cite{part1} was formulated under the assumption that the operator acts on columns of half-densities, but this assumption was never used in an essential way; what was used is the presence of an inner product.
\end{proof}

In fact, we have a stronger version of Theorem~\ref{theorem about projections tilde}.

\begin{theorem}
\label{theorem about projections tilde stronger version}
Let $\tilde P_0$, $\tilde P_+$ and $\tilde P_-$  be pseudodifferential operators of order zero, acting on 1-forms, satisfying conditions
\eqref{theorem about projections tilde equation 5},
\eqref{theorem about projections tilde equation 1}
and
\begin{equation}
\label{theorem about projections tilde stronger version equation 1}
(\tilde P_\aleph)^2=\tilde P_\aleph\mod \Psi^{-\infty}
\,, \qquad \aleph\in\{0,+,-\}.
\end{equation}
Then we have \eqref{theorem about Pj and tilde Pj equation 1}.
\end{theorem}

What happens here is that
in Theorem~\ref{theorem about projections tilde}
assumptions
\eqref{theorem about projections tilde equation 2}
and
\eqref{theorem about projections tilde equation 4}
are redundant,
and
\eqref{theorem about projections tilde equation 3}
can be relaxed to
\eqref{theorem about projections tilde stronger version equation 1}.

\begin{proof}[Proof of Theorem~\ref{theorem about projections tilde stronger version}]
The proof is analogous to that of Theorem~\ref{theorem about projections tilde},
only additionally taking into account \cite[Theorem 4.1]{part1}.
\end{proof}

The overall strategy of our paper hinges on the key observation \eqref{theorem about Pj and tilde Pj equation 1}.
Indeed, this gives us an effective way of constructing the full symbols of our projections $P_\aleph$, $\aleph \in\{0,+,-\}$, without resorting explicitly to the representations \eqref{8 November 2022 equation 5} and \eqref{whose unique solution is}. 

Let introduce refined notation for the principal symbol. Namely, we denote by $(\,\cdot\,)_{\mathrm{prin},s}$ the principal symbol of the expression within brackets, regarded as an operator in $\Psi^{-s}$. To appreciate the need for such notation, consider the following example. Let $Q$ and $R$ be pseudodifferential operators in $\Psi^{-s}$ with the same principal symbol. Then, as an operator in $\Psi^{-s}$, $Q-R$ has vanishing principal symbol: $(Q-R)_{\mathrm{prin},s}=0$. But this tells us that $Q-R$ is, effectively, an operator in $\Psi^{-s-1}$ and, as such, it may have nonvanishing principal symbol $(Q-R)_{\mathrm{prin},s+1}$. This refined notation will be used whenever there is risk of confusion.

\begin{proposition}
\label{proposition algorighm Pj}
The operators $P_0$, $P_\pm$ can be explicitly constructed, modulo $\Psi^{-\infty}$, as follows.
\begin{itemize}
\item {\bf Step 1}. Choose three arbitrary pseudodifferential operators $P_{\aleph,0}\in \Psi^0$, $\aleph\in\{0,+,-\}$, such that $(P_{\aleph,0})_\prin=P^{(\aleph)}$.

\item {\bf Step 2}. For $k=1,2,\ldots$ define
\begin{equation*}
\label{algorithm equation 1}
P_{\aleph,k}:=P_{\aleph,0}+\sum_{n=1}^k X_{\aleph,n}, \qquad X_{\aleph,n}\in \Psi^{-n}.
\end{equation*}

Assuming we have determined the pseudodifferential operator $P_{\aleph,k-1}$, compute, one after the other, the following quantities:
\begin{enumerate}[(a)]
\item 
$
R_{\aleph,k}:=-((P_{\aleph,k-1})^2-P_{\aleph,k-1})_\mathrm{prin,k}\,,
$

\item 
$
S_{\aleph,k}:=-R_{\aleph,k}+P^{(\aleph)} R_{\aleph,k}+R_{\aleph,k}P^{(\aleph)}\,,
$

\item
$
T_{\aleph,k}:=[P_{\aleph,k-1},\curl]_{\mathrm{prin},k-1}+[S_{\aleph,k},\curl_\mathrm{prin}].
$
\end{enumerate}

\item \textbf{Step 3}. Choose a pseudodifferential operator $X_{\aleph,k}\in \Psi^{-k}$ satisfying
$$
\displaystyle
(X_{\aleph,k})_\mathrm{prin}=S_{\aleph,k}+\sum_{\beth\ne \aleph}\frac{P^{(\aleph)} T_{\aleph,k}P^{(\beth)}-P^{(\beth)} T_{\aleph,k}P^{(\aleph)}}{h^{(\aleph)}-h^{(\beth)}}\,.
$$
\textcolor{black}{Recall that $h^{(\,\cdot\,)}$ is defined in accordance with~\eqref{eigenvalues principal symbol curl}  and~\eqref{norm of xi}.}
\item {\bf Step 4}.
Then
\begin{equation}
\label{proposition algorighm Pj equation 1}
P_\aleph \sim P_{\aleph,0}+\sum_{n=1}^{+\infty}X_{\aleph,n}\,.
\end{equation}
Here $\sim$ stands for an asymptotic expansion in `smoothness': truncated sums give an approximation modulo pseudodifferential operators of lower order.
\end{itemize}
\end{proposition}

\begin{proof}
The algorithm presented here is a straightforward adaptation of that from \cite[subsection~4.3]{part1}, which can be applied to the operator $\curl$ in view of Theorem~\ref{theorem about projections tilde}.
\end{proof}

\begin{remark}
\label{remark about two algorithms}
The paper \cite{part1} offers two different algorithms for the construction of an orthonormal basis of pseudodifferential projections.
The first algorithm, using the full set of conditions from Theorem~\ref{theorem about projections tilde}, is given in \cite[subsection 3.4]{part1}.
The second algorithm, using the minimal set of conditions from Theorem~\ref{theorem about projections tilde stronger version}, is given in \cite[subsection 4.3]{part1}.
Of course, the end result is the same, modulo $\Psi^{-\infty}$,
but the latter is much simpler.
\end{remark}

\begin{remark}
Let us stress once again that the algorithm from Proposition~\ref{proposition algorighm Pj} does not require ellipticity of the original operator, but it \emph{does} require the eigenvalues of the principal symbol to be simple --- see Step 3.
\end{remark}

We are now in a position to prove part (c) of Theorem~\ref{main theorem 1}, which will be quite useful in the analysis of the asymmetry operator.

\begin{proof}[Proof of Theorem~\ref{main theorem 1}(c)]
For starters, let us observe that the claim \eqref{main theorem 1 equation 4} is \emph{invariant} under changes of local coordinates, because the $(P_\aleph)_\sub$ are covariant under such transformations. Therefore, it suffices to fix a point $z\in M$, choose normal coordinates at $z$, and check that \eqref{main theorem 1 equation 4} is satisfied at $(z,\zeta)\in T^*M\setminus\{0\}$ for all $\zeta$, in the chosen coordinate system.

Indeed, let us fix a point $z$ and work in normal coordinates centred at $z$. Let us choose arbitrary pseudodifferential operators $P_{\aleph,0}\in \Psi^0$, $\aleph\in\{0,+,-\}$, satisfying $(P_{\aleph,0})_\prin=(P_\aleph)_\prin$ (recall that the latter are given by \eqref{main theorem 1 equation 1} and \eqref{main theorem 1 equation 2}) and
\begin{equation}
\label{Proof of main theorem 1 equation 1}
(P_{\aleph,0})_\sub=0.
\end{equation}
Then it is easy to see from the algorithm in Proposition~\ref{proposition algorighm Pj} --- in particular, from formula \eqref{proposition algorighm Pj equation 1} --- that 
\begin{equation}
\label{Proof of main theorem 1 equation 2}
(P_\aleph)_\sub=(X_{\aleph,1})_\prin.
\end{equation}
So proving \eqref{main theorem 1 equation 4} reduces to proving that
\begin{equation}
\label{Proof of main theorem 1 equation 3}
[(X_{\aleph,1})_\prin(z,\zeta)]_\alpha{}^\beta=0
\end{equation}
in the chosen coordinate system. Let us compute, one after the other, the quantities from Step 2 in Proposition~\ref{proposition algorighm Pj}. 

Using \eqref{theorem about subprincipal symbol of composition equation 1} and \eqref{Proof of main theorem 1 equation 1} we get
\begin{equation}
\label{Proof of main theorem 1 equation 4}
[R_{\aleph,1}]_\alpha{}^\beta(z,\zeta)=-[(P_\aleph^2)_\sub]_\alpha{}^\beta(z,\zeta)
=
-\frac{i}2\{\{(P_\aleph)_\prin\,, (P_\aleph)_\prin\}\}_\alpha{}^\beta(z,\zeta)=0\,.
\end{equation}
In the last equality we used the fact that first spatial derivatives of $(P_\aleph)_\prin$ and Christoffel symbols vanish at $(z,\zeta)$ in normal coordinates. 

Formula \eqref{Proof of main theorem 1 equation 4} immediately implies $[S_{\aleph,k}]_\alpha{}^\beta(z,\zeta)=0$. Hence, on account of Lemma~\ref{lemma subprincipal symbol curl}, we have
\begin{equation*}
\label{Proof of main theorem 1 equation 5}
[T_{\aleph,1}]_\alpha{}^\beta(z,\zeta)=\frac{i}2 \{\{(P_\aleph)_\prin\,,\curl_\prin\}\}(z,\zeta)- \frac{i}2\{\{\curl_\prin\,,(P_\aleph)_\prin\}\}(z,\zeta)\,.
\end{equation*}
But the latter also vanishes in normal coordinates, so one arrives at \eqref{Proof of main theorem 1 equation 3}.
\end{proof}

\begin{remark}
\label{explicit formulae for the homogeneous components}
Observe that
\eqref{main theorem 1 equation 1}--\eqref{main theorem 1 equation 4}
and
\eqref{definition subprincipal symbol operators 1 forms}
give us explicit formulae for the homogeneous components of degree $-1$ of the symbols of the operators $P_0$ and $P_\pm$ in arbitrary local coordinates.
\end{remark}

\section{The asymmetry operator}
\label{The asymmetry operator}

\subsection{Definition of the asymmetry operator}
\label{Definition of the asymmetry operator subsection}

The remainder of this paper is devoted to the study of the pseudodifferential operator $P_+-P_-$ which encodes information about the asymmetry of the Riemannian manifold $(M,g)$ under change of orientation. In particular, one expects this information to be carried by the trace of $P_+-P_-\,$, either in its pointwise or global (operator-theoretic) versions.

Unfortunately, one encounters a problem: the pseudodifferential operator $P_+-P_-$ is of order zero and its spectrum is the set of three points $\{-1,0,+1\}$, all three being eigenvalues of infinite multiplicity. Hence, it is not trace class. In order to circumvent this problem, acting in the spirit of Section~\ref{Trace of pseudodifferential operators acting on 1-forms}, we introduce the following definition.

\begin{definition}
\label{definition asymmetry operator}
We define the \emph{asymmetry operator} to be the (scalar) pseudodifferential operator
\begin{equation}
\label{definition asymmetry operator formula}
A:=\operatorname{\mathfrak{tr}} (P_+-P_-)\,,
\end{equation}
where $\operatorname{\mathfrak{tr}}$ is defined in accordance with Definition~\ref{definition of matrix trace of a pdo acting on 1-forms}.
\end{definition}

The above definition warrants a number of remarks.

\begin{remark}
\textcolor{white}{.}
\begin{enumerate}[(a)]
\item
The asymmetry operator depends, \emph{a priori}, on the choice of the parameter  $\epsilon$ and the cut-off $\chi$. However, we suppress the dependence on $\epsilon$ and $\chi$ in our notation for the asymmetry operator because, as we shall see later on, our main results do not depend on these choices. The dependence on $\epsilon$ and $\chi$ is, in a sense, trivial: two operators $A$ corresponding to different choices of $\epsilon$ and $\chi$ differ by an integral operator with infinitely smooth integral kernel which vanishes in a neighbourhood of the diagonal in $M\times M$. The trace of the latter is clearly zero.
\item
The Schwartz kernel of $P_+-P_-$ is real, hence the Schwartz kernel of $A$ is also real.
\item
The operator $P_+-P_-$ is self-adjoint, hence, by Corollary~\ref{self-adjointness of matrix trace} the asymmetry operator $A$ is also self-adjoint.
\end{enumerate}
\end{remark}

\subsection{The order of the asymmetry operator}
\label{The order of the asymmetry operator subsection}

Definition \ref{definition asymmetry operator} implies that, \emph{a priori}, our asymmetry operator $A$ is a pseudodifferential operator of order 0. The goal of this subsection is to show that, as a result of cancellations in the symbol, the order of $A$ is in fact $-3$.  This will be done in three steps.
\begin{enumerate}
\item 
First, we will show that $A$ is of order $-2$. This will follow easily from the facts already established earlier in the paper.

\item 
Next, we will show that, in normal coordinates, parallel transport does not contribute to homogeneous components of the symbol of $A$ of degree $-2$ and $-3$.

\item
Finally, we will show that $A$ is of order $-3$. This is the subject of Theorem~\ref{Theorem A order -3} which will come in the end of the subsection and is one of the main results of this paper.
\end{enumerate}

\begin{lemma}
\label{Lemma A order -2}
The asymmetry operator $A$ is a pseudodifferential operator of order $-2$.
\end{lemma}

\begin{proof}
The claim of the lemma is equivalent to the following two statements:
\begin{equation}
\label{Aprin is zero}
A_\prin=0\,,
\end{equation}
\begin{equation}
\label{Asub is zero}
A_\sub=0\,.
\end{equation}
Formula \eqref{Aprin is zero} follows from \eqref{definition asymmetry operator formula}, \eqref{properties of matrix trace of a pdo acting on 1-forms lemma equation 1} and~\eqref{main theorem 1 equation 2}, whereas formula \eqref{Asub is zero} follows from \eqref{definition asymmetry operator formula}, \eqref{properties of matrix trace of a pdo acting on 1-forms lemma equation 2} and~\eqref{main theorem 1 equation 4}.
\end{proof}

Let us now fix an arbitrary point $z\in M$ and work in normal coordinates centred at $z$. In our chosen coordinate system the operator $A$ (modulo $\Psi^{-\infty}$) reads
\begin{equation}
\label{A in normal coordinates}
A: f(x) \mapsto
\frac{1}{(2\pi)^3}
\int e^{i(x-y)^\mu\xi_\mu}\,
[p_+-p_-]_{\alpha}{}^\beta(x,\xi)
\,
Z_\beta{}^\alpha(y,x)
\,
f(y)
\,\dr y\,\dr\xi\,,
\end{equation}
where
\begin{equation}
[p_\pm]_\alpha{}^\beta(x,\xi)\sim \sum_{j=0}^{+\infty} [(p_\pm)_{-j}]_\alpha{}^\beta(x,\xi), \qquad [(p_\pm)_{-k}]_\alpha{}^\beta(x,\lambda\xi)=\lambda^{-k}[(p_\pm)_{-k}]_\alpha{}^\beta(x,\xi) \quad \forall\lambda>0,
\end{equation}
is the full symbol of $P_\pm$ in the chosen coordinate system.

Next, we observe that \eqref{A in normal coordinates} can be recast as
\begin{equation}
\label{A in normal coordinates v2}
A=A_\mathrm{diag}+A_\mathrm{pt}\,,
\end{equation}
where 
\begin{equation}
\label{A diag}
A_\mathrm{diag}:=\frac{1}{(2\pi)^3}
\int e^{i(x-y)^\mu\xi_\mu}\,
[p_+-p_-]_{\alpha}{}^\alpha(x,\xi)
\,
f(y)
\,\dr y\,\dr\xi
\end{equation}
and 
\begin{equation}
\label{A pt}
A_\mathrm{pt}:=A-A_\mathrm{diag}.
\end{equation}
Here the subscript ``pt'' stands for ``parallel transport''. The operator $A_\mathrm{pt}$ describes the contribution to the asymmetry operator arising from the fact that our original operator $\curl$ acts on 1-forms and not on 3-columns of scalar fields,

The decomposition \eqref{A in normal coordinates v2} is not invariant, it relies on our particular choice of local coordinates but it is convenient for our purposes. Namely, in what follows we will show that $A_\mathrm{pt}$ will not contribute to the principal symbol of order $-3$ of the asymmetry operator $A$ at the point $z$.

Let us denote by $a_\mathrm{pt}(x,\xi)\sim \sum_{j=0}^{+\infty} (a_\mathrm{pt})_{-j}(x,\xi)$ the full symbol of $A_\mathrm{pt}\,$. Our task is to show that
\begin{equation}
\label{components of apt are zero}
(a_\mathrm{pt})_{-j}(z,\xi)=0 \qquad \text{for} \ j=0,1,2,3.
\end{equation}
In what follows, we will be using the expansion \eqref{prop: expansion of parallel transoport equation 3 tau zero} for $Z_{\alpha}{}^\beta$.

The claim \eqref{components of apt are zero} for $j=0$ and $j=1$ follows immediately from \eqref{prop: expansion of parallel transoport equation 3 tau zero}.

\begin{lemma}
\label{lemma apt -2 is zero}
We have
\begin{equation*}
\label{lemma apt -2 is zero equation}
(a_\mathrm{pt})_{-2}(z,\xi)=0.
\end{equation*}
\end{lemma}

\begin{proof}
Upon integration by parts, we get from \eqref{prop: expansion of parallel transoport equation 3 tau zero}
\begin{equation}
\label{apt -2 equation 1}
(a_\mathrm{pt})_{-2}(z,\xi)=\frac16\, \Riem^{\alpha}{}_{\mu\kappa\nu}(z)\,\dfrac{\partial^2[(p_+-p_-)_{0}]_{\alpha}{}^\kappa(z,\xi)}{\partial \xi_\mu \partial \xi_\nu}
=\frac16 \,\Riem_{\alpha\mu\kappa\nu}(z)\,\dfrac{\partial^2[i\,
|\xi|^{-1}\,
\varepsilon^{\alpha\gamma\kappa}\,
\xi_\gamma]}{\partial \xi_\mu \partial \xi_\nu}\,,
\end{equation}
where $|\cdot|$ is the Euclidean norm. By the symmetries of the Riemann tensor we have
\begin{equation}
\label{riemann1}
\Riem_{\alpha\mu\kappa\nu}=\frac12(\textcolor{black}{\Riem}_{\alpha\mu\kappa\nu}+\Riem_{\kappa\nu\alpha\mu}).
\end{equation}
Since 
\begin{equation}
\label{riemann1a}
\frac{\partial^2[i\,
|\xi|^{-1}\,
\varepsilon^{\alpha\gamma\kappa}\,
\xi_\gamma]}{\partial \xi_\mu \partial \xi_\nu}
\end{equation}
is symmetric in the pair of indices $\mu$ and $\nu$, we can replace \eqref{riemann1} in \eqref{apt -2 equation 1} with its symmetrised version
\begin{equation}
\label{riemann2}
\Riem_{\alpha\mu\kappa\nu}=\frac12(\Riem_{\alpha\mu\kappa\nu}+\Riem_{\kappa\nu\alpha\mu}+\Riem_{\alpha\nu\kappa\mu}+\Riem_{\kappa\mu\alpha\nu}).
\end{equation}
But the quantity \eqref{riemann2} is symmetric in the pair of indices $\alpha$ and $\kappa$, whereas \eqref{riemann1a} is antisymmetric in the same pair of indices. Therefore \eqref{apt -2 equation 1} vanishes.
\end{proof}

Next, we observe that formulae \eqref{main theorem 1 equation 4}, \eqref{subprincipal symbol operators 1-forms} and \eqref{main theorem 1 equation 2} imply
\begin{equation}
\label{ppm minus 1 is zero}
[(p_+-p_-)_{-1}]_\alpha{}^\beta(z,\xi)=0.
\end{equation}
This fact is needed in establishing the following lemma.

\begin{lemma}
\label{lemma apt -3 is zero}
We have
\begin{equation*}
\label{lemma apt -3 is zero equation}
(a_\mathrm{pt})_{-3}(z,\xi)=0.
\end{equation*}
\end{lemma}

\begin{proof}
The expansion \eqref{prop: expansion of parallel transoport equation 3 tau zero} and formula \eqref{ppm minus 1 is zero} imply, upon integration by parts,
\begin{multline}
\label{apt -3 equation 1}
(a_\mathrm{pt})_{-3}(z,\xi)=-\frac{i}{6}\, \dfrac{\partial^2\Gamma^\alpha{}_{\sigma\kappa}}{\partial x^\mu \partial x^\nu}(z)\,\dfrac{\partial^3[(p_+-p_-)_{0}]_{\alpha}{}^\kappa(z,\xi)}{\partial \xi_\sigma\partial \xi_\mu \partial \xi_\nu}\\
=
-\frac{i}{6} \,\delta_{\alpha\beta}\,\dfrac{\partial^2\Gamma^\beta{}_{\sigma\kappa}}{\partial x^\mu \partial x^\nu}(z)\,\dfrac{\partial^3[i\,
|\xi|^{-1}\,
\varepsilon^{\alpha\gamma\kappa}\,
\xi_\gamma]}{\partial \xi_\sigma\partial \xi_\mu \partial \xi_\nu}\,.
\end{multline}

We observe that in formula \eqref{apt -3 equation 1} one can replace the quantity $\delta_{\alpha\beta}\,\dfrac{\partial^2\Gamma^\beta{}_{\sigma\kappa}}{\partial x^\mu \partial x^\nu}(z)$ with its symmetrised version in the three indices $\sigma$, $\mu$ and $\nu$.
It is known \cite[Eqn.~(11.9)]{expansions} that the latter is equal to
\begin{equation}
\label{symmetrised christoffel xx}
\frac12 \nabla_\nu \Riem_{\alpha \sigma \mu \kappa}(z)\,.
\end{equation}
Using the symmetries of the Riemann tensor, \eqref{symmetrised christoffel xx} can be equivalently rewritten as
\begin{equation}
\label{symmetrised christoffel yy}
\frac12 \nabla_\nu \Riem_{\alpha \sigma \mu \kappa}
=
\frac14 [\nabla_\nu \Riem_{\alpha \sigma \mu \kappa}+\nabla_\nu \Riem_{\mu \kappa \alpha \sigma}].
\end{equation}
Now, when contracted with
\begin{equation}
\label{apt -3 equation 2}
\dfrac{\partial^3[i\,
|\xi|^{-1}\,
\varepsilon^{\alpha\gamma\kappa}\,
\xi_\gamma]}{\partial \xi_\sigma\partial \xi_\mu \partial \xi_\nu},
\end{equation}
the quantity \eqref{symmetrised christoffel yy} can be replaced with its symmetrised version in the pair of indices $\mu$ and $\sigma$:

\begin{equation}
\label{apt -3 equation 3}
\frac18 [\nabla_\nu \Riem_{\alpha \sigma \mu \kappa}+\nabla_\nu \Riem_{\mu \kappa \alpha \sigma}
+
\nabla_\nu \Riem_{\alpha \mu \sigma \kappa}+\nabla_\nu \Riem_{\sigma \kappa \alpha \mu}]\,.
\end{equation}
But the quantity \eqref{apt -3 equation 3} is now symmetric in the pair of indices $\alpha$ and $\kappa$, whereas \eqref{apt -3 equation 2} is anti\-symmetric in the same pair of indices. Therefore \eqref{apt -3 equation 1} vanishes.
\end{proof}

Lemmata~\ref{lemma apt -2 is zero} and~\ref{lemma apt -3 is zero} give us \eqref{components of apt are zero}.

\

We are now in a position to prove the main result of this subsection.

\begin{theorem}
\label{Theorem A order -3}
The asymmetry operator $A$ is a pseudodifferential operator of order $-3$.
\end{theorem}

\begin{remark}
\label{remark common sense 1}
Before rigorously proving Theorem~\ref{Theorem A order -3}, let us give an intuitive explanation of why the claim should be true.

The principal symbol $A_{\prin,2}$ (that is, the principal symbol of the asymmetry operator $A$ viewed as an operator of order $-2$) should be proportional to
\begin{itemize}
\item Ricci curvature --- because $a_{-2}$ involves two derivatives of the metric and because in dimension three the Riemann curvature tensor can be expressed in terms of the Ricci tensor;
\item
the totally antisymmetric tensor $E$ --- because the asymmetry operator $A$ is expected to feel orientation.
\end{itemize}
If one attempts to construct a scalar quantity out of the above two objects, the metric, and momentum $\xi$, one realises that the only possible outcome is the the scalar zero.
\end{remark}

\begin{proof}[Proof of Theorem~\ref{Theorem A order -3}]
Let us denote by $a_\mathrm{diag}$ the full symbol of $A_\mathrm{diag}\,$. Formulae~\eqref{main theorem 1 equation 2} and \eqref{ppm minus 1 is zero} immediately imply
\begin{equation}
[(a_\mathrm{diag})_{0}](z,\xi)=0
\end{equation}
and
\begin{equation}
[(a_\mathrm{diag})_{-1}](z,\xi)=0,
\end{equation}
respectively.

Hence, in view of \eqref{A in normal coordinates v2} and \eqref{components of apt are zero}, proving the claim reduces to showing that
\begin{equation}
\label{proof Theorem A order -3 equation 1}
[(a_\mathrm{diag})_{-2}](z,\xi)=[(p_+-p_-)_{-2}]_\alpha{}^\alpha(z,\xi)=0.
\end{equation}

Our strategy for proving \eqref{proof Theorem A order -3 equation 1} is to check it explicitly by implementing the algorithm from Proposition~\ref{proposition algorighm Pj} in chosen normal coordinates. As the latter is computationally challenging, we describe below how to do so in a cunning way.

In what follows $z=0$ (the centre of our normal coordinates).

\

\textbf{Simplification 1}. It is well known that
\begin{equation}
\label{14 September 2021 equation 1}
g_{\alpha\beta}(x)
=
\delta_{\alpha\beta}
-
\frac13
\Riem_{\alpha\mu\beta\nu}(0)\,x^\mu x^\nu
+O(|x|^3).
\end{equation}
The approximation \eqref{14 September 2021 equation 1} is sufficiently accurate for the computation of \eqref{proof Theorem A order -3 equation 1}.

\

\textbf{Simplification 2}. Express the Riemann tensor $\Riem$ in terms of the Ricci tensor $\Ric$ via the identity
\begin{multline}
\label{14 September 2021 equation 5}
\Riem_{\alpha\beta\gamma\delta}(x)
=
\Ric_{\alpha\gamma}(x)\,g_{\beta\delta}(x)
-
\Ric_{\alpha\delta}(x)\,g_{\beta\gamma}(x)
+
\Ric_{\beta\delta}(x)\,g_{\alpha\gamma}(x)
-
\Ric_{\beta\gamma}(x)\,g_{\alpha\delta}(x)
\\
+
\frac{\Sc(x)}{2}
(
g_{\alpha\delta}(x)g_{\beta\gamma}(x)
-
g_{\alpha\gamma}(x)g_{\beta\delta}(x)
)\,,
\end{multline}
which holds in dimension three, where $\Sc(x):=g^{\alpha\beta}(x)\Ric_{\alpha\beta}(x)$ is scalar curvature. Then substitute \eqref{14 September 2021 equation 5} into \eqref{14 September 2021 equation 1} and discard cubic and higher order terms.

\

\textbf{Simplification 3}. It is clear that $\Ric$ will appear in $[(a_\mathrm{diag})_{-2}](z,\xi)=[(p_+-p_-)_{-2}]_\alpha{}^\alpha(0,\xi)$ in a linear fashion. This means that we can set
\begin{equation}
\label{Ric at 0 simplification 3}
\Ric(0)=
\begin{pmatrix}
c_1 & c_2 & c_3\\
c_2 & c_4 & c_5\\
c_3 & c_5 & c_6
\end{pmatrix}
\end{equation}
and do calculations assuming that only one of the constants $c_j$, $j=1,\dots,6$, is nonzero. In particular, one can choose the nonzero $c_j$ to be equal to 1. To establish \eqref{proof Theorem A order -3 equation 1} one just needs to run the algorithm six times.

\

\textbf{Simplification 4}. Without loss of generality, evaluate the final quantities (as soon as there are no more differentiations with respect to momentum) at
\begin{equation}
\label{xi0}
\xi_0=\begin{pmatrix}
0
\\
0
\\
1
\end{pmatrix}.
\end{equation}
This is sufficient because
\begin{itemize}
\item
one can rotate the normal coordinate system so that $\xi$ aligns with the third spatial coordinate and
\item
one can scale $\xi$ so as to normalise it, using the homogeneity of $[(a_\mathrm{diag})_{-2}](0,\xi)$.
\end{itemize}
Moreover, one can write
$
\xi=\xi_0+\eta
$
and replace $\|\xi\|$ with its Taylor expansion in $\eta$. The latter simplifies differentiations in $\xi$.

\

\textbf{Simplification 5}. In implementing the algorithm, expand all quantities at the $k$-th step in powers of $x$ and $\eta$ retaining only terms up to order $2-k$, jointly in $x$ and $\eta$. Here we are assuming that $x$ and $\eta$ are of the same order.

\

Simplifications 1--5 reduce our algorithm to the manipulation of polynomials. Note that the vector spaces of polynomials in the variables $x^\alpha$ and $\eta_\alpha\,$, $\alpha=1,2,3$, of degree 0, 1 and 2 have dimension 1, 7 and 28 respectively\footnote{The dimension of the vector space of polynomials of degree $k$ in $n$ variables is
$
{{k+n}\choose{k}}\,.
$}.

\

We implement below the algorithm in the special case 
\begin{equation}
\label{case c1=1}
c_j=\delta_{1j}.
\end{equation} 
The other cases are analogous.

Formulae \eqref{14 September 2021 equation 1}--\eqref{Ric at 0 simplification 3} and \eqref{case c1=1} imply that in the chosen coordinate system the metric tensor reads
\begin{equation*}
\label{proof Theorem A order -3 equation 2}
g_{\alpha\beta}(x)
=
\begin{pmatrix}
1-\frac16(x^2)^2-\frac16(x^3)^2&\frac16x^1x^2&\frac16x^1x^3\\
\frac16x^1x^2&1-\frac16(x^1)^2+\frac16(x^3)^2&-\frac16x^2x^3\\
\frac16x^1x^3&-\frac16x^2x^3&1-\frac16(x^1)^2+\frac16(x^2)^2
\end{pmatrix}
+O(|x|^3)
\end{equation*}
and
\begin{equation*}
\label{proof Theorem A order -3 equation 3}
\rho(x)=1-\frac16(x^1)^2+O(|x|^3).
\end{equation*}

Let us now retrace, step by step, the algorithm from Proposition~\ref{proposition algorighm Pj}. In view of Lemma~\ref{Lemma A order -2}, let us assume that we have carried out the first iteration of the algorithm and have determined pseudo\-differential operators $P_{\pm,1}\in \Psi^0$ satisfying
\begin{equation*}
\label{proof Theorem A order -3 equation 4}
P_{\pm,1}^2=P_{\pm,1} \mod \Psi^{-2}, \qquad P_{+,1}P_{-,1}=0\mod \Psi^{-2},\qquad [P_{\pm,1},\curl]=0 \mod \Psi^{-1}\,.
\end{equation*} 
This gives us
\begin{equation*}
\label{proof Theorem A order -3 equation 8}
(p_{\pm,1})_{-1}(x,\xi_0+\eta)=-\frac1{12}\left(
\begin{array}{ccc}
 i x^3 &\pm x^3 & 0 \\
 \pm x^3& -i x^3 & 0 \\
\pm x^2-3 i x^1 & \pm x^1+i x^2& 0 \\
\end{array}
\right)
+O(|x|^2+|\eta|^2)\,.
\end{equation*}

Without loss of generality, we can assume that
\begin{equation}
\label{proof Theorem A order -3 equation 5}
(p_{\pm,1})_{-2}=0.
\end{equation}
The latter can be achieved by choosing $X_{\pm,1}$ appropriately. Therefore, proving the theorem reduces to showing that
\begin{equation}
\label{proof Theorem A order -3 equation 6}
[\operatorname{tr}(X_{+,2}-X_{-,2})_\prin](0,\xi_0)=0.
\end{equation}

The first quantity we need to compute is $R_{\pm,2}$. Due to \eqref{proof Theorem A order -3 equation 5} we have
\begin{equation*}
\label{proof Theorem A order -3 equation 7}
R_{\pm,2}=-(P_{\pm,1}^2)_{-2}
\end{equation*}
where $(P_{\pm,1}^2)_{-2}$ stands for the homogeneous component of degree $-2$ of the symbol of the operator $P_{\pm,1}^2\,$.
To compute the latter, one needs to use the formula for the symbol of the composition of pseudodifferential operators: given two operators $A$ and $B$ with symbols $a$ and $b$, the symbol $\sigma_{BA}$ of their composition $BA$ is given by the formula

\begin{equation}
\label{5 November 2021 equation 1}
\sigma_{BA}\sim\sum_{k=0}^{\infty} \frac{1}{i^k k!}  \dfrac{\partial^k b}{\partial \xi_{\alpha_1}\dots \partial \xi_{\alpha_k}}\dfrac{\partial^k a}{\partial x^{\alpha_1}\dots \partial x^{\alpha_k}}\,,
\end{equation}
see \cite[Theorem~3.4]{shubin}.

By performing the relevant calculations with the simplifications described above, one obtains
\begin{equation*}
\label{proof Theorem A order -3 equation 9}
R_{\pm,2}(0,\xi_0)=
-\frac1{12}
\begin{pmatrix}
1 & \mp 2i &0\\
\pm 2i& 1 & 0\\
0 & 0 & 0
\end{pmatrix}\,,
\end{equation*}
\begin{equation*}
\label{proof Theorem A order -3 equation 10}
S_{\pm,2}(0,\xi_0)=
-\frac1{12}
\begin{pmatrix}
2 & \mp i &0\\
\pm i& 2 & 0\\
0 & 0 & 0
\end{pmatrix}\,,
\end{equation*}
\begin{equation*}
\label{proof Theorem A order -3 equation 11}
T_{\pm,2}(0,\xi_0)=
\pm\frac1{12}
\begin{pmatrix}
1 & 0 &0\\
0 & -1 & 0\\
0 & 0 & 0
\end{pmatrix}
\,
\end{equation*}
and, finally,
\begin{equation}
\label{proof Theorem A order -3 equation 12}
(X_{\pm,2})_\prin(0,\xi_0)=
-\frac1{24}
\begin{pmatrix}
4 & \mp 3i &0\\
\pm i & 4 & 0\\
0 & 0 & 0
\end{pmatrix}
\,.
\end{equation}
Formula \eqref{proof Theorem A order -3 equation 12} implies \eqref{proof Theorem A order -3 equation 6}.
\end{proof}

\subsection{The principal symbol of the asymmetry operator}
\label{The principal symbol of the asymmetry operator}

This subsection is concerned with the proof of Theorem~\ref{main theorem 2}. As in the previous subsection, in what follows we also assume to have chosen normal coordinates centred at $z$.

\begin{remark}
\label{remark common sense 2}
Before moving on to the proof, let us try to reproduce the intuitive argument from Remark~\ref{remark common sense 1} at the next order of homogeneity.

The principal symbol $A_{\prin,3}$ (that is, the principal symbol of the asymmetry operator $A$ viewed
as an operator of order $-3$) should be proportional to
\begin{itemize}
\item the covariant derivative of Ricci curvature --- because $a_{-3}$ involves three derivatives of the metric and because in dimension three the Riemann curvature tensor can be expressed in terms of the Ricci tensor;

\item 
the totally antisymmetric tensor $E$ --- because the operator $A$ is expected to feel orientation.
\end{itemize}
If one attempts to construct a scalar quantity out of the above two objects, the metric, and momentum $\xi$, one realises that the only possible nonzero outcome, up to a constant scaling, is
\[
E^{\alpha\beta \gamma}(x)\, \nabla_\alpha \operatorname{Ric}_{\beta}{}^\rho(x)\, \frac{\xi_\gamma\xi_\rho}{\|\xi\|^5}\,.
\]
\end{remark}

\begin{proof}[Proof of Theorem~\ref{main theorem 2}]
In view of Lemma~\ref{lemma apt -3 is zero}, proving \eqref{main theorem 2 equation 1} reduces to showing that, in chosen normal coordinates, we have
\begin{equation}
\label{Proof of Aprin equation 1}
[(p_+-p_-)_{-3}]_{\alpha}{}^\alpha(z,\xi)=-\frac1{2|\xi|^5}\varepsilon^{\alpha\beta \gamma}\, \nabla_\alpha \Ric_{\beta}{}^\rho(z)\, \xi_\gamma\xi_\rho\,.
\end{equation}

The strategy for proving \eqref{Proof of Aprin equation 1} consists in implementing the algorithm from Proposition~\ref{proposition algorighm Pj} in chosen normal coordinates up to order $-3$.
Even with the help of modern computer algebra, doing this to the required accuracy is a highly nontrivial task. Therefore, as in the proof of Theorem~\ref{Theorem A order -3}, we describe below some helpful simplifications, which once again reduce all calculations to the manipulation of polynomials.

In what follows $z=0$ (the centre of our normal coordinates).

\

\textbf{Simplification 1}.
It is known, see \cite[p.~211 formula (3.4)]{Schoen and Yau} or \cite[formula (6)]{Uwe Muller}, that
\begin{equation}
\label{14 September 2021 equation 1bis}
g_{\alpha\beta}(x)
=
\delta_{\alpha\beta}
-
\frac13
\Riem_{\alpha\mu\beta\nu}(0)\,x^\mu\,x^\nu
-
\frac16
(\nabla_\sigma \Riem_{\alpha\mu\beta\nu})(0)\,x^\sigma\,x^\mu\,x^\nu
+O(|x|^4).
\end{equation}
The approximation \eqref{14 September 2021 equation 1bis} is sufficiently accurate for the computation of \eqref{Proof of Aprin equation 1}.

\

\textbf{Simplification 2}.
Express the Riemann tensor $\Riem$ in terms of the Ricci tensor $\Ric$ via the identity \eqref{14 September 2021 equation 5}. Then substitute \eqref{14 September 2021 equation 5} into \eqref{14 September 2021 equation 1bis} and discard quartic and higher order terms.

\

\textbf{Simplification 3}. It is clear that $\Ric$ and $\nabla\Ric$ will appear in $(p_{\pm})_{-3}$ in a linear fashion. We can choose \eqref{Ric at 0 simplification 3} and
\begin{multline}
\label{nablaG at 0 simplification 3}
\nabla_1\Ric_{\alpha\beta}(0)
=
\begin{pmatrix}
c_7&c_8&c_9\\
c_8&c_{10}&c_{11}\\
c_9&c_{11}&c_{12}
\end{pmatrix},
\qquad
\nabla_2\Ric_{\alpha\beta}(0)
=
\begin{pmatrix}
c_{13}&c_{14}&c_{15}\\
c_{14}&c_{16}&c_{17}\\
c_{15}&c_{17}&c_{18}
\end{pmatrix},
\\
\nabla_3\Ric_{\alpha\beta}(0)
=
\begin{pmatrix}
c_{19}&c_{20}&c_{21}\\
c_{20}&c_{22}&c_{23}\\
c_{21}&c_{23}&c_{24}
\end{pmatrix},
\end{multline}
and do calculations assuming that only one of the constants $c_j$, $j=1,\dots,24$, is nonzero. In particular, one can choose the nonzero $c_j$ to be equal to 1. To establish \eqref{Proof of Aprin equation 1} one just needs to run the algorithm 24 times.

\

\textbf{Simplification 4}. Without loss of generality, evaluate the final quantities (as soon as there are no more differentiations with respect to momentum) at the particular $\xi_0$ given by \eqref{xi0}.
This is sufficient due to rotational symmetry and homogeneity. Moreover, one can write $\xi=\xi_0+\eta$
and replace $\|\xi\|$ with its Taylor expansion in $\eta$. The latter simplifies differentiations in $\xi$.

\

\textbf{Simplification 5}. In implementing the algorithm, expand all quantities at the $k$-th step in powers of $x$ and $\eta$ retaining only terms up order $3-k$, jointly in $x$ and $\eta$. Here we are assuming that $x$ and $\eta$ are of the same order.

\

Simplifications 1--5 reduce our algorithm to the manipulation of polynomials.

\

We implement below the algorithm in the special case 
\begin{equation}
\label{case c5=1}
c_{11}=1, \qquad c_j=0, \ \text{for }j\ne 11.
\end{equation} 
The other cases are analogous.

Formulae \eqref{Ric at 0 simplification 3} and \eqref{14 September 2021 equation 1bis}--\eqref{case c5=1} imply that in the chosen coordinate system the metric tensor reads
\begin{equation*}
\label{Proof of Aprin equation 2}
g_{\alpha\beta}(x)
=
\begin{pmatrix}
1-\frac{x^1x^2x^3}3&\frac{(x^1)^2x^3}6&\frac{(x^1)^2x^2}6\\
\frac{(x^1)^2x^3}6&1&-\frac{(x^1)^3}6\\
\frac{(x^1)^2x^2}6&-\frac{(x^1)^3}6&1
\end{pmatrix}
+O(|x|^4)
\end{equation*}
and
\begin{equation*}
\label{Proof of Aprin equation 3}
\rho(x)=1-\frac{x^1x^2x^3}6+O(|x|^4).
\end{equation*}

Let us now retrace, step by step, the algorithm from Proposition~\ref{proposition algorighm Pj}. In view of Theorem~\ref{Theorem A order -3}, let us assume that we have carried out the first two iterations of the algorithm and have determined pseudo\-differential operators $P_{\pm,2}\in \Psi^0$ satisfying
\begin{equation*}
\label{Proof of Aprin equation 4}
P_{\pm,2}^2=P_{\pm,2} \mod \Psi^{-3}, \qquad P_{+,2}P_{-,2}=0\mod \Psi^{-3},\qquad [P_{\pm,2},\curl]=0 \mod \Psi^{-2}\,.
\end{equation*} 
This gives us
\begin{equation*}
\label{Proof of Aprin equation 5}
(p_{\pm,2})_{-1}(x,\xi_0+\eta)=
\frac1{24}
\begin{pmatrix}
\mp5\,(x^1)^2-4ix^1 x^2 & 6i\,(x^1)^2\mp3x^1 x^2 &0\\
-2i\,(x^1)^2\pm x^1 x^2& \mp3\, (x^1)^2 & 0\\
2i\,x^2 x^3 & -6i\,x^1 x^3 & 0
\end{pmatrix}+ O(|x|^3+|\eta|^3)
\end{equation*}
and
\begin{equation*}
\label{Proof of Aprin equation 6}
(p_{\pm,2})_{-2}(x,\xi_0+\eta)=
-\frac1{24}
\begin{pmatrix}
\pm ix^3& 6 x^3 &0\\
-2x^3&  \pm3ix^3  & 0\\
\pm 9i\,x^1 -2x^2 & 6x^1\pm3i x^2& 0
\end{pmatrix}+ O(|x|^2+|\eta|^2)\,.
\end{equation*}

Without loss of generality, we can assume that
\begin{equation*}
\label{Proof of Aprin equation 7}
(p_{\pm,1})_{-3}=0.
\end{equation*}
The latter can be achieved by choosing $X_{\pm,2}$ appropriately. Therefore, proving \eqref{Proof of Aprin equation 1} reduces to showing that
\begin{equation}
\label{Proof of Aprin equation 8}
[\operatorname{tr}(X_{+,3}-X_{-,3})_\prin](0,\xi_0)=-\frac12.
\end{equation}

By performing the relevant calculations with the simplifications described above and with account of \eqref{5 November 2021 equation 1}, one obtains
\begin{equation*}
\label{Proof of Aprin equation 9}
R_{\pm,3}(0,\xi_0)=
\frac{i}8
\begin{pmatrix}
0 & 1 &0\\
-1 & 0 & 0\\
0 & 0 & 0
\end{pmatrix},
\end{equation*}
\begin{equation*}
\label{Proof of Aprin equation 10}
S_{\pm,3}(0,\xi_0)=\mp\frac18
\begin{pmatrix}
1 & 0 &0\\
0& 1 & 0\\
0 & 0 & 0
\end{pmatrix},
\end{equation*}
\begin{equation*}
\label{Proof of Aprin equation 11}
T_{\pm,3}(0,\xi_0)=
\mp\frac{i}4
\begin{pmatrix}
0 & 1 &0\\
1 & 0 & 0\\
0 & 0 & 0
\end{pmatrix},
\end{equation*}
and, finally,
\begin{equation}
\label{Proof of Aprin equation 12}
(X_{\pm,3})_\prin(0,\xi_0)=
\mp\frac14
\begin{pmatrix}
1 & 0 & 0\\
0 & 0 & 0\\
0 & 0 & 0
\end{pmatrix}.
\end{equation}
Formula \eqref{Proof of Aprin equation 12} implies \eqref{Proof of Aprin equation 8}.
\end{proof}

\begin{remark}
Observe that neither in this section nor elsewhere in the paper our arguments make use of the second (differential) Bianchi identity. Of course, the second Bianchi identity implies that the number of independent components in formula \eqref{nablaG at 0 simplification 3} is 15, not 18. Namely, $c_{21}$, $c_{23}$ and $c_{24}$ can be expressed in terms of the remaining components.

The second  Bianchi identity may be needed for the calculation of the subprincipal symbol of the asymmetry operator. This matter is outside the scope of the current paper.
\end{remark}

Since formula \eqref{main theorem 2 equation 1} is one of the key results of this paper, we provide in Appendix~\ref{An alternative derivation of formula} an independent verification relying on formula~\eqref{formula for Pplus minus Pminus}.

\section{The regularised trace of the asymmetry operator}
\label{The regularised trace of $A$}

\subsection{Structure of the integral kernel near the diagonal}
\label{Structure of Aprin near the diagonal}

The asymmetry operator $A$ can be written as an integral operator
\begin{equation}
\label{structure frac a equation 1}
(Af)(x)=\int_M \mathfrak{a}(x,y)\,f(y)\,\rho(y)\,\dr y \,.
\end{equation}

Now, the operator $A$ is not guaranteed to be of trace class, so one cannot apply formula \eqref{properties of matrix trace of a pdo acting on 1-forms lemma equation 3} to it. The issue here is that in dimension 3 a sufficient condition for a pseudodifferential operator to be of trace class is that its order be strictly less than $-3$, whereas according to Theorem~\ref{main theorem 2} for a generic Riemannian 3-manifold the order of $A$ is exactly $-3$. This means that we need to perform some sort of very basic regularisation in order to get over the finish line and define, in an invariant and geometrically meaningful manner, the trace of $A$ .

In order to define an appropriate regularisation, we need to understand the structure of the integral kernel $\mathfrak{a}(x,y)$. More precisely, we need to characterise explicitly the way in which the latter fails to be continuous.

\begin{theorem}
\label{singularity theorem}
In a neighbourhood of the diagonal $\{(x,x)\,|\, x\in M\}\subset M\times M$ the integral kernel of $A$ can be written in local coordinates as 
\begin{equation}
\label{singularity theorem equation 1}
\mathfrak{a}(x,y)=\mathfrak{a}_d(x,y)+\mathfrak{a}_c(x,y)\,,
\end{equation}
where
\begin{equation}
\label{singularity theorem equation 2}
\mathfrak{a}_d(x,y):=\frac{1}{12\pi^2}\, E^{\alpha\beta}{}_\gamma(x)\, \nabla_\alpha \operatorname{Ric}_{\beta\rho}(x)\, \frac{(x-y)^\gamma (x-y)^\rho}{\dist^2(x,y)}
\end{equation}
and
$\mathfrak{a}_c$ is continuous as a function of two variables. Recall that the tensor $E$ is defined in accordance with \eqref{main theorem 1 equation 3}.
\end{theorem}

\begin{remark}
\label{singularity theorem tensor is trace-free}
The tensor $\,E^{\alpha\beta}{}_\gamma\,\nabla_\alpha \operatorname{Ric}_{\beta\rho}\,$
appearing in formula \eqref{singularity theorem equation 2} is trace-free,
\begin{equation}
\label{singularity theorem tensor is trace-free equation}
E^{\alpha\beta}{}_\gamma\,\nabla_\alpha \operatorname{Ric}_{\beta\rho}
\,g^{\gamma\rho}
=0.
\end{equation}
This is an important property which will be used on two occasions in the subsequent arguments in this section. We will see that property \eqref{singularity theorem tensor is trace-free equation} implies that the singularity of the integral kernel of our asymmetry operator $A$ is weaker than one would expect of a generic pseudodifferential operator of order $-3$. For a generic pseudodifferential operator of order $-3$ one would have a logarithm in the leading term of the singularity of the integral kernel, but there is no logarithm in \eqref{singularity theorem equation 2}.
\end{remark}

In order to prove Theorem~\ref{singularity theorem} we need to establish first a number of preparatory lemmata. Until further notice we will be working in Euclidean space $\mathbb{R}^3$ equipped with Cartesian coordinates.
In what follows we use the \emph{Japanese bracket} notation
\begin{equation}
\label{Japanese bracket}
\langle\xi\rangle
:=
(1+|\xi|^2)^{1/2}.
\end{equation}

\begin{lemma}
\label{lemma Bessel}
We have
\begin{equation}
\label{lemma Bessel equation}
\frac1{(2\pi)^3}\int_{\mathbb{R}^3}\frac{1}{\langle\xi\rangle^5} \,e^{-i y^\mu\,\xi_\mu}\,  \, \dr\xi=
\frac{1}{6\pi^2}
|y|\,K_1(|y|)\,,
\end{equation}
where $K_1$ is the modified Bessel function of the second kind \cite[\S~9.6, p.~374]{AS}.
\end{lemma}

\begin{proof}
Let us observe that the function
\begin{equation*}
\label{proof lemma Bessel equation 1}
y \mapsto \frac1{(2\pi)^3}\int_{\mathbb{R}^3}\frac{1}{\langle\xi\rangle^5} \,e^{-i y^\mu\,\xi_\mu}\,  \, \dr\xi
\end{equation*}
is real and spherically symmetric. Therefore, we have
\begin{equation}
\label{proof lemma Bessel equation 2}
\begin{split}
\frac1{(2\pi)^3}\int_{\mathbb{R}^3}\frac{1}{\langle\xi\rangle^5} \,e^{-i y^\mu\,\xi_\mu}\,  \, \dr\xi
&
=
\frac1{(2\pi)^3}\int_{\mathbb{R}^3}\frac{1}{\langle\xi\rangle^5} \,\cos \left( |y|\,\xi_1\right)\,  \, \dr\xi\,
\\
&
=
\frac{2\pi}{(2\pi)^3}\int_{-\infty}^{+\infty}\cos \left( |y|\,\xi_1\right) \left(\int_{0}^{+\infty}\frac{r}{(1+(\xi_1)^2+r^2)^{5/2}} \, \dr r\right) \, \dr\xi_1\,
\\
&
=
\frac{1}{12\pi^2}\int_{-\infty}^{+\infty}\frac{\cos \left(|y|\,\xi_1\right)}{(1+(\xi_1)^2)^{3/2}} \, \dr\xi_1\,,
\end{split}
\end{equation}
where at the second step we performed the change of variable $(\xi_2,\xi_3)=(r \cos \theta, r \sin \theta)$. Formula \eqref{lemma Bessel equation} now follows by explicitly evaluating the integral in the single variable $\xi_1$ appearing in the RHS of \eqref{proof lemma Bessel equation 2}:
\begin{equation*}
\label{proof lemma Bessel equation 3}
\int_{-\infty}^{+\infty}\frac{\cos \left(|y|\,\xi_1\right)}{(1+(\xi_1)^2)^{3/2}} \, d\xi_1=2 |y|\, K_1(|y|)\,.
\end{equation*}
The latter is known as Basset’s Integral \cite[\href{https://dlmf.nist.gov/10.32.E11}{(10.32.11)}]{DLMF}.
\end{proof}

\begin{lemma}
\label{lemma Bessel singularity}
We have
\begin{equation}
\label{lemma Bessel singularity equation}
\frac1{(2\pi)^3}\int_{\mathbb{R}^3}\frac{1}{\langle\xi\rangle^5} \,e^{-i y^\mu\,\xi_\mu}\,  \, \dr\xi=
\frac{1}{12\pi^2}
\,|y|^2\,\ln |y|
+
h(y),
\quad\text{where}\quad
h\in C^3(\mathbb{R}^3).
\end{equation}
\end{lemma}

\begin{proof}
It is known that the modified Bessel function of the second kind $K_1$ admits the asymptotic expansion
\begin{equation}
\label{proof lemma Bessel singularity equation 1}
K_1(t)=\frac1t+\frac{t}4 \left(2 \ln t +2\gamma -1-\ln 4 \right) + O(t^3\ln t) \quad \text{as}\quad t\to 0^+,
\end{equation}
where $\gamma$ is the Euler--Mascheroni constant --- see, e.g., \cite[\href{https://dlmf.nist.gov/10.31.E1}{(10.31.11)} and \href{https://dlmf.nist.gov/10.25.E2}{(10.25.2)}]{DLMF}. By combining \eqref{proof lemma Bessel singularity equation 1} and \eqref{lemma Bessel equation} one obtains \eqref{lemma Bessel singularity equation}.
\end{proof}

\begin{lemma}
\label{lemma Bessel singularity differentiated}
We have
\begin{equation}
\label{lemma Bessel singularity  differentiated equation}
\frac1{(2\pi)^3}\int_{\mathbb{R}^3}\frac{\xi_\gamma\xi_\rho}{\langle\xi\rangle^5} \,e^{-i y^\mu\,\xi_\mu}\,  \, \dr\xi
=
-\frac{1}{12\pi^2}
\left[
2\,\frac{y^\gamma y^\rho}{|y|^2}
+
(1+2\ln|y|)
\,\delta_{\gamma\rho}
\right]
+
h_{\gamma\rho}(y),
\quad\text{where}\quad
h_{\gamma\rho}\in C^1(\mathbb{R}^3).
\end{equation}
\end{lemma}

\begin{proof}
Formula \eqref{lemma Bessel singularity  differentiated equation} is a straightforward consequence of \eqref{lemma Bessel singularity equation}, once one observes that
\[
\frac{\xi_\gamma\xi_\rho}{\langle\xi\rangle^5} \,e^{-i y^\mu\,\xi_\mu}=-\frac{\partial^2}{\partial y^\gamma \partial y^\rho} \left( \frac{1}{\langle\xi\rangle^5} \,e^{-i y^\mu\,\xi_\mu}\right)\,.
\]
\end{proof}

\begin{proof}[Proof of Theorem~\ref{singularity theorem}]
In the previous section we have established that the asymmetry operator $A$ is a pseudodifferential operator of order $-3$ with principal symbol \eqref{main theorem 2 equation 1}. Locally, the symbol $a(x,\xi)$ of $A$ and its integral kernel $\mathfrak{a}(x,y)$ are related via the (distributional) identity
\begin{equation}
\label{structure frac a equation 2}
\mathfrak{a}(x,y)=\frac{1}{(2\pi)^3\,\rho(y)}\int_{\mathbb{R}^3} e^{i (x-y)^\mu\,\xi_\mu}\, a(x,\xi)\, \dr\xi\,.
\end{equation}

Since a pseudodifferential operator of order $-4$ on a 3-manifold has a continuous integral kernel, formulae \eqref{main theorem 2 equation 1}, \eqref{main theorem 1 equation 3} and 
\eqref{structure frac a equation 2} imply that problem at hand reduces,
in normal coordinates centred at $x$, 
to examining the quantity
\begin{equation*}
\label{structure frac a equation 3}
-\frac{1}{16 \pi^3} \varepsilon^{\alpha\beta\gamma}\, \nabla_\alpha \operatorname{Ric}_{\beta}{}^\rho(0)\,\int_{\mathbb{R}^3} e^{-iy^\mu\,\xi_\mu}\, \frac{\xi_\gamma\xi_\rho}{|\xi|^5} \,
(1-\chi(|\xi|))\,\dr\xi\,,
\end{equation*}
where $\chi$ is a cut-off as in \eqref{Q acting on 1-forms 6}.
But
\begin{equation*}
\label{structure frac a equation 4}
\frac{\xi_\gamma\xi_\rho}{|\xi|^5}
=
\frac{\xi_\gamma\xi_\rho}{\langle\xi\rangle^5}
+
O
\left(
\frac{1}{|\xi|^5}
\right),
\end{equation*}
so the required result follows from Lemma~\ref{lemma Bessel singularity differentiated}.
Note that the term with $\delta_{\gamma\rho}$ from
\eqref{lemma Bessel singularity  differentiated equation}
vanishes upon contraction with
$\,\varepsilon^{\alpha\beta\gamma}\, \nabla_\alpha \operatorname{Ric}_{\beta}{}^\rho(0)\,$
due to formula \eqref{singularity theorem tensor is trace-free equation}.
\end{proof}

\

Theorem~\ref{singularity theorem}
tells us that the singularity of the integral kernel $\mathfrak{a}(x,y)$ of our asymmetry operator $A$ is very weak. We are looking at a bounded function of two variables with a discontinuity on the diagonal. The discontinuity of $\mathfrak{a}(x,y)$ on the diagonal manifests itself in the fact that limit of $\mathfrak{a}(x,y)$ as $y$ tends to $x$ depends on the direction along which $y$ tends to $x$.

The bottom line is that working with the asymmetry operator $A$ does not require the use of the theory of distributions or microlocal analysis. It is an integral operator whose integral kernel has a mild discontinuity on the diagonal.

\subsection{Local and global trace}
\label{Local and global trace}

The decomposition
\eqref{singularity theorem equation 1},
\eqref{singularity theorem equation 2}
depends on the choice of local coordinates.
However, the quantity
\begin{equation}
\label{local trace definitiion 1}
\psi_{\operatorname{curl}}^\loc(x):=\mathfrak{a}_c(x,x)
\end{equation}
does not depend on the choice of local coordinates, i.e.~is a true scalar function $\psi_{\operatorname{curl}}:M\to\mathbb{R}$. Let us elaborate. Denote
\begin{equation}
\label{elaborate 1}
f(x,y)
:=
12\pi^2
\,\mathfrak{a}_d(x,y)
\,\dist^2(x,y)
=E^{\alpha\beta}{}_\gamma(x)\, \nabla_\alpha \operatorname{Ric}_{\beta\rho}(x)\,(x-y)^\gamma (x-y)^\rho\,.
\end{equation}
The issue at hand is that the expression \eqref{elaborate 1} is not a scalar because the quantities $(x-y)^\gamma$ and $(x-y)^\rho$ are not vectors based at the point $x$. Let us switch from local coordinates $x$ to local coordinates $\tilde x$ and let $\tilde f(\tilde x,\tilde y)$ be the representation of the quantity \eqref{elaborate 1} in these new coordinates. It is easy to see, by writing down the appropriate Jacobians, that
\begin{equation}
\label{elaborate 2}
\tilde f(\tilde x(x),\tilde y(y))-f(x,y)
=
O(\dist^3(x,y)).
\end{equation}
Now, let $\tilde{\mathfrak{a}}_c(\tilde x,\tilde y)$ be the representation of the quantity $\mathfrak{a}_c(x,y)$ in the new coordinates.
Formulae
\eqref{singularity theorem equation 1},
\eqref{singularity theorem equation 2},
\eqref{elaborate 1},
\eqref{elaborate 2}
and the fact that $\mathfrak{a}(x,y)$ is a true scalar in two variables imply
\begin{equation*}
\label{elaborate 3}
\tilde{\mathfrak{a}}_c(\tilde x(x),\tilde y(y))-\mathfrak{a}_c(x,y)
=
O(\dist(x,y)),
\end{equation*}
which, in turn,implies 
$\tilde{\mathfrak{a}}_c(\tilde x(x),\tilde x(x))=\mathfrak{a}_c(x,x)$.

Using formulae
\eqref{singularity theorem equation 1},
\eqref{singularity theorem equation 2}
and
\eqref{local trace definitiion 1}
for the calculation of the function $\psi_{\operatorname{curl}}$
is impractical. The following lemma provides an alternative equivalent representation.

\begin{lemma}
\label{lemma averaging over sphere}
We have
\begin{equation}
\label{lemma averaging over sphere equation}
\psi_{\operatorname{curl}}^\loc(x)
=
\lim_{r\to0^+}
\frac{1}{4\pi r^2}
\int_{\mathbb{S}_r(x)}
\mathfrak{a}(x,y)\,\dr S_y\,,
\end{equation}
where $\mathfrak{a}(x,y)$ is the integral kernel
of the asymmetry operator
\eqref{definition asymmetry operator formula}
defined in accordance with
\eqref{structure frac a equation 1},
$\mathbb{S}_r(x)=\{y\in M|\dist(x,y)=r\}$ is the sphere of radius $r$ centred at $x$
and $\dr S_y$ is the surface area element on this sphere.
\end{lemma}

\begin{proof}
Let us fix an $x\in M$ and choose normal coordinates centred at $x$.
Theorem~\ref{singularity theorem} tells us that
\begin{equation*}
\label{lemma averaging over sphere proof equation 1}
\mathfrak{a}(x,y)
=
\mathfrak{a}_c(x,y)
+
c_{\gamma\rho}(x)\,\frac{(x-y)^\gamma(x-y)^\rho}{|x-y|^2}\,,
\end{equation*}
where
\begin{equation*}
\label{lemma averaging over sphere proof equation 2}
c_{\gamma\rho}(x)
=
\frac{1}{12\pi^2}\, E^{\alpha\beta}{}_\gamma(x)\, \nabla_\alpha \operatorname{Ric}_{\beta\rho}(x)\,.
\end{equation*}
We have
\begin{equation*}
\label{lemma averaging over sphere proof equation 3}
\int_{\mathbb{S}_r(x)}
(x-y)^\gamma(x-y)^\rho
\,\dr S_y
=
\begin{cases}
\frac{4\pi r^2}{3}\quad&\text{if}\quad\gamma=\rho,
\\
\ \ 0\quad&\text{if}\quad\gamma\ne\rho,
\end{cases}
\end{equation*}
hence,
\begin{equation}
\label{lemma averaging over sphere proof equation 4}
\frac{1}{4\pi r^2}
\int_{\mathbb{S}_r(x)}
\mathfrak{a}(x,y)\,dS_y
=
\frac{1}{4\pi r^2}
\int_{\mathbb{S}_r(x)}
\mathfrak{a}_c(x,y)\,\dr S_y
+
\frac{c_{11}(x)+c_{22}(x)+c_{33}(x)}{3}\,.
\end{equation}
But formula \eqref{singularity theorem tensor is trace-free equation} implies
$c_{11}(x)+c_{22}(x)+c_{33}(x)=0$, so the second term in the right-hand side of
\eqref{lemma averaging over sphere proof equation 4} vanishes.
Formulae
\eqref{lemma averaging over sphere proof equation 4}
and
\eqref{local trace definitiion 1}
now imply
\eqref{lemma averaging over sphere equation}
by the continuity of $\mathfrak{a}_c(x,y)$.
\end{proof}

The advantage of formula \eqref{lemma averaging over sphere equation} is that it does not involve the use of local coordinates. Hence, it can be viewed as a more convenient and natural definition of the scalar function $\psi_{\operatorname{curl}}^\loc\,$.

\begin{definition}
\label{local regularised trace of the asymmetry operator}
We call the scalar continuous function $\psi_{\operatorname{curl}}^\loc(x)$, $\psi_{\operatorname{curl}}^\loc:M\to\mathbb{R}$, defined in accordance with formulae
\eqref{lemma averaging over sphere equation}
and
\eqref{structure frac a equation 1}
\emph{the regularised local trace of the asymmetry operator}.
\end{definition}
\begin{definition}
\label{global regularised trace of the asymmetry operator}
We call the number
\begin{equation*}
\label{global regularised trace of the asymmetry operator equation}
\psi_{\operatorname{curl}}
:=
\int_M
\psi_{\operatorname{curl}}^\loc(x)
\,\rho(x)
\,\dr x
\end{equation*}
\emph{the regularised global trace of the asymmetry operator}.
\end{definition}

The function $\psi_{\operatorname{curl}}^\loc(x)$ and the number $\psi_{\operatorname{curl}}$
are determined by the Riemannian 3-manifold and its orientation. This means that we are looking at geometric invariants.

\section{Challenges in higher dimensions}
\label{Challenges in higher dimensions}

Let $(M,g)$ be a connected closed Riemannian manifold of dimension
\begin{equation}
\label{higher dimensions}
d=4k+3,\qquad k=1,2,\ldots. 
\end{equation}
In this case one can define $\curl$ in the usual way \eqref{curl differential experssion} as an operator
acting in $\Omega^{2k+1}(M)$. It is easy to see that $\curl$ remains formally self-adjoint and at a formal level everything appears to be similar to what happens in dimension 3.

\begin{remark}
One could have also considered the case $d=4k+1$, $k=1,2,\ldots$,
with $\curl$ acting in $\Omega^{2k}(M)$ in accordance with \eqref{curl differential experssion}.
However, in this case $\curl$ is formally anti-self-adjoint, the spectrum is purely imaginary
and there is no spectral asymmetry \cite[Theorem 3.2]{baer_curl}.
As our interest is in spectral asymmetry, we do not discuss the case $d=4k+1$.
Note also that the spectral problem for $\curl$ in dimension $d=4k+1$ cannot be formulated in
terms of the spectral theory of self-adjoint operators in \emph{real} Hilbert spaces: one can introduce the
factor $\,i\,$ in the right-hand side of \eqref{curl differential experssion} to make the operator
formally self-adjoint (as was done in \cite[Definition 2.2]{baer_curl}),
but the resulting operator will no longer be an operator
acting in the vector space of \emph{real} differential forms.
The bottom line is that the two cases,
$d=4k+3$ and $d=4k+1$, are fundamentally different.
We stick with \eqref{higher dimensions} because this gives the
only genuine generalisation of the operator $\curl$ to higher dimensions.
\end{remark}

\

Analysis of higher dimensions \eqref{higher dimensions} presents two major challenges.

\

\textbf{First challenge}.
We expect that in higher dimensions the asymmetry operator will still be a pseudodifferential operator of order $-3$, as it is in the 3-dimensional case. However, a sufficient condition for a pseudodifferential operator to be of trace class is that its order be less than $-d$. This means that defining the regularised trace of the asymmetry operator becomes challenging: one would need to analyse singularities of the integral kernel and perform regularisation reducing the order of the singularity by $4k+1$. In other words, one would need to perform $4k+1$ steps of regularisation. Doing this in a geometrically invariant fashion will not be an easy task.

\

\textbf{Second challenge}.
In the current paper we used the results of \cite{part1} which apply to operators whose principal symbols have simple eigenvalues (see also \cite{part2}). In dimension 3 the operator $\curl$ does possess this property --- its principal symbol has three simple eigenvalues \eqref{eigenvalues principal symbol curl}. In higher dimensions the eigenvalues of the principal symbol of the operator $\curl$ are still given by formulae \eqref{eigenvalues principal symbol curl}, only now they have multiplicity:
zero is an eigenvalue of multiplicity $\binom{4k+2}{2k}$,
whereas each of the two nonzero eigenvalues has multiplicity $\frac{1}{2}\binom{4k+2}{2k+1}$.
This means that addressing the issue of higher dimensions would require extending the results of \cite{part1} to operators whose principal symbols have multiple eigenvalues. This will be a difficult task
as the construction presented in \cite{part1} relied heavily on the assumption that the eigenvalues of the principal symbol of the operator are simple. Overall, partial differential operators
whose principal symbols have multiple eigenvalues are known to present a major challenge in micrololcal analysis.

\

Despite the challenges involved, examining the asymmetry operator in higher dimension could be an piece of research worth pursuing in the future. For example, since the seminal works by Milnor and Kervaire \cite{milnor1,milnor2}
the study of exotic smooth structures on 7-spheres has been an active area of research in differential geometry. Our asymmetry operator, being very sensitive to the underlying geometry, might help to make some progress in this area.

\section*{Acknowledgements}
\addcontentsline{toc}{section}{Acknowledgements}

MC was partially supported by a Leverhulme Trust Research Project Grant RPG-2019-240, by a Research Grant (Scheme 4) of the London Mathematical Society, by a grant of the Heilbronn Institute for Mathematical Research (HIMR) via the UKRI/EPSRC Additional Funding Programme for Mathematical Sciences, and by EPSRC Fellowship EP/X01021X/1.

We are grateful to Christian B\"ar, Bruno Colbois, Stuart Dowker, Michal Kwasigroch, Jason Lotay and  Nikolai Saveliev for useful comments and suggestions. We also thank Michael Levitin for kindly sharing his Mathematica scripts for the Berger sphere.

\begin{appendices}

\section{Exterior calculus}
\label{Exterior calculus}

In this appendix we set out out conventions on exterior calculus. Henceforth, we identify differential forms with covariant antisymmetric tensors, and $M$ is an oriented 3-manifold equipped with Riemannian metric $g$
and Levi-Civita connection $\nabla$.

It is well known that the metric $g$ induces a canonical isomorphism
between the tangent bundle $TM$
and the cotangent bundle $T^*M$, the so-called \emph{musical isomorphism}.
We denote it by $\flat:TM\to T^*M$ (lowering of indices) and its inverse by $\sharp:T^*M\to TM$ (raising of indices).

Given a scalar field $f\in C^\infty(M)$, its
exterior derivative $\mathrm{d}f$ is defined as the gradient.
Given a 1-form $u\in\Omega^1(M)$, its
exterior derivative $\mathrm{d}u\in\Omega^2(M)$
is defined, componentwise, as
\[
(\dr u)_{\alpha\beta}
=
\partial_{x^\alpha}u_\beta-\partial_{x^\beta}u_\alpha\,.
\]

We define the action of the Hodge star $*:\Omega^k(M)\to \Omega^{3-k}(M)$ on a rank $k$ antisymmetric tensor as
\begin{equation}
\label{definition of Hodge star}
(*v)_{\mu_{k+1}\ldots\mu_3}\!:=\frac1{k!}\,
\sqrt{\det g_{\alpha\beta}}
\ v^{\mu_1\ldots\mu_k}\,\varepsilon_{\mu_1\mu_2\mu_3}\,,
\end{equation}
where $\varepsilon$ is the totally antisymmetric symbol,
$\varepsilon_{123}:=+1$. 

With regards to the totally antisymmetric symbol, we adopt the convention of raising indices using the Kronecker symbol, so that, for example,
\[
\varepsilon^{\alpha_1\alpha_2\alpha_3}:=
\varepsilon_{\alpha_1\alpha_2\alpha_3}.
\]

Let $\wedge$ denote the exterior product. Given a pair of real-valued rank $k$ covariant antisymmetric tensors $u$ and $v$
we define their 
$L^2$ inner product as
\begin{equation*}
\begin{split}
\langle u,v\rangle
:&=
\int_M
\frac{1}{k!}\,
u_{\alpha_1\dots\alpha_k}
\,
v_{\beta_1\dots\beta_k}
\,
g^{\alpha_1\beta_1}\cdots g^{\alpha_k\beta_k}
\,\sqrt{\det g_{\mu\nu}}
\ \dr x 
\\
&
=
\int_M u\wedge *v =\int_M *u\wedge v\,.
\end{split}
\end{equation*}
For the sake of clarity, let us mention that the exterior product of 1-forms
reads
\[
(u\wedge v)_{\alpha\beta}=u_\alpha v_\beta-u_\beta v_\alpha\,.
\]

Given $w\in\Omega^k(M)$ and $u\in\Omega^{k-1}(M)$
we define the action of the codifferential 
$\delta:\Omega^k(M)\to\Omega^{k-1}(M)$ in accordance with
\[
\langle w,\dr u\rangle=\langle\delta w,u\rangle.
\]
In particular, when $u\in\Omega^1(M)$ and $w\in\Omega^2(M)$,
we get in local coordinates
\[
\delta u=-\nabla^\alpha u_\alpha,
\]
\[
(\delta w)_\alpha=\nabla^\beta w_{\alpha\beta}.
\]

\section{Basic properties of the operator curl}
\label{Basic properties of the operator curl}

\color{black}

In this appendix we provide, in a self-contained fashion, proofs of the fundamental properties of $\operatorname{curl}$ stated in Section~\ref{The operator curl}.

\

In order to prove the Theorem~\ref{theorem properties of curl}, let us introduce an auxiliary elliptic operator: \emph{extended curl}.

\subsection{Extended curl}

It is known that one can extend curl to a $4\times 4$ elliptic operator acting in $\Omega^1\oplus \Omega^0$.

\begin{definition}
\label{definition extended curl}
We define \emph{extended curl} to be the operator
\begin{equation}
\label{definition extended curl equation 1}
\operatorname{curl}_E:=\begin{pmatrix}
\operatorname{curl}&\dr\\
\delta&0
\end{pmatrix}:
(\Omega^1  \cap H^1) \oplus (\Omega^0\cap H^1)
\to
\Omega^1\oplus \Omega^0\,.
\end{equation}
\end{definition}

The operator \eqref{definition extended curl equation 1} appeared in
\cite{millson},
see also \cite[p.~229]{asymm1}, \cite[pp.~44--45 and p.~66]{asymm2}, \cite[p.~405]{asymm3}, \cite[pp.~190--191]{singerICM} and \cite[\S 4]{dowker66}.

\begin{remark}
Note that $\curl_E$ is an operator of Dirac type, namely, its square is $\delta \dr+\dr\delta$ on $\Omega^1\oplus\Omega^0$ (minus the nonpostive Hodge Laplacian).
Furthermore, \eqref{definition extended curl equation 1} is the restriction of the \emph{signature operator} $\dr+\delta$ to differential forms of even degree, in the sense that
\begin{equation*}
\begin{pmatrix}
*&0\\
0&1
\end{pmatrix}
\operatorname{curl}_E
\begin{pmatrix}
*&0\\
0&1
\end{pmatrix}
=(\dr+\delta)*:
\Omega^2\oplus\Omega^0\to\Omega^2\oplus\Omega^0.
\end{equation*}
\end{remark}

A straightforward calculation shows that the principal symbol of extended curl reads
\begin{equation}
\label{principal symbol extended curl}
(\operatorname{curl}_E)_\mathrm{prin}(x,\xi)
=
\begin{pmatrix}
[(\operatorname{curl})_\mathrm{prin}]_\alpha{}^\beta(x,\xi) & i\xi_\alpha\\
-i g^{\beta\gamma}(x)\,\xi_\gamma & 0
\end{pmatrix}.
\end{equation}
The eigenvalues of \eqref{principal symbol extended curl} are
\begin{equation}
\label{eigenvalues principal symbol extended curl}
\pm\|\xi\|\, ,
\end{equation}
each with multiplicity two,  so that we have
\begin{equation}
\label{determinant principal symbol extended curl}
\det \left[(\operatorname{curl}_E)_\mathrm{prin}\right]=\|\xi\|^4\,,
\end{equation}
Compare formulae
\eqref{eigenvalues principal symbol extended curl}
and
\eqref{determinant principal symbol extended curl}
with
\eqref{eigenvalues principal symbol curl}
and
\eqref{determinant principal symbol curl} respectively.

Formulae \eqref{determinant principal symbol extended curl} and \eqref{definition extended curl equation 1} imply that $\operatorname{curl}_E$ is elliptic and self-adjoint.

\begin{definition}[{\cite[Definition~1.7]{schmudgen}}]
\label{definition invariant subspace}
Let $L:\mathcal{D}(L)\to H$ be a self-adjoint (possibly unbounded) linear operator in the Hilbert space $H$. We say that a closed vector subspace $V\subseteq H$ is an \emph{invariant subspace} of the operator $L$ if
\begin{equation*}
\label{definition invariant subspace equation 1}
L(\mathcal{D}(L) \cap V) \subseteq  V.
\end{equation*}
\end{definition}

\begin{proposition}
\label{lemma invariant subspaces extended curl}
The Hilbert space $\Omega^1\oplus \Omega^0$ decomposes into a direct sum of three orthogonal Hilbert subspaces  
\begin{equation}
\label{lemma invariant subspaces extended curl equation 1}
\Omega^1\oplus \Omega^0=[\dr\Omega^0 \oplus \delta\Omega^1] \oplus[\delta\Omega^2  \oplus 0] \oplus [ \mathcal{H}^1  \oplus \mathcal{H}^0]
\end{equation}
which are invariant subspaces of the operator $\operatorname{curl}_E$ in the sense of Definition~\ref{definition invariant subspace}.
Accordingly, the operator $\operatorname{curl}_E$ decomposes into a direct sum of three operators
\begin{equation}
\label{lemma invariant subspaces extended curl equation 2}
\operatorname{curl}_E=\operatorname{curl}_{E,\dr} \oplus \operatorname{curl}_{E,\delta} \oplus\operatorname{curl}_{E,\mathcal{H}}\,,
\end{equation}
where
\begin{equation}
\label{lemma invariant subspaces extended curl equation 3}
\operatorname{curl}_{E,\dr}:(\dr\Omega^0  \cap H^1) \oplus (\delta\Omega^1\cap H^1)\to \dr\Omega^0 \oplus \delta\Omega^1,
\end{equation}
\begin{equation}
\label{lemma invariant subspaces extended curl equation 4}
\operatorname{curl}_{E,\delta}:(\delta\Omega^2  \cap H^1) \oplus 0 \to \delta\Omega^2  \oplus 0,
\end{equation}
\begin{equation}
\label{lemma invariant subspaces extended curl equation 5}
\operatorname{curl}_{E,\mathcal{H}}: \mathcal{H}^1  \oplus \mathcal{H}^0 \to \mathcal{H}^1  \oplus \mathcal{H}^0.
\end{equation}
\end{proposition}

\begin{proof}
That the three subspaces in the RHS of \eqref{lemma invariant subspaces extended curl equation 1} are mutually orthogonal and their direct sum gives the whole space $\Omega^1\oplus\Omega^0$ follows immediately from the Hodge decomposition theorem. Let us show that they are invariant subspaces of $\curl_E\,$. We will do this in several steps.

{\bf Step 1}.
Observe that $\curl_E$ maps $\mathcal{H}^1\oplus\mathcal{H}^0$ to zero. Hence, $\mathcal{H}^1\oplus\mathcal{H}^0$ is an invariant subspace of~$\curl_E\,$.

{\bf Step 2}.
The facts that
\begin{itemize}
\item
$\mathcal{H}^1\oplus\mathcal{H}^0$ is an invariant subspace of $\curl_E$ and
\item
$[\dr\Omega^0\oplus\delta\Omega^1]\oplus[\delta\Omega^2\oplus 0]$ is the orthogonal complement of $\mathcal{H}^1\oplus\mathcal{H}^0$
\end{itemize}
imply that $[\dr\Omega^0\oplus\delta\Omega^1]\oplus[\delta\Omega^2\oplus 0]$ is an invariant subspace of $\curl_E\,$.

{\bf Step 3}.
Choose a $\lambda$ in the resolvent set $\rho(\curl_E)$ of the operator $\curl_E$. It will be convenient for us to choose a real $\lambda$,
\begin{equation}
\label{lemma invariant subspaces extended curl equation Dima 1}
\lambda\in\mathbb{R}.
\end{equation}
This can always be achieved because the operator $\curl_E$ is elliptic and self-adjoint, hence its spectrum is real and discrete. Note that zero is an eigenvalue of $\curl_E$ (see Step 1 above), so
\begin{equation}
\label{lemma invariant subspaces extended curl equation Dima 2}
\lambda\ne0.
\end{equation}

Let $\,\Delta:=-\delta\dr\,$ be the (nonpositive) Laplace--Beltrami operator. Let
\begin{equation}
\label{eigenvalues of Laplace--Beltrami operator}
0=\mu_0<\mu_1\le\mu_2\le\dots
\end{equation}
be the eigenvalues of $\,-\Delta\,$ enumerated in increasing order with account of multiplicities.
It will be convenient for us to further restrict the choice of $\lambda$ by imposing the condition
\begin{equation}
\label{lemma invariant subspaces extended curl equation Dima extra}
\lambda\not\in\{\pm\sqrt{\mu_1}\,,\,\pm\sqrt{\mu_2}\,,\dots\}.
\end{equation}

{\bf Step 4}.
Observe that in order to prove that the two subspaces $\dr\Omega^0\oplus\delta\Omega^1$ and $\delta\Omega^2\oplus 0$ are invariant subspaces of $\curl_E$ it is sufficient to prove that these are invariant subspaces of the resolvent $\mathcal{R}(\lambda):=(\curl_E- \lambda\operatorname{Id})^{-1}\,$.

{\bf Step 5}.
In view of \eqref{lemma invariant subspaces extended curl equation Dima 1} the operator $\mathcal{R}(\lambda)$ is self-adjoint. The facts that
\begin{itemize}
\item
$[\dr\Omega^0\oplus\delta\Omega^1]\oplus[\delta\Omega^2\oplus 0]$ is an invariant subspace of $\mathcal{R}(\lambda)$ and
\item
the subspaces $\dr\Omega^0\oplus\delta\Omega^1$ and $\delta\Omega^2\oplus 0$ are mutually orthogonal
\end{itemize}
imply that it suffices to prove that only one of the subspaces $\dr\Omega^0\oplus\delta\Omega^1$ and $\delta\Omega^2\oplus 0$ is an invariant subspace of $\mathcal{R}(\lambda)$. In what follows we prove that $\delta\Omega^2\oplus 0$ is an invariant subspace of $\mathcal{R}(\lambda)$.

{\bf Step 6}.
Let $\Omega^k_\infty$ denote the vector space of infinitely smooth real-valued $k$-forms and let
\begin{equation}
\label{lemma invariant subspaces extended curl equation Dima 3}
\dr\Omega^{k-1}_\infty
:=
\{
\dr\omega\ |\ \omega\in\Omega^{k-1}_\infty
\},
\end{equation}
\begin{equation}
\label{lemma invariant subspaces extended curl equation Dima 4}
\delta\Omega^{k+1}_\infty
:=
\{
\delta\omega\ |\ \omega\in\Omega^{k+1}_\infty
\}.
\end{equation}
The Hilbert spaces
$\Omega^k$, $\dr\Omega^{k-1}$ and $\delta\Omega^{k+1}$
are defined as $L^2$ closures of the vector spaces
$\Omega^k_\infty$, $\dr\Omega^{k-1}_\infty$ and $\delta\Omega^{k+1}_\infty$ respectively,
see \cite[Corollary~3.4.2]{jost}. This implies that the vector space
$\delta\Omega^2_\infty\oplus 0$
is dense in the Hilbert spaces
$\delta\Omega^2\oplus 0$.
As $\mathcal{R}(\lambda)$ is a bounded operator,
in order to prove that $\delta\Omega^2\oplus 0$ is an invariant subspace of $\mathcal{R}(\lambda)$ it is sufficient to prove that
\begin{equation}
\label{lemma invariant subspaces extended curl equation Dima 5}
[\mathcal{R}(\lambda)](\delta\Omega^2_\infty\oplus 0)
\subseteq
\delta\Omega^2_\infty\oplus 0\,.
\end{equation}

{\bf Step 7}.
Let
$z
=
\begin{pmatrix}
z_1
\\
0
\end{pmatrix}
\in
\delta\Omega^2_\infty\oplus 0$.
Put
\begin{equation}
\label{lemma invariant subspaces extended curl equation Dima 6}
v:=\mathcal{R}(\lambda)z
=
\begin{pmatrix}
v_1
\\
v_0
\end{pmatrix}
\in
[\dr\Omega^0_\infty\oplus\delta\Omega^1_\infty]\oplus[\delta\Omega^2_\infty\oplus 0].
\end{equation}
Here $z_1$, $v_1$ are 1-forms and $v_0$ is a 0-form (scalar).
The $v_1$ and $v_0$ are infinitely smooth because the operator $\curl_E$ is elliptic. 
Formula \eqref{lemma invariant subspaces extended curl equation Dima 6} can be rewritten more explicitly as
\begin{align}
\label{lemma invariant subspaces extended curl equation Dima 7}
\curl v_1+\dr v_0-\lambda v_1
&=
z_1\,,
\\
\label{lemma invariant subspaces extended curl equation Dima 8}
\delta v_1-\lambda v_0
&=
0\,.
\end{align}
We have \cite[Section 6.8, p.~223]{warner}
\begin{equation}
\label{lemma invariant subspaces extended curl equation Dima 9}
\Omega^k_\infty
=
\dr\Omega^{k-1}_\infty
\oplus
\delta\Omega^{k+1}_\infty
\oplus
\mathcal{H}^k,
\end{equation}
compare with \eqref{Hodge decompostion}. We already know (see Step 2) that the 1-form $v_1$ does not contain a contribution from harmonic 1-forms, so formulae \eqref{lemma invariant subspaces extended curl equation Dima 9}, \eqref{lemma invariant subspaces extended curl equation Dima 3} and \eqref{lemma invariant subspaces extended curl equation Dima 4} give us
\begin{equation}
\label{lemma invariant subspaces extended curl equation Dima 10}
v_1=v_{1\dr}+v_{1\delta}\,,
\end{equation}
where
\begin{equation}
\label{lemma invariant subspaces extended curl equation Dima 11}
v_{1\dr}=\dr w_0\,,
\end{equation}
\begin{equation}
\label{lemma invariant subspaces extended curl equation Dima 12}
v_{1\delta}=\delta w_2
\end{equation}
for some infinitely smooth 0-form (scalar) $w_0$ and some infinitely smooth 2-form $w_2\,$. 

{\bf Step 8}.
Applying the operator $\delta$ to \eqref{lemma invariant subspaces extended curl equation Dima 7}
and using formulae
\eqref{lemma invariant subspaces extended curl equation Dima 10}--\eqref{lemma invariant subspaces extended curl equation Dima 12}
as well as the fact that $\delta z_1=0$, we get $\Delta(\lambda w_0-v_0)=0$, which implies
\begin{equation}
\label{lemma invariant subspaces extended curl equation Dima 13}
v_0=\lambda w_0+C,
\end{equation}
where $C\in\mathbb{R}$ is a constant.
Substituting
\eqref{lemma invariant subspaces extended curl equation Dima 10}
and
\eqref{lemma invariant subspaces extended curl equation Dima 13}
into \eqref{lemma invariant subspaces extended curl equation Dima 8}
and using formulae
\eqref{lemma invariant subspaces extended curl equation Dima 11},
\eqref{lemma invariant subspaces extended curl equation Dima 12},
we get
\begin{equation}
\label{lemma invariant subspaces extended curl equation Dima 14}
-\Delta w_0-\lambda^2 w_0=\lambda C.
\end{equation}
In view of
\eqref{lemma invariant subspaces extended curl equation Dima 2}
and
\eqref{lemma invariant subspaces extended curl equation Dima extra}
equation
\eqref{lemma invariant subspaces extended curl equation Dima 14}
has the unique solution $w_0=-\lambda^{-1}C$,
which immediately implies
\begin{equation}
\label{lemma invariant subspaces extended curl equation Dima 15}
v_0=0,\qquad v_{1\dr}=0.
\end{equation}

It only remains to observe that formulae
\eqref{lemma invariant subspaces extended curl equation Dima 6},
\eqref{lemma invariant subspaces extended curl equation Dima 10},
\eqref{lemma invariant subspaces extended curl equation Dima 12}
and
\eqref{lemma invariant subspaces extended curl equation Dima 15}
imply
$v\in\delta\Omega^2_\infty\oplus 0$,
thus proving the inclusion
\eqref{lemma invariant subspaces extended curl equation Dima 5}.
\end{proof}

\subsection{Decomposition of the spectrum of extended curl}
\label{Decomposition of the spectrum of extended curl}

Recall that extended curl \eqref{definition extended curl equation 1} is a self-adjoint elliptic operator, hence its spectrum is discrete and eigenfunctions infinitely smooth.

\begin{proposition}
\label{corollary decomposition of spectrum of extended curl}
The spectrum of $\curl_E$ decomposes into three parts as
\begin{equation}
\label{decomposition of spectrum of extended curl}
\sigma(\curl_E)
=
\sigma(\curl)
\cup
\sigma(\curl_{E,\dr})
\cup
\sigma(\curl_{E,\mathcal{H}}),
\end{equation}
where $\curl$, $\curl_{E,\dr}$ and $\curl_{E,\mathcal{H}}$ are the operators
\eqref{definition curl},
\eqref{lemma invariant subspaces extended curl equation 3}
and
\eqref{lemma invariant subspaces extended curl equation 5}
respectively.
Formula \eqref{decomposition of spectrum of extended curl} is understood as the (disjoint) union of spectra with account of multiplicities.
\end{proposition}

\begin{proof}
Proposition~\ref{corollary decomposition of spectrum of extended curl} is an immediate consequence of Proposition~\ref{lemma invariant subspaces extended curl} and the observation that the operator $\,\operatorname{curl}_{E,\delta}\,$ defined by formulae
\eqref{definition extended curl equation 1}
and
\eqref{lemma invariant subspaces extended curl equation 4}
is precisely our original operator $\curl$ defined by formula~\eqref{definition curl}.
\end{proof}

Proposition~\ref{corollary decomposition of spectrum of extended curl} tells us that the spectrum of curl sits inside the spectrum of extended curl and the two spectra differ by
$\sigma(\curl_{E,\dr})
\cup
\sigma(\curl_{E,\mathcal{H}})$.
So the task at hand is working out the spectra of  $\operatorname{curl}_{E,\dr}$ and $\operatorname{curl}_{E,\mathcal{H}}\,$. The following two lemmata address this issue.

\begin{lemma}
\label{spectrum zero}
We have $\sigma(\curl_{E,\mathcal{H}})=\{0\}$, zero being an eigenvalue of multiplicity $\operatorname{dim}\mathcal{H}^1+1$.
\end{lemma}

\begin{proof}
Lemma~\ref{spectrum zero} is an immediate consequence of formulae
\eqref{definition extended curl equation 1}
and
\eqref{lemma invariant subspaces extended curl equation 5}.
\end{proof}

\begin{lemma}
\label{spectrum plus minus square root of Laplacian}
We have
\begin{equation}
\label{spectrum plus minus square root of Laplacian Dima 0}
\sigma(\curl_{E,\dr})=\{\pm\sqrt{\mu_1}\,,\,\pm\sqrt{\mu_2}\,,\dots\},
\end{equation}
where the $\mu_j$ are the eigenvalues \eqref{eigenvalues of Laplace--Beltrami operator} of $\,-\Delta\,$.
\end{lemma}

\begin{proof}
Let
$v=
\begin{pmatrix}
\dr w_0\\
v_0
\end{pmatrix}
$
be an eigenfunction of $\curl_{E,\dr}$ corresponding to an eigenvalue $\lambda$.
Here $w_0$ and $v_0$ are some infinitely smooth 0-forms (scalars).
We have
\begin{align}
\label{spectrum plus minus square root of Laplacian Dima 1}
\dr v_0
&=
\lambda\dr w_0\,,
\\
\label{spectrum plus minus square root of Laplacian Dima 2}
-\Delta w_0
&=
\lambda v_0\,.
\end{align}
Elementary analysis of
\eqref{spectrum plus minus square root of Laplacian Dima 1}
and
\eqref{spectrum plus minus square root of Laplacian Dima 2}
yields
\eqref{spectrum plus minus square root of Laplacian Dima 0}.
\end{proof}

Proposition~\ref{corollary decomposition of spectrum of extended curl}
and Lemmata~\ref{spectrum zero},~\ref{spectrum plus minus square root of Laplacian}
tell us that the spectrum of $\curl$ can be recovered from the spectrum of $\curl_E\,$, provided one knows the spectrum of the Laplace--Beltrami operator $\Delta$. Moreover, the eta functions of $\curl$ and $\curl_E\,$ are the same, because contributions to $\sigma(\curl_E)$ coming from the Laplace--Beltrami operator are symmetric about zero.


\subsection{Proof of Theorem~\ref{theorem properties of curl}}

\begin{proof}[Proof of Theorem~\ref{theorem properties of curl}]
\

(a)
Standard elliptic theory tells us that the operator $\curl_E$ defined by formula \eqref{definition extended curl equation 1} is self-adjoint in the operator theoretic sense.
Proposition \ref{lemma invariant subspaces extended curl} and formulae
\eqref{definition extended curl equation 1}
and
\eqref{lemma invariant subspaces extended curl equation 4}
tell us that our original operator $\curl$ defined by formula~\eqref{definition curl}
is part of the operator $\curl_E\,$; here the concept of `part of an operator' is understood in accordance with
\cite[Definition~1.7]{schmudgen}.
Self-adjointness of $\curl$ now follows from \cite[Proposition~3.11]{schmudgen}.

(b)
Discreteness of the spectrum of $\curl$ follows from Proposition~\ref{corollary decomposition of spectrum of extended curl}. As to the statement that the spectrum of $\curl$ is unbounded both from above and from below, we prove it by writing down spectral asymptotics.

Let $\lambda_k$, $k\in\mathbb{Z}\setminus\{0\}$, be the eigenvalues of $\curl$. The choice of particular enumeration is irrelevant for our purposes, but what is important is that eigenvalues are enumerated with account of their multiplicity. We define the positive ($+$) and negative ($-$) counting functions for $\curl$ as
\begin{equation}
\label{global counting functions for curl}
N^\pm_{\curl}(\lambda)
:=
\begin{cases}
0\quad&\text{for}\quad\lambda\le0,
\\
\sum_{0<\pm\lambda_k<\lambda}
\,1
\quad&\text{for}\quad\lambda>0.
\end{cases}
\end{equation}
We define, in a similar fashion, counting functions $N^\pm_{\curl_E}(\lambda)$ and $N^\pm_{\curl_{E,d}}(\lambda)$ for the operators $\curl_E$ and $\curl_{E,d}$ respectively. Proposition~\ref{corollary decomposition of spectrum of extended curl} and Lemma~\ref{spectrum zero} tell us that
\begin{equation}
\label{difference of two couunting functions}
N^\pm_{\curl}(\lambda)
=
N^\pm_{\curl_E}(\lambda)
-
N^\pm_{\curl_{E,d}}(\lambda).
\end{equation}

Lemma~\ref{spectrum plus minus square root of Laplacian} tells us that the evaluation of $N^\pm_{\curl_{E,d}}(\lambda)$ reduces to the evaluation of eigenvalues of the Laplace--Beltrami operator. The latter is a subject which has been extensively studied over the last 100 years. Application of \cite[Theorem~1.1]{Hor 1968} gives the following asymptotic formula for $N^\pm_{\curl_{E,d}}(\lambda)$ with sharp remainder term estimate:
\begin{equation}
\label{asymptotic formula 1}
N^\pm_{\curl_{E,d}}(\lambda)
=
\frac{\operatorname{Vol}(M)}{6\pi^2}\lambda^3+O(\lambda^2)
\qquad
\text{as}
\qquad
\lambda\to+\infty,
\end{equation}
where $\operatorname{Vol}(M)$ is the Riemannian volume of the manifold.

Application of \cite[Theorem~0.1]{Ivr82} gives the following asymptotic formula for $N^\pm_{\curl_E}(\lambda)$:
\begin{equation}
\label{asymptotic formula 2}
N^\pm_{\curl_E}(\lambda)
=
\frac{\operatorname{Vol}(M)}{3\pi^2}\lambda^3+O(\lambda^2)
=
2N^\pm_{\curl_{E,d}}(\lambda)+O(\lambda^2)
\qquad
\text{as}
\qquad
\lambda\to+\infty.
\end{equation}
The factor 2 appears in the RHS of \eqref{asymptotic formula 2} because the eigenvalues of the principal symbol of the operator $\curl_E\,$, given by formula \eqref{eigenvalues principal symbol extended curl}, have multiplicity two.

Formulae \eqref{difference of two couunting functions}--\eqref{asymptotic formula 2}
imply B\"ar’s \cite[Theorem~3.6]{baer_curl} asymptotic formula
\begin{equation}
\label{asymptotic formula 3}
N^\pm_{\curl}(\lambda)
=
\frac{\operatorname{Vol}(M)}{6\pi^2}\lambda^3+O(\lambda^2)
\qquad
\text{as}
\qquad
\lambda\to+\infty.
\end{equation}
Formula \eqref{asymptotic formula 3} shows that the spectrum of $\operatorname{curl}$ accumulates to $+\infty$ and $-\infty$.

(c)
Suppose that zero is an eigenvalue of $\curl$. Let
\begin{equation}
\label{zero mode 1}
u\in\delta\Omega^2_\infty
\end{equation}
be the corresponding eigenform (zero mode). Then $\curl u=0$, which implies
\begin{equation}
\label{zero mode 2}
\dr u=0.
\end{equation}
But \eqref{zero mode 1} implies
\begin{equation}
\label{zero mode 3}
\delta u=0.
\end{equation}
Formulae \eqref{zero mode 2} and \eqref{zero mode 3} tell us that the 1-form $u$ is harmonic. As $\delta\Omega^2\cap\mathcal{H}^1=\{0\}$, we conclude that $u=0$, which is a contradiction.

(d)
Proposition~\ref{corollary decomposition of spectrum of extended curl},
Lemmata~\ref{spectrum zero} and \ref{spectrum plus minus square root of Laplacian}
and part (c) of Theorem~\ref{theorem properties of curl}
tell us that zero is an eigenvalue of $\curl_E$ of multiplicity $\operatorname{dim}\mathcal{H}^1+1$.
Let $P_{\mathcal{H}^1}$ and $P_{\mathcal{H}^0}$ be orthogonal projections onto the spaces of harmonic 1-forms and 0-forms respectively. Put
\begin{equation*}
\label{definition extended curl equation 1 invertible}
\tilde\curl_E
:=
\curl_E
+
\begin{pmatrix}
P_{\mathcal{H}^1}&0\\
0&P_{\mathcal{H}^0}
\end{pmatrix}.
\end{equation*}
The operator $\tilde\curl_E$ is invertible. By standard elliptic theory $\tilde\curl{}_E^{-1}$ is a bounded operator acting from the Sobolev space $H^s$ to the Sobolev space $H^{s+1}$. The operator $\curl^{-1}$ is part of the operator $\tilde\curl{}_E^{-1}\,$, hence it possesses the same mapping properties.
\end{proof}

%

\color{black}

\section{The spectrum of the Laplacian on a Berger sphere}
\label{The spectrum of the Laplacian on a Berger sphere}

In preparation for the calculations that will appear in Appendix~\ref{The spectrum of curl on a Berger sphere}, we list here, in the form of a theorem, the eigenvalues of the Laplacian on a Berger sphere. These were first obtained in \cite[Lemma~4.1]{tanno} and reproduced in \cite[Theorem~5.5]{lotay} and \cite[Proposition~3.9]{lauret}.

\subsection{Definitions and notation}
\label{subsec:Definitions and notation berger}

In Euclidean space $\mathbb{R}^4$ equipped with Cartesian coordinates $\mathbf{x}^\alpha$, $\alpha=1,2,3,4$, let
\[
\mathbb{S}^3:=\{\mathbf{x}\in \mathbb{R}^4 \ | \ \|\mathbf{x}\|=1\}
\]
be the $3$-sphere.
We use bold script for 4-dimensional objects and normal script for 3-dimensional objects.

We prescribe orientation of $\mathbb{S}^3$ as follows.
We define local coordinates $y=(y^1,y^2,y^3)$ on $\mathbb{S}^3$ to be positively oriented if
\begin{equation*}
\label{definition of orientatioin via embedding}
\det
\begin{pmatrix}
\frac{\partial\mathbf{x}^1}{\partial y^1}(y)&
\frac{\partial\mathbf{x}^1}{\partial y^2}(y)&
\frac{\partial\mathbf{x}^1}{\partial y^3}(y)&
\mathbf{x}^1(y)
\\
\frac{\partial\mathbf{x}^2}{\partial y^1}(y)&
\frac{\partial\mathbf{x}^2}{\partial y^2}(y)&
\frac{\partial\mathbf{x}^2}{\partial y^3}(y)&
\mathbf{x}^2(y)
\\
\frac{\partial\mathbf{x}^3}{\partial y^1}(y)&
\frac{\partial\mathbf{x}^3}{\partial y^2}(y)&
\frac{\partial\mathbf{x}^3}{\partial y^3}(y)&
\mathbf{x}^3(y)
\\
\frac{\partial\mathbf{x}^4}{\partial y^1}(y)&
\frac{\partial\mathbf{x}^4}{\partial y^2}(y)&
\frac{\partial\mathbf{x}^4}{\partial y^3}(y)&
\mathbf{x}^4(y)
\end{pmatrix}
<0\,.
\end{equation*}
The above definition of orientation agrees with that from \cite[Appendix~A]{sphere}.
Spherical coordinates
\begin{equation*}
\label{spherical coordinates for 3-sphere}
\begin{pmatrix}
\mathbf{x}^1\\
\mathbf{x}^2\\
\mathbf{x}^3\\
\mathbf{x}^4
\end{pmatrix}
=
\begin{pmatrix}
\cos y^1\\
\sin y^1\cos y^2\\
\sin y^1\sin y^2\cos y^3\\
\sin y^1\sin y^2\sin y^3\\
\end{pmatrix},
\quad
y^1,y^2\in(0,\pi),
\quad
y^3\in[0,2\pi),
\end{equation*}
are an example of positively oriented local coordinates on $\mathbb{S}^3$.

Of course, if one is only interested in the study of the Laplace--Beltrami operator, orientation is irrelevant. However, with $\curl$ or the Dirac operator \cite[subsection 9.1]{dirac} prescribing orientation of the manifold is essential.

Consider the vector fields
\begin{equation}
\label{berger equation 1}
\begin{aligned}
&\mathbf{V}_{1}:=
- \textbf{x}^4\dfrac{\partial}{\partial \textbf{x}^1}- \textbf{x}^3\frac{\partial}{\partial \textbf{x}^2} +\textbf{x}^2\dfrac{\partial}{\partial \textbf{x}^3}+\textbf{x}^1\dfrac{\partial}{\partial \textbf{x}^4}\,,\\
&\mathbf{V}_{2}:=
\phantom{-} \textbf{x}^3\dfrac{\partial}{\partial \textbf{x}^1}-\textbf{x}^4\frac{\partial}{\partial \textbf{x}^2} -\textbf{x}^1 \dfrac{\partial}{\partial \textbf{x}^3}+\textbf{x}^2\dfrac{\partial}{\partial \textbf{x}^4}\,,\\
&\mathbf{V}_{3}:=
- \textbf{x}^2\dfrac{\partial}{\partial \textbf{x}^1}+ \textbf{x}^1\frac{\partial}{\partial \textbf{x}^2} -\textbf{x}^4\dfrac{\partial}{\partial \textbf{x}^3}+\textbf{x}^3\dfrac{\partial}{\partial \textbf{x}^4}\,.\\
\end{aligned}
\end{equation}
The vector fields \eqref{berger equation 1} are tangent to $\mathbb{S}^3$, i.e., they are orthogonal to the gradient of $\|\textbf{x}\|$, or, equivalently, $\mathbf{V}_{j} (\|\textbf{x}\|^2)=0$, $j=1,2,3$. For $j=1,2,3$ let $V_j$ be the restriction of $\mathbf{V}_{j}$ to $\mathbb{S}^3$.

Let $a\in \mathbb{R}$ be a positive parameter. Consider the differential operator
\begin{equation}
\label{berger equation 2}
B:=V_1^2+V_2^2+\frac{1}{a^2}V_3^2\,.
\end{equation}
The \emph{Berger metric} $g$ on $\mathbb{S}^3$ is defined via the identity
\begin{equation}
\label{berger equation 3}
B_\prin(y,\eta)=-g^{\alpha\beta}(y)\eta_\alpha\eta_\beta
\end{equation}
or, equivalently,
\begin{equation*}
\label{berger equation 4}
g^{\alpha\beta}(y):=-\frac{\partial^2 B_\mathrm{prin}}{\partial \eta_\alpha\partial\eta_\beta}\,.
\end{equation*}
We call $(\mathbb{S}^3,g)$ the \emph{Berger sphere}.

The notion of a Berger sphere is well established in differential geometry, nevertheless, for the benefit of a wider audience, we provide here an explicit description in the simplest possible coordinate system.
Working in the southern hemisphere $\mathbf{x}^4<0$, consider the positively oriented local coordinates
\begin{equation*}
\label{berger equation 5}
y^\alpha=\mathbf{x}^\alpha,\qquad  \alpha=1,2,3,
\end{equation*}
so that
\begin{equation*}
\label{berger equation 6}
\mathbf{x}^4
=
-\sqrt{1-(\mathbf{x}^1)^2-(\mathbf{x}^2)^2-(\mathbf{x}^3)^2}
=
-\sqrt{1-|y|^2}.
\end{equation*}
In these coordinates the vector fields $V_j$ read
\begin{equation}
\label{berger equation 9}
\begin{split}
V_1
&
=
\sqrt{1-|y|^2}
\dfrac{\partial}{\partial y^1}
-
y^3
\dfrac{\partial}{\partial y^2}
+
y^2
\dfrac{\partial}{\partial y^3},
\\
V_2
&
=
y^3
\dfrac{\partial}{\partial y^1}
+
\sqrt{1-|y|^2}
\dfrac{\partial}{\partial y^2}
-
y^1
\dfrac{\partial}{\partial y^3},
\\
V_3
&
=
-y^2
\dfrac{\partial}{\partial y^1}
+
y^1
\dfrac{\partial}{\partial y^2}
+
\sqrt{1-|y|^2}
\dfrac{\partial}{\partial y^3}.
\end{split}
\end{equation}
Formulae
\eqref{berger equation 2},
\eqref{berger equation 3}
and
\eqref{berger equation 9} imply
\begin{equation}
\label{berger equation 10}
\footnotesize
g^{\alpha\beta}(y)=
\begin{pmatrix}
1-(y^1)^2-\frac{a^2-1}{a^2}(y^2)^2& - \frac{y^1 y^2}{a^2} &- y^1 y^3+\frac{a^2-1}{a^2}y^2\sqrt{1-|y|^2}
\\
- \frac{y^1 y^2}{a^2} &1-(y^2)^2-\frac{a^2-1}{a^2}(y^1)^2& - y^2 y^3-\frac{a^2-1}{a^2}y^1\sqrt{1-|y|^2}
\\
-y^1 y^3+\frac{a^2-1}{a^2}y^2\sqrt{1-|y|^2} & - y^2 y^3-\frac{a^2-1}{a^2}y^1\sqrt{1-|y|^2} &\frac{a^2-1}{a^2}(y^{\color{black}1\color{black}})^2+\frac{a^2-1}{a^2}(y^2)^2\color{black}+\color{black}\frac{1-(y^3)^2}{a^2}\\
\end{pmatrix}.
\end{equation}

We summarise the properties of the vector fields $V_j$, $j=1,2,3$, and the Berger metric $g$ in the following Lemma. The proof amounts to straightforward calculations and is omitted.

\begin{lemma}
We have the following properties.
\begin{enumerate}[(a)]
\item
$[V_j,V_k]=-2 \varepsilon_{jkl}\,V_l\,$,
$j,k=1,2,3$.

\item
The vector fields $V_j$ are positively oriented, i.e.~$\det V_j{}^\alpha>0$, $j,\alpha=1,2,3$.
Here the $V_j{}^\alpha$ are the components of the vector fields $V_j$ with respect to the basis
$\frac{\partial}{\partial y^\alpha}\,$.

\item
The vector fields
\begin{equation}
\label{orthonormalised vector fields}
\tilde V_1:=V_1\,,
\qquad
\tilde V_2:=V_2\,,
\qquad
\tilde V_3:=\frac{1}{a}V_3
\end{equation}
are orthonormal with respect to the metric $g$ defined by \eqref{berger equation 10}.

\item
The Riemannian density $\rho(y)$ generated by the Berger metric $g$ reads
\[
\rho(y)=a\,\rho_0(y),
\]
where $\rho_0(y)$ is the standard Riemannian density of the round $\mathbb{S}^3$.

\item
The vector fields $V_j$, $j=1,2,3$, viewed as differential operators are skew-Hermitian.

\item
The vector fields $V_j$, $j=1,2,3$, are Killing vector fields for the standard round metric $g_0$ on $\mathbb{S}^3$.

\item
The vector field $V_3$ is a Killing vector field for the Berger metric $g$.

\item
The operator $B$ defined in accordance with \eqref{berger equation 2} is the Laplace--Beltrami operator (henceforth denoted by $\Delta$) on the Berger sphere.

\item
The operator $\curl$ on the Berger sphere can be represented as
\begin{equation}
\label{berger equation 11}
\operatorname{curl}_s
=
\begin{pmatrix}
\frac{2}{a} &-\frac1{a}V_3  & V_2\\
\frac1{a}V_3& \frac{2}{a} & - V_1\\
-V_2 & V_1& 2a \\
\end{pmatrix}.
\end{equation}
Here the subscript $\,s\,$ indicates that the operator appearing in \eqref{berger equation 11} acts on 3-columns of scalar functions. The latter can be obtained by projecting 1-forms onto the
orthonormal framing \eqref{orthonormalised vector fields}.
The operators $\curl$ and $\curl_s$ are related as
\[
\curl_\alpha{}^\beta
=
\tilde V^j{}_\alpha\,[\curl_s]_j{}^k\,\tilde V_k{}^\beta,
\qquad
\tilde V^j{}_\alpha:=\delta^{jl}\,\tilde V_l{}^\gamma\,g_{\alpha\gamma}\,.
\]
See also the projection operator $\mathbf{S}$ in \cite[Section 5.1]{obstructions}.
\end{enumerate}
\end{lemma}

\subsection{The spectrum}

\begin{theorem}
\label{theorem spectrum berger laplacian}
The eigenvalues of the operator $-\Delta$ on the Berger 3-sphere are
\begin{equation}
\label{eigenvalue laplacian berger formula 1}
n(n+2)
+
\left( 
a^{-2}-1
\right)
\left(
n-2l
\right)^2
\,,
\quad
n=0,1,2,3,\dots,
\quad
0
\le
l
\le
\left\lfloor
\frac{n}{2}
\right\rfloor
\end{equation}
with
\begin{enumerate}[(i)]

\item multiplicity 
\[
1
\]
if $n=0$,

\item multiplicity
\[
2n+2
\]
if $n$ is odd,

\item multiplicity
\[
2n+2
\]
if $n$ is even and $0\le l< \frac{n}{2}$, and 

\item multiplicity
\[
n+1
\]
if $n$ is even and $l =\frac{n}{2}$.
\end{enumerate}
\end{theorem}

For a proof of the above theorem we refer the reader to
\cite[Theorem~5.5]{lotay}. Somewhat more involuted proofs may also be found in \cite[Lemma~4.1]{tanno} and \cite[Proposition~3.9]{lauret}.

\begin{remark}
\label{remark berger spectrum of Laplacian}
When $a=1$ Theorem~\ref{theorem spectrum berger laplacian} reduces to the usual spectrum of $\,-\Delta\,$ on the round 3-sphere: eigenvalues $n(n+2)$, $n=0,1,2,\dots$, with multiplicity $(n+1)^2\,$.
\end{remark}

%
%

\section{The spectrum of curl on a Berger sphere}
\label{The spectrum of curl on a Berger sphere}

\textcolor{black}{From a perturbative perspective, and in analogy with what happens for the massless Dirac operator, one expects the spectrum of curl to be asymmetric about zero for an open dense subset of the space of metrics on a given Riemannian 3-manifold --- put differently, spectral asymmetry of curl is a generic condition. However, we were unable to identify a precise reference for this fact in the literature.}

In this appendix we will write down explicitly the eigenvalues and the $\eta$-invariant for the operator curl on a Berger sphere\textcolor{black}{, using classical (as opposed to pseudodifferential) techniques. On the one hand, this will serve as an illustration of the spectral asymmetry of curl. On the other hand, it shall serve as a useful reference for further explicit calculations in future work}.

\

Throughout this appendix we use the notation from Appendix~\ref{subsec:Definitions and notation berger}.

\subsection{The spectrum}

The spectrum of curl \eqref{berger equation 11} was first computed by Gibbons \cite[Section~5]{gibbons}. The positive spectrum of curl \eqref{berger equation 11} also features in \cite[Proposition~5.6]{lotay}. We give below, for future reference and in the form of a theorem, explicit formulae for the \emph{full} spectrum of curl, in a slightly reformulated form. The proof is omitted.

\begin{theorem}
\label{theorem spectrum of curl}
The spectrum of the operator $\curl$ on the Berger 3-sphere is the (disjoint) union of the following four series of eigenvalues:
\begin{enumerate}[I.]
\item
Eigenvalues
\begin{equation}
\label{20 August 2021 equation 1}
\frac{n}{a}\,,\qquad
n=2,3,\dots,
\end{equation}
with multiplicity
\begin{equation}
\label{20 August 2021 equation 2}
2n-2.
\end{equation}
\item
Eigenvalues
\begin{equation}
\label{20 August 2021 equation 3}
\frac{n+2(a^2-1)}{a}\,,\qquad
n=2,3,\dots.
\end{equation}
Here the multiplicity is as follows.
\begin{enumerate}[(a)]
\item
If $n=2$ the multiplicity is
\begin{equation}
\label{20 August 2021 equation 4}
1.
\end{equation}
\item
If $n=3,4,\dots$ the multiplicity is
\begin{equation}
\label{20 August 2021 equation 5}
2n-2.
\end{equation}
\end{enumerate}
\item
Eigenvalues
\begin{equation}
\label{20 August 2021 equation 6}
a
+
\sqrt{
a^2
+
n(n+2)
+
\left( 
a^{-2}-1
\right)
\left(
n-2l
\right)^2
}\ ,
\quad
n=2,3,\dots,
\quad
1
\le
l
\le
\left\lfloor
\frac{n}{2}
\right\rfloor.
\end{equation}
Here the multiplicity is as follows.
\begin{enumerate}[(a)]
\item
If $n$ is odd the multiplicity is
\begin{equation}
\label{20 August 2021 equation 7}
2n+2.
\end{equation}
\item
If $n$ is even and $l<\frac{n}{2}$ the multiplicity is
\begin{equation}
\label{20 August 2021 equation 8}
2n+2.
\end{equation}
\item
If $n$ is even and $l=\frac{n}{2}$ the multiplicity is
\begin{equation}
\label{20 August 2021 equation 9}
n+1.
\end{equation}
\end{enumerate}
\item
Eigenvalues
\begin{equation}
\label{20 August 2021 equation 10}
a
-
\sqrt{
a^2
+
n(n+2)
+
\left( 
a^{-2}-1
\right)
\left(
n-2l
\right)^2
}\ ,
\quad
n=2,3,\dots,
\quad
1
\le
l
\le
\left\lfloor
\frac{n}{2}
\right\rfloor,
\end{equation}
with multiplicities as in \eqref{20 August 2021 equation 7}--\eqref{20 August 2021 equation 9}.
\end{enumerate}
\end{theorem}

\begin{remark}
\label{remark berger spectrum of curl}
\phantom{?}
\begin{enumerate}[(i)]
\item
Series I, II and III form the positive spectrum, whereas series IV is the negative spectrum.
\item
For $a=1$ the Berger sphere reduces to the standard 3-sphere. It is not hard to check that for $a=1$ Theorem~\ref{theorem spectrum of curl} gives us eigenvalues
\begin{equation}
\label{spectrum curl 3 sphere eigenvalues}
\pm n, \qquad n=2,3,\dots,
\end{equation}
with multiplicity 
\begin{equation}
\label{spectrum curl 3 sphere multiplicity}
n^2-1.
\end{equation}
In particular, the spectrum is symmetric about $0$. 
\item Formulae 
\eqref{spectrum curl 3 sphere eigenvalues}
and
\eqref{spectrum curl 3 sphere multiplicity}
agree with \cite[Theorem~5.2]{baer_curl}.

\item Note that in \eqref{20 August 2021 equation 6} the lowest value the integer $l$ can take is $1$. The case $l=0$ is special: in this case the square root becomes superfluous and the expression \eqref{20 August 2021 equation 6} turns into \eqref{20 August 2021 equation 3}. We have highlighted the special status of the case $l=0$ by separating out the series of eigenvalues \eqref{20 August 2021 equation 3}.

\item
Observe that the expression \eqref{eigenvalue laplacian berger formula 1} appears under the square root in formulae
\eqref{20 August 2021 equation 6}
and
\eqref{20 August 2021 equation 10}.
\end{enumerate}
\end{remark}

\subsection{The eta invariant}

In this subsection we will compute the eta invariant for the operator $\curl$ on a Berger 3-sphere directly, using the explicit formulae for the eigenvalues established  above in Theorem~\ref{theorem spectrum of curl}. Recall that the eta-function is defined, in general, in accordance with \eqref{eta function intro}.

\begin{theorem}
\label{theorem eta invariant berger}
Consider the operator $\curl$ on a Berger 3-sphere, defined as in Appendix~\ref{subsec:Definitions and notation berger}. We have
\begin{equation}
\label{theorem eta invariant berger equation 1}
\eta_{\curl}(0)=\frac23 (a^2-1)^2.
\end{equation}
\end{theorem}

\begin{remark}
The above theorem warrants a number of remarks.
\begin{enumerate}[(i)]
\item
The eta invariant for the Dirac operator on a Berger 3-sphere was computed by Hitchin \cite[p.~34]{hitchin} and reads
\begin{equation}
\label{eta invariant Dirac}
\eta_{\mathrm{Dirac}}(0)=-\frac{1}{6}(a^2-1)^2\,.
\end{equation}
Note that the parameter $\lambda$ in \cite{hitchin} corresponds to our parameter $a$, and that Hitchin adopts the same (or rather, equivalent) conventions as we do when defining the Dirac operator. The latter can be seen if one compares the spectra of the Dirac operator in \cite{hitchin} and \cite{sphere}: Hitchin's \cite[Proposition 3.2]{hitchin} with $p=q=1$ and minus sign in front of the square root agrees with \cite[formulae (6.5) and (6.8)]{sphere}.

The result \eqref{eta invariant Dirac} was obtained under the assumption $a<4$. We do not need to impose any restriction on $a$ because, unlike for Dirac operator, where the dimension of the space of harmonic spinors is highly dependent on the Riemannian metric, zero is never an eigenvalue of curl (see Theorem~\ref{theorem properties of curl}).

\item
For a general connected oriented closed Riemannian 3-manifold the invariants $\eta_{\curl}(0)$ and  $\eta_{\mathrm{Dirac}}(0)$ are related as
\begin{equation}
\label{relation between eta curl and eta Dirac}
\eta_{\curl}(0)= -4\eta_{\mathrm{Dirac}}(0)+ k
\end{equation}
for some $k\in \mathbb{Z}$. This follows from abstract algebraic topology arguments involving Hirzebruch polynomials, see \cite[Section~1]{asymm1} and \cite[Theorems 4.2 and 4.14]{asymm2}. Theorem~\ref{theorem eta invariant berger} establishes that, in the case of a Berger 3-sphere, we have \eqref{relation between eta curl and eta Dirac} with $k=0$ for $a<4$. The significance of the value $a=4$ is that for this particular value of the parameter $a$ the Dirac operator has a zero mode (harmonic spinor).

\item
The properties of the Dirac operator on a Berger 3-sphere, including an alternative (to Hitchin's \cite{hitchin}) derivation of the spectrum, were further studied by B\"ar in \cite{baer92}. Note that B\"ar adopts the opposite convention when defining the Dirac operator, so that the eigenvalues of the Dirac operator in \cite{baer92} differ from those in \cite{hitchin,sphere} by sign.

\item
Formula \eqref{theorem eta invariant berger equation 1} appears in \cite[Problem~7]{dowker84}.
\end{enumerate}
\end{remark}

In preparation for the proof of Theorem~\ref{theorem eta invariant berger}, let us make the following observation.

\begin{lemma}
The function $\eta_{\curl}$ admits the representation
\begin{equation}
\label{eta curl dima 1}
\eta_{\curl}(s)
=
\theta(s)
+
(2a)^{-s}+4a^s \zeta(s-1)\,,
\end{equation}
where
\begin{equation}
\label{eta curl dima 2}
\theta(s)
:=
\sum_{j=1}^\infty
\left[
\left(
\sqrt{a^2+\mu_j}
\,+\,a
\right)^{-s}
-
\left(
\sqrt{a^2+\mu_j}
\,-\,a
\right)^{-s}
\right],
\end{equation}
the $\mu_j$ are the eigenvalues \eqref{eigenvalues of Laplace--Beltrami operator} of $\,-\Delta\,$ enumerated in increasing order with account of multiplicities, and $\,\zeta(s):=\sum_{n=1}^\infty n^{-s}\,$ is the Riemann zeta function.
\end{lemma}

\begin{proof}
Formulae \eqref{eta curl dima 1} and \eqref{eta curl dima 2} are a straightforward consequence of
Theorems~\ref{theorem spectrum berger laplacian}~and~\ref{theorem spectrum of curl}.
%
\end{proof}

\begin{definition}
\label{Minakshisundaram--Pleijel zeta function definitiion}
We define
\begin{equation*}
\label{Minakshisundaram--Pleijel zeta function definitiion equation 1}
\zeta_{-\Delta}(s)
:=
\sum_{j=1}^\infty\mu_j^{-s}
\end{equation*}
to be the zeta function\footnote{The function $s\mapsto \zeta_{-\Delta}(s)$ is often called the \emph{Minakshisundaram--Pleijel zeta function}, see~\cite[equation~(6)]{minak}.} associated with (minus) the Laplace--Beltrami operator $-\Delta$. Similarly, we define
\begin{equation*}
\label{Minakshisundaram--Pleijel zeta function definitiion equation 2}
\zeta_{\sqrt{-\Delta}}(s)
:=
\sum_{j=1}^\infty\mu_j^{-s/2}
=
\zeta_{-\Delta}(s/2)\,.
\end{equation*}
\end{definition}

It is known \cite[p.~243]{minak} that the function $\zeta_{\sqrt{-\Delta}}(s)$ is meromorphic in $\mathbb{C}$ with simple poles at $s=3,1,-1,\dots$.

\begin{lemma}
\label{new lemma zeta function laplacian}
For a closed Riemannian 3-manifold $(M,g)$ we have
\begin{equation}
\label{new lemma zeta function laplacian equation 1}
\zeta_{\sqrt{-\Delta}}(s)=\frac{1}{2\pi^2}\left[\frac{\operatorname{Vol}(M)}{s-3}+\frac{\int_M\Sc(x)\,\rho(x)\,\dr x}{12(s-1)}+\ldots \right].
\end{equation}
\end{lemma}

\begin{proof}
Let $N_{\sqrt{-\Delta}}(\lambda)$ be the counting function for the positive eigenvalues of $\sqrt{-\Delta}$, see, e.g., \cite[Appendix~B]{wave}. The functions $\zeta_{\sqrt{-\Delta}}$ and $N_{\sqrt{-\Delta}}$ are related as
\begin{equation}
\label{proof new lemma zeta function laplacian equation 1}
\zeta_{\sqrt{-\Delta}}(s)=\int_{0}^{+\infty} \lambda^{-s}\, N'_{\sqrt{-\Delta}}(\lambda)\, \dr\lambda\,,
\end{equation}
where the prime stands for differentiation with respect to $\lambda$.
By combining \cite[equation~(B.1) and Theorem~B.2]{wave} with formula \eqref{proof new lemma zeta function laplacian equation 1}, we obtain \eqref{new lemma zeta function laplacian equation 1}.
\end{proof}

Since for the Berger 3-sphere $\operatorname{Vol}(M)=2\pi^2a$ and $\Sc=8-2a^2$, we have the following.

\begin{corollary}
\label{corollary zeta function laplacian}
For the Berger 3-sphere
\begin{equation}
\label{corollary zeta function laplacian equation 1}
\zeta_{\sqrt{-\Delta}}(s)=a\left[\frac{1}{s-3}+\frac{4-a^2}{6(s-1)}+\ldots \right].
\end{equation}
\end{corollary}

We are now in a position to prove Theorem~\ref{theorem eta invariant berger}.

\begin{proof}[Proof of Theorem~\ref{theorem eta invariant berger}]
Since $\,\zeta(-1)=-\frac1{12}\,$, formula \eqref{eta curl dima 1} implies that the proof of \eqref{theorem eta invariant berger equation 1} reduces to the proof of
\begin{equation}
\label{proof of eta invariant berger curl dima 1}
\theta(0)
=
\frac23\,
a^2(a^2-2)\,,
\end{equation}
where $\theta$ is the function \eqref{eta curl dima 2}.
Formulae \eqref{proof of eta invariant berger curl dima 1} and \eqref{eta curl dima 2} reduce the analysis of the operator $\,\curl\,$ to the analysis of the Laplace--Beltrami operator.

To establish \eqref{proof of eta invariant berger curl dima 1} we use a trick due to Hitchin \cite[p.~34]{hitchin}. We expand the terms in the RHS of \eqref{eta curl dima 2} in inverse powers of $\mu_j$ as
\begin{multline}
\label{proof of eta invariant berger curl dima 2}
\Bigl(
\sqrt{a^2+\mu_j}
\,\pm\,a
\Bigr)^{-s}
=
\mu_j^{-s/2}
\Bigl(
\sqrt{1+a^2\mu_j^{-1}}
\,\pm\,a\mu_j^{-1/2}
\Bigr)^{-s}
\\
=
\mu_j^{-s/2}
\Bigl(
1
\pm
a\mu_j^{-1/2}
+
\frac{1}{2}a^2\mu_j^{-1}
+
O\left(\mu_j^{-2}\right)
\Bigr)^{-s}
\\
=
\mu_j^{-s/2}
\Bigl(
1
-
s
\Bigl[
\pm
a\mu_j^{-1/2}
+
\frac{1}{2}a^2\mu_j^{-1}
\Bigr]
+
\frac{s(s+1)}{2}
\Bigl[
a^2\mu_j^{-1}
\pm
a^3\mu_j^{-3/2}
\Bigr]
\mp
\frac{s(s+1)(s+2)}{6}
a^3\mu_j^{-3/2}
+
O\left(\mu_j^{-2}\right)
\Bigr)
\\
=
\mu_j^{-s/2}
\Bigl(
1
\mp
sa\mu_j^{-1/2}
+
\frac{s^2}{2}
a^2\mu_j^{-1}
\pm
\frac{s(1-s^2)}{6}
a^3\mu_j^{-3/2}
+
O\left(\mu_j^{-2}\right)
\Bigr).
\end{multline}
Substituting \eqref{proof of eta invariant berger curl dima 2} into \eqref{eta curl dima 2} we get
\begin{equation}
\label{proof of eta invariant berger curl dima 3}
\theta(s)
=
-\,2sa\,\zeta_{\sqrt{-\Delta}}(s+1)
\,+\,\frac{s(1-s^2)}{3}\,a^3\,\zeta_{\sqrt{-\Delta}}(s+3)
\,+\,h(s)\,,
\end{equation}
where $h$ is analytic at $s=0$ and
\begin{equation}
\label{proof of eta invariant berger curl dima 4}
h(0)=0\,.
\end{equation}
Formulae
\eqref{proof of eta invariant berger curl dima 3},
\eqref{corollary zeta function laplacian equation 1}
and
\eqref{proof of eta invariant berger curl dima 4}
imply
\begin{equation*}
\label{proof of eta invariant berger curl dima 5}
\lim_{s\to0}\theta(s)
=
-\,2a^2\,\frac{4-a^2}{6}
\,+\,\frac{1}{3}\,a^4
=
\frac23\,
a^2(a^2-2)\,,
\end{equation*}
which proves \eqref{proof of eta invariant berger curl dima 1}.
\end{proof}

\section{Maxwell's equations: electromagnetic chirality and polarisation}
\label{Maxwell's equations and electromagnetic chirality}

The focus of our paper is the study of asymmetry between positive and negative eigenvalues of the operator $\curl$. It is natural to ask the question: does the sign of an eigenvalue of $\curl$ have a physical meaning? We show in this appendix that it does \textcolor{black}{by means of a rather elementary argument}. The key here \textcolor{black}{are the notions of \emph{electromagnetic chirality} and \emph{polarisation}, see Definitions~\ref{definition of chirality} and~\ref{definition of polarised harmonic solution}}. \textcolor{black}{Let us emphasise that the physical interpretation we provide in this appendix is not new: it has appeared in the literature in various related, if more abstract and `higher-tech', formulations --- see, e.g., \cite[Sec.~2]{etnyre},  \cite[Sec.~1.1.2]{gerner} and references therein.}

Let $(M,g)$ be a connected oriented closed Riemannian 3-manifold. Consider homogeneous vacuum Maxwell equations on $M\times\mathbb{R}$
\begin{equation}
\label{Maxwell 1}
\begin{pmatrix}
0&\curl\\
-\curl&0
\end{pmatrix}
\begin{pmatrix}
E\\
B
\end{pmatrix}
=
\frac{\partial}{\partial t}
\begin{pmatrix}
E\\
B
\end{pmatrix},
\end{equation}
\begin{equation}
\label{Maxwell 2}
\delta E=\delta B=0,
\end{equation}
where $E$ (electric field) and $B$ (magnetic field) are the unknown quantities, time-dependent real-valued 1-forms.

We list below two basic symmetries of Maxwell's equations \eqref{Maxwell 1}, \eqref{Maxwell 2}. All  the results in this appendix are given without proof because the corresponding proofs are elementary.

\begin{definition}
\label{definition of duality transform}
Given a pair of real-valued 1-forms $E$ and $B$, we define the \emph{duality transform} as the linear map
\begin{equation}
\label{definition of duality transform equation}
\operatorname{dual}:
\begin{pmatrix}
E\\
B
\end{pmatrix}
\mapsto
\begin{pmatrix}
-B\\
E
\end{pmatrix}.
\end{equation}
\end{definition}

The duality transform can be interpreted as the action of the 4-dimensional Lorentzian Hodge star
on the electromagnetic tensor
\begin{equation*}
F_{\alpha\beta}
=
\begin{pmatrix}
0&E_1&E_2&E_3\\
-E_1&0&-B_3&B_2\\
-E_2&B_3&0&-B_1\\
-E_3&-B_2&B_1&0
\end{pmatrix}.
\end{equation*}
Observe that the duality transform is a skew-involution, i.e.~applying it twice gives minus identity.

\begin{lemma}
\label{lemma about duality transform}
Maxwell's equations \eqref{Maxwell 1}, \eqref{Maxwell 2} are invariant under the action of the duality transform \eqref{definition of duality transform equation}, i.e.~the duality transform maps solutions to solutions.
\end{lemma}

\begin{definition}[{\cite{{tang and cohen},lipkin}}]
\label{definition of chirality}
Given a pair of real-valued 1-forms $E$ and $B$, we define \emph{electromagnetic chirality} as the quadratic functional
\begin{equation}
\label{definition of chirality equation}
\operatorname{chir}(E,B)
:=
\int_M(E\wedge\dr E+B\wedge\dr B).
\end{equation}
\end{definition}

\begin{lemma}
\label{lemma about chirality}
Chirality \eqref{definition of chirality equation} is a conserved quantity for Maxwell's equations \eqref{Maxwell 1}, \eqref{Maxwell 2}, i.e.~if $(E,B)$ is a solution then $\operatorname{chir}(E,B)$ does not depend on time $t$.
\end{lemma}

We will be seeking harmonic (in the time variable $t$) solutions of Maxwell's equations. 

\begin{definition}
\label{definition of polarised harmonic solution}
A harmonic solution $(E,B)$ of Maxwell's equations \eqref{Maxwell 1}, \eqref{Maxwell 2} is said to be \emph{polarised} if $\,(\dot E,\dot B)=\pm\lambda\operatorname{dual}(E,B)\,$. Here the dot stands for derivative in $t$ and $\lambda$ is the angular frequency.
\end{definition}

Recall that $(\lambda_j, u_j)$, $j=\pm1, \pm2, \ldots$, is the eigensystem for $\curl$.
\textcolor{black}{Consider a harmonic solution
\begin{equation}
\label{harmonic solution 0}
\begin{pmatrix}
E\\
B
\end{pmatrix}
=
\begin{pmatrix}
u_j\cos(\lambda_jt+\varphi)\\
-u_j\sin(\lambda_jt+\varphi)
\end{pmatrix}
\end{equation}
of Maxwell's equations \eqref{Maxwell 1}, \eqref{Maxwell 2}, where $\varphi\in\mathbb{R}$ is an arbitrary phase. Of course, \eqref{harmonic solution 0} is a linear combination of two basic harmonic solutions}
\begin{equation}
\label{harmonic solution 1}
\begin{pmatrix}
E\\
B
\end{pmatrix}
=
\begin{pmatrix}
u_j\cos(\lambda_jt)\\
-u_j\sin(\lambda_jt)
\end{pmatrix}
\end{equation}
and
\begin{equation}
\label{harmonic solution 2}
\begin{pmatrix}
E\\
B
\end{pmatrix}
=
\begin{pmatrix}
u_j\sin(\lambda_jt)\\
u_j\cos(\lambda_jt)
\end{pmatrix}.
\end{equation}

Let $E_0$ and $B_0$ be a pair of infinitely smooth real-valued 1-forms which do not depend on $t$ and satisfy the condition $\delta E_0=\delta B_0=0$.
Consider the Cauchy problem
\begin{equation*}
\label{Cauchy problem}
\left.E\right|_{t=0}=E_0,
\qquad
\left.B\right|_{t=0}=B_0
\end{equation*}
for Maxwell's equations \eqref{Maxwell 1}, \eqref{Maxwell 2}. \textcolor{black}{Then,}
\eqref{harmonic solution 1}
and
\eqref{harmonic solution 2}
can be used to write the solution as a series
\begin{equation}
\label{solution to Maxwell as a series}
\begin{pmatrix}
E\\
B
\end{pmatrix}
=
\sum_{j\in\mathbb{Z}\setminus\{0\}}
\left[
\begin{pmatrix}
u_j\cos(\lambda_jt)\\
-u_j\sin(\lambda_jt)
\end{pmatrix}
\langle E_0\,,u_j\rangle
+
\begin{pmatrix}
u_j\sin(\lambda_jt)\\
u_j\cos(\lambda_jt)
\end{pmatrix}
\langle B_0\,,u_j\rangle
\right]
+
\begin{pmatrix}
P_{\mathcal{H}^1}E_0\\
P_{\mathcal{H}^1}B_0
\end{pmatrix},
\end{equation}
where $P_{\mathcal{H}^1}$ is the orthogonal projection onto the space of harmonic 1-forms and
$\langle\,\cdot\ ,\,\cdot\,\rangle$ is the inner product \eqref{inner product}. 

The following proposition is the main result of this appendix.

\begin{proposition}
\label{proposition about chirality}
For \eqref{harmonic solution 0} we have
\begin{equation}
\label{proposition about chirality equation}
\operatorname{chir}(E,B)=\lambda_j\,,
\qquad
\color{black}
(\dot E,\dot B)=-\lambda_j\operatorname{dual}(E,B).
\color{black}
\end{equation}
\end{proposition}

Formula \eqref{proposition about chirality equation} tells us that the sign of an eigenvalue of $\curl$ has \textcolor{black}{a physical meaning and it can be interpreted in two different ways --- in terms of chirality of the corresponding harmonic solution \eqref{harmonic solution 0} of Maxwell's equations or in terms of its polarisation. Note that the concepts of chirality and polarisation are different: polarisation involves time, whereas chirality does not.}

\section{Taylor expansions for the operator of parallel transport}
\label{appendix parallel transport}

This appendix is concerned with local expansions of parallel transport maps in normal coordinates. These can, in principle, be found in many guises and with varied degrees of accuracy in the literature. Nevertheless, for the convenience of the reader and for future reference, we provide here a concise self-contained derivation, one which agrees with the definitions and sign conventions set out in this paper.

\

Let us fix a point $x\in M$ and let us choose normal coordinates centred at $x$. Let $y\in M$ be a point in a (small) neighbourhood of $x$.

Recall that $Z_\alpha{}^\beta(x,y)$ is the map realising the parallel transport of vectors along the unique geodesic connecting $x$ to $y$, whereas $Z^\alpha{}_\beta(x,y)$ is the map realising the parallel transport of covectors along the same geodesic, see~\eqref{Q acting on 1-forms 3}--\eqref{Q acting on 1-forms 3a}.

\begin{proposition}
\label{prop: expansion of parallel transoport}
We have
\begin{equation}
\label{prop: expansion of parallel transoport equation 1}
Z_\alpha{}^\beta(0,y)=\delta_\alpha{}^\beta+\frac16 \Riem^\beta{}_{\mu\alpha\nu}(0)\,y^\mu y^\nu -\frac16 \frac{\partial^2\Gamma^\beta{}_{\mu\alpha}}{\partial y^\nu \partial y^\rho}(0) \,y^\mu y^\nu y^\rho+O(|y|^4)\,
\end{equation}
and
\begin{equation}
\label{prop: expansion of parallel transoport equation 2}
Z^\alpha{}_\beta(0,y)=\delta^\alpha{}_\beta-\frac16 \Riem^\alpha{}_{\mu\beta\nu}(0)\, y^\mu y^\nu
+\frac16 \frac{\partial^2\Gamma^\alpha{}_{\mu\beta}}{\partial y^\nu \partial y^\rho}(0)\,y^\mu y^\nu y^\rho+O(|y|^4)\,.
\end{equation}
Furthermore, for all $\tau\in[0,1]$ we have
\begin{equation}
\label{prop: expansion of parallel transoport equation 3}
Z_\alpha{}^\beta(y,\tau y)=\delta_\alpha{}^\beta+\frac{\tau^2-1}6 \Riem^\beta{}_{\mu\alpha\nu}(0)\,y^\mu y^\nu -\frac{\tau^3-1}6 \frac{\partial^2\Gamma^\beta{}_{\mu\alpha}}{\partial y^\nu \partial y^\rho}(0) \,y^\mu y^\nu y^\rho+O(|y|^4)\,
\end{equation}
and
\begin{equation}
\label{prop: expansion of parallel transoport equation 4}
Z^\alpha{}_\beta(y,\tau y)=\delta^\alpha{}_\beta-\frac{\tau^2-1}6 \Riem^\alpha{}_{\mu\beta\nu}(0)\, y^\mu y^\nu
+\frac{\tau^3-1}6 \frac{\partial^2\Gamma^\alpha{}_{\mu\beta}}{\partial y^\nu \partial y^\rho}(0) \,y^\mu y^\nu y^\rho+O(|y|^4)\,.
\end{equation}
\end{proposition}

Note that formula \eqref{prop: expansion of parallel transoport equation 1} agrees with \cite[formula (7.1)]{dirac}.

\begin{proof}[Proof of Proposition~\ref{prop: expansion of parallel transoport}]
Let us start with a vector $v^\alpha$ at the origin and let us parallel transport it along a straight line to $y$ --- the unique shortest geodesic in the chosen coordinate system. This means imposing the condition
\begin{equation}
\label{proof expansion of parallel transoport equation 1}
y^\mu\nabla_\mu v^\beta= y^\mu
\left(
\partial_\mu v^\beta+\Gamma^\beta{}_{\mu\alpha}v^\alpha
\right)=0
\end{equation}
along the straight line.
Let us now parameterise our straight line with parameter $t\in[0,1]$ as $z(t)=ty$. Then
\begin{equation}
\label{proof expansion of parallel transoport equation 2}
y^\mu
\partial_\mu v^\beta
=\dot v^\beta.
\end{equation}
Formulae \eqref{proof expansion of parallel transoport equation 1} and \eqref{proof expansion of parallel transoport equation 2} imply
\begin{multline}
\label{proof expansion of parallel transoport equation 3}
\left.v^\beta\right|_{t=1}
=
\left.v^\beta\right|_{t=0}
-
\int_0^1
y^\mu\Gamma^\beta{}_{\mu\gamma}(ty)\,v^\gamma(ty)
\,dt
\\
=
\left.v^\beta\right|_{t=0}
-
\int_0^1
y^\mu
\left[
\frac{\partial\Gamma^\beta{}_{\mu\alpha}}{\partial y^\nu}(0)\,y^\nu\,t
+
\frac12\frac{\partial^2\Gamma^\beta{}_{\mu\alpha}}{\partial y^\nu\partial y^\rho}(0)\,y^\nu y^\rho\,t^2
+
O(|y|^3)
\right]
\left[
\left.v^\alpha\right|_{t=0}
+
O(|y|^2)
\right]
\,dt\,.
\end{multline}

In the above equation we used the fact that
\begin{equation*}
\label{proof expansion of parallel transoport equation 4}
\left.\dot{v}^\alpha\right|_{t=0}=0\,,
\end{equation*}
which follows from \eqref{proof expansion of parallel transoport equation 1}, \eqref{proof expansion of parallel transoport equation 2} and the fact that Christoffel symbols vanish at the origin.

Using the elementary differential-geometric identity
\begin{equation*}
\label{proof expansion of parallel transoport equation 5}
\frac{\partial\Gamma^\beta{}_{\mu\alpha}}{\partial y^\nu}(0)=-\frac13\left(
\Riem^\beta{}_{\mu\alpha\nu}(0)
+
\Riem^\beta{}_{\alpha\mu\nu}(0)
\right)
\end{equation*}
and performing integration in \eqref{proof expansion of parallel transoport equation 3}, we obtain
\begin{equation*}
\label{proof expansion of parallel transoport equation 6}
\left.v^\beta\right|_{t=1}
=
\left.v^\beta\right|_{t=0}
+
\frac{1}{6}
\Riem^\beta{}_{\mu\alpha\nu}(0)
\,
y^\mu
y^\nu
\left.v^\alpha\right|_{t=0}
-
\frac16
\frac{\partial^2\Gamma^\beta{}_{\mu\alpha}}{\partial y^\nu\partial y^\rho}(0)\,y^\mu y^\nu y^\rho\left.v^\alpha\right|_{t=0}
+O(|y|^4)\,.
\end{equation*}
The latter implies \eqref{prop: expansion of parallel transoport equation 1}.

Combining \eqref{prop: expansion of parallel transoport equation 1} with  \eqref{Q acting on 1-forms 3a} one immediately obtains \eqref{prop: expansion of parallel transoport equation 2}.

In order to obtain \eqref{prop: expansion of parallel transoport equation 3} let us parallel transport along our straight line from $y$ to $\tau y$. Arguing as above, we get
\begin{multline*}
\label{proof expansion of parallel transoport equation 7}
\left.v^\beta\right|_{t=\tau}
=
\left.v^\beta\right|_{t=1}
-
\int_1^{\tau}
y^\mu\,\Gamma^\beta{}_{\mu\gamma}(ty)\,v^\gamma(ty)
\,ds
\\
=
\left.v^\beta\right|_{t=1}
-
\int_1^{\tau}
y^\mu
\left[
-\frac{1}{3}
\Riem^\beta{}_{\mu\alpha\nu}(0)\,y^\nu\,t
+
\frac12\frac{\partial^2\Gamma^\beta{}_{\mu\alpha}}{\partial y^\nu\partial y^\rho}(0)\,y^\nu y^\rho\,t^2
+
O(|y|^3)
\right]
\left[
\left.v^\alpha\right|_{t=1}
+
O(|y|^2)
\right]
\,ds
\\
=
\left.v^\beta\right|_{t=1}
+
\frac{\tau^2-1}{6}
\Riem^\beta{}_{\mu\alpha\nu}(0)
\,
y^\mu
y^\nu
\left.v^\alpha\right|_{t=1}
-
\frac{\tau^3-1}6
\frac{\partial^2\Gamma^\beta{}_{\mu\alpha}}{\partial y^\nu\partial y^\rho}(0)\,y^\mu y^\nu y^\rho\left.v^\alpha\right|_{t=1}
+O(|y|^4)\,.
\end{multline*}
The latter implies \eqref{prop: expansion of parallel transoport equation 3} which, in turn, combined with \eqref{Q acting on 1-forms 3a} immediately  gives us \eqref{prop: expansion of parallel transoport equation 4}.
\end{proof}

\begin{remark}
For $\tau=0$ expansions \eqref{prop: expansion of parallel transoport equation 3} and \eqref{prop: expansion of parallel transoport equation 4} simplify to read
\begin{equation}
\label{prop: expansion of parallel transoport equation 3 tau zero}
Z_\alpha{}^\beta(y,0)=\delta_\alpha{}^\beta-\frac16 \Riem^\beta{}_{\mu\alpha\nu}(0)\,y^\mu y^\nu +\frac16 \frac{\partial^2\Gamma^\beta{}_{\mu\alpha}}{\partial y^\nu \partial y^\rho}(0) \,y^\mu y^\nu y^\rho+O(|y|^4)\,
\end{equation}
and
\begin{equation}
\label{prop: expansion of parallel transoport equation 4 tau zero}
Z^\alpha{}_\beta(y,0)=\delta^\alpha{}_\beta+\frac16 \Riem^\alpha{}_{\mu\beta\nu}(0)\, y^\mu y^\nu
-\frac16 \frac{\partial^2\Gamma^\alpha{}_{\mu\beta}}{\partial y^\nu \partial y^\rho}(0) \,y^\mu y^\nu y^\rho+O(|y|^4)\,.
\end{equation}
Formulae \eqref{prop: expansion of parallel transoport equation 1}, \eqref{prop: expansion of parallel transoport equation 2}, \eqref{prop: expansion of parallel transoport equation 3 tau zero} and \eqref{prop: expansion of parallel transoport equation 4 tau zero} agree with
\eqref{Q acting on 1-forms 3a extra} and \eqref{Q acting on 1-forms 3a}.
\end{remark}

\section{An alternative derivation of formula \eqref{main theorem 2 equation 1}}
\label{An alternative derivation of formula}

In this appendix we verify, using an alternative method, formula \eqref{main theorem 2 equation 1}, one of the main results of our paper.

In what follows we work in normal coordinates and we assume that curvature (but not its covariant derivative) vanishes at the origin, so that 
\begin{equation}
\label{14 September 2021 equation 1bis simplified}
g_{\alpha\beta}(x)
=
\delta_{\alpha\beta}
-
\frac16
(\nabla_\sigma \Riem_{\alpha\mu\beta\nu})(0)\,x^\sigma\,x^\mu\,x^\nu
+O(|x|^4).
\end{equation}
We continue using the notation $\boldsymbol{\Delta}=-(\dr \delta+\delta \dr)$ for the Hodge Laplacian on 1-forms \eqref{Hodge Laplacian};
$|\xi|$ stands for the Euclidean norm of $\xi$, whereas $\|\xi\|$ stands for the Riemannian norm thereof.

\begin{proposition}
\label{alternative proposition 1}
Let
\begin{equation*}
[s]_\alpha{}^\beta(x,\xi)
\sim
\|\xi\|^{-1}\delta_\alpha{}^\beta
+
[s_{-2}]_\alpha{}^\beta(x,\xi)
+
[s_{-3}]_\alpha{}^\beta(x,\xi)
+
[s_{-4}]_\alpha{}^\beta(x,\xi)
+
\dots
\end{equation*}
be the (left) symbol of the operator $(-\boldsymbol{\Delta})^{-1/2}$
(recall \eqref{Hodge Laplacian to the power s}).
Then
\begin{equation}
\label{29 July 2023 equation 6}
A_\mathrm{prin}(x,\xi)
=
-
\varepsilon_\beta{}^{\alpha\gamma}
\left(
\,([s_{-3}]_\alpha{}^\beta)_{x^\gamma}
+i
\,\xi_\gamma
\,[s_{-4}]_\alpha{}^\beta
\right)
+
O(|\xi|^{-3}|x|)\,.
\end{equation}
\end{proposition}

\begin{proof}
The claim follows from
\eqref{formula for Pplus minus Pminus},
\eqref{components of apt are zero},
\eqref{5 November 2021 equation 1}
and the fact that the full symbol of $\curl$ is
\eqref{principal symbol curl}.
\end{proof}

In the remainder of this appendix, we will derive explicit formulae for $[s_{-3}]_\alpha{}^\beta$ and $[s_{-3}]_\alpha{}^\beta$. The proofs are straightforward, hence omitted.

\begin{lemma}
\label{alternative lemma 1}
Let
\begin{equation*}
[q]_\alpha{}^\beta(x,\xi)
\sim
\|\xi\|^{2}\delta_\alpha{}^\beta
+
[q_1]_\alpha{}^\beta(x,\xi)
+
[q_0]_\alpha{}^\beta(x,\xi)
\end{equation*}
be the (left) symbol of the operator $-\boldsymbol{\Delta}$. We have
\begin{equation}
\label{10 August 2023 equation 16}
[q_1]_\alpha{}^\beta(x,\xi)
=
i\,a_\alpha{}^{\beta\gamma}{}_{\mu\nu}\,\xi_\gamma\,x^\mu\,x^\nu
+
O(|\xi|\,|x|^3)\,,
\end{equation}
\begin{equation}
\label{10 August 2023 equation 17}
[q_0]_\alpha{}^\beta(x)
=
b_\alpha{}^\beta{}_\nu\,x^\nu
+
O(|x|^2)\,,
\end{equation}
where
\begin{multline}
\label{10 August 2023 equation 18}
a_\alpha{}^{\beta\gamma}{}_{\mu\nu}
:=
\left[
\frac12\nabla_\mu\operatorname{Ric}^\gamma{}_\nu
-
\frac1{12}\nabla^\gamma\operatorname{Ric}_{\mu\nu}
\right]
\delta_\alpha{}^\beta
\\
-
\frac16
\left[
\nabla_\alpha\operatorname{Riem}^\gamma{}_\mu{}^\beta{}_\nu
-3
\nabla_\mu\operatorname{Riem}^\gamma{}_\alpha{}^\beta{}_\nu
+5
\nabla_\nu\operatorname{Riem}^\gamma{}_\mu{}^\beta{}_\alpha
\right],
\end{multline}
\begin{equation}
\label{10 August 2023 equation 19}
b_\alpha{}^\beta{}_\nu
:=
-
\frac{1}{6}
\nabla^\beta\operatorname{Ric}_{\alpha\nu}
+
\frac{1}{2}
\nabla_\alpha\operatorname{Ric}^\beta{}_\nu
+
\frac{1}{2}
\nabla_\nu\operatorname{Ric}_\alpha{}^\beta\,.
\end{equation}
\end{lemma}

\begin{lemma}
\label{alternative lemma 2}
The symbols
\begin{equation*}
[r]_\alpha{}^\beta(x,\xi)
\sim
\|\xi\|\delta_\alpha{}^\beta
+
[r_0]_\alpha{}^\beta(x,\xi)
+
[r_{-1}]_\alpha{}^\beta(x,\xi)
+
[r_{-2}]_\alpha{}^\beta(x,\xi)
+
\dots
\end{equation*}
and $s_\alpha{}^\beta$ of $(-\boldsymbol{\Delta})^{1/2}$ and $(-\boldsymbol{\Delta})^{-1/2}$ are expressed in terms of \eqref{10 August 2023 equation 16} and \eqref{10 August 2023 equation 17} via the following hierarchy of identities:
\begin{equation*}
\label{10 August 2023 equation 20}
[r_0]_\alpha{}^\beta=\frac1{2|\xi|}[q_1]_\alpha{}^\beta-\frac{1}{2i|\xi|^2} \xi^\mu(\|\xi\|)_{x^\mu}\,\delta_\alpha{}^\beta + O(|x|^3)\,,
\end{equation*}
\begin{equation*}
\label{10 August 2023 equation 21}
[r_{-1}]_\alpha{}^\beta
\\
=
\frac1{2|\xi|}[q_0]_\alpha{}^\beta
-\frac1{2i|\xi|^2} \xi^\mu ([r_0]_\alpha{}^\beta)_{x^\mu}
\\
+\frac1{4|\xi|}(|\xi|)_{\xi_\mu \xi_\nu}(\|\xi\|)_{x^\mu x^\nu}\,\delta_\alpha{}^\beta
\\
+
O(|\xi|^{-1}|x|^2)\,,
\end{equation*}
\begin{multline*}
\label{10 August 2023 equation 22}
[r_{-2}]_\alpha{}^\beta
=
-\frac1{2i|\xi|^2}
\xi^\mu([r_{-1}]_\alpha{}^\beta)_{x^\mu}
\\
+\frac1{4|\xi|} 
(|\xi|)_{\xi_\mu\xi_\nu}([r_0]_\alpha{}^\beta)_{x^\alpha x^\beta}
+\frac1{12i|\xi|}(|\xi|)_{\xi_\mu \xi_\nu\xi_\rho}(\|\xi\|)_{x^\mu x^\nu x^\rho}\,\delta_\alpha{}^\beta
\\
+
O(|\xi|^{-2}|x|)\,,
\end{multline*}
\begin{equation*}
\label{10 August 2023 equation 5}
[s_{-2}]_{\alpha}{}^\beta= -|\xi|^{-2}[r_{0}]_{\alpha}{}^\beta-
\frac{1}{i|\xi|^2}\,\xi^\mu\,(\|\xi\|^{-1})_{x^\mu} \,\delta_\alpha{}^\beta 
+
O(|\xi|^{-2}|x|^3)\,,
\end{equation*}
\begin{multline}
\label{10 August 2023 equation 6}
[s_{-3}]_\alpha{}^\beta
=
-
|\xi|^{-2}[r_{-1}]_\alpha{}^\beta
-\frac1{i|\xi|^2}
\xi^\mu\,([s_{-2}]_\alpha{}^\beta)_{x^\mu}
\\
+\frac1{2|\xi|} (|\xi|)_{\xi_\mu\xi_\nu}(\|\xi\|^{-1})_{x^\mu x^\nu}\, \delta_\alpha{}^\beta
+
O(|\xi|^{-3}|x|^2)\,,
\end{multline}
\begin{multline}
\label{10 August 2023 equation 7}
[s_{-4}]_\alpha{}^\beta
=
-
|\xi|^{-2}[r_{-2}]_\alpha{}^\beta
-\frac1{i|\xi|^2}\,
\xi^\mu\,([s_{-3}]_\alpha{}^\beta)_{x^\mu}
\\
+\frac1{2|\xi|}
(|\xi|)_{\xi_\mu\xi_\nu}([s_{-2}]_\alpha{}^\beta)_{x^\mu x^\nu}
+\frac{1}{6i|\xi|}(|\xi|)_{\xi_\mu\xi_\nu\xi_\rho}(\|\xi\|^{-1})_{x^\mu x^\nu x^\rho}\, \delta_\alpha{}^\beta
\\
+
O(|\xi|^{-4}|x|)\,.
\end{multline}
\end{lemma}

\begin{remark}
Lemmata \ref{alternative lemma 1} and \ref{alternative lemma 2} are true in any dimension $d$, not only in dimension $d=3$.
\end{remark}

In order to verify our formula \eqref{main theorem 2 equation 1} for $A_\prin\,$, it only remains to compute (for example, with the help of computer algebra) the right-hand sides of \eqref{10 August 2023 equation 6} and \eqref{10 August 2023 equation 7}, and substitute the resulting expressions into \eqref{29 July 2023 equation 6}. 

For instance, for the particular choice of metric determined by
\begin{equation*}
\Ric(x)
=
\begin{pmatrix}
0 & 0 & 0\\
0 & 0 & x^1\\
0 & x^1 & 0\\
\end{pmatrix}
+
O(|x|^2)\
\end{equation*}
(recall \eqref{14 September 2021 equation 5} and \eqref{14 September 2021 equation 1bis simplified}) we get
\begin{equation}
\label{s-3 and s-4}
\left.s_{-3}\right|_{(x,\xi_0)}=-\frac1{12}
\begin{pmatrix}
0 & 3x^3 & 3x^2\\
-x^3 & 0 &  3x^1\\
-x^2 & 3x^1 & 0
\end{pmatrix}
+
O(|x|^2), 
\quad
\left.s_{-4}\right|_{(x,\xi_0)}=-\frac{i}{2}
\begin{pmatrix}
0 & 0 & 0\\
1 & 0 &  0\\
0 & 0 & 0
\end{pmatrix}
+
O(|x|)\,,
\end{equation}
where $\xi_0$ is given by \eqref{xi0}.
Substituting \eqref{s-3 and s-4} into \eqref{29 July 2023 equation 6} we obtain
\begin{equation*}
A_\prin(0,\xi_0)=-\frac12\,,
\end{equation*}
which agrees with \eqref{main theorem 2 equation 1} --- see also \eqref{Proof of Aprin equation 8}.

\end{appendices}


\begin{thebibliography}{42}
\addcontentsline{toc}{section}{References}

\bibitem{AS}
M.~Abramowitz and I.A.~Stegun, \emph{Handbook of Mathematical Functions: With Formulas, Graphs, and Mathematical Tables}. Dover Books on Advanced Mathematics, Dover Publications, 1965.
DOI: \href{https://doi.org/10.1119/1.15378}{10.1119/1.15378}.

\bibitem{asymm1}
M.F.~Atiyah, V.K.~Patodi and I.M.~Singer, 
Spectral asymmetry and Riemannian geometry,
{\it Bull. London Math. Soc.} \textbf{5} (1973) 229--234.
DOI: \href{https://doi.org/10.1112/blms/5.2.229}{10.1112/blms/5.2.229}.

\bibitem{asymm2}
M.F.~Atiyah, V.K.~Patodi and I.M.~Singer, 
Spectral asymmetry and Riemannian geometry I,
{\it Math. Proc. Camb. Phil. Soc.} \textbf{77} (1975) 43--69.
DOI: \href{https://doi.org/10.1017/S0305004100049410}{10.1017/S0305004100049410}.

\bibitem{asymm3}
M.F.~Atiyah, V.K.~Patodi and I.M.~Singer, 
Spectral asymmetry and Riemannian geometry II,
{\it Math. Proc. Camb. Phil. Soc.} \textbf{78} (1975) 405--432.
DOI: \href{https://doi.org/10.1017/S0305004100051872}{10.1017/S0305004100051872}.

\bibitem{asymm4}
M.F.~Atiyah, V.K.~Patodi and I.M.~Singer, 
Spectral asymmetry and Riemannian geometry III,
{\it Math. Proc. Camb. Phil. Soc.} \textbf{79} (1976) 71--99.
DOI: \href{https://doi.org/10.1017/S0305004100052105}{10.1017/S0305004100052105}.


%

%
%
%

\bibitem{baer92}
C.~B\"ar, 
The Dirac operator on homogeneous spaces and its spectrum on 3-dimensional lens spaces,
{\it Arch. Math.} {\bf 59} (1992) 65--79.
DOI: \href{https://doi.org/10.1007/BF01199016}{10.1007/BF01199016}.


\bibitem{baer_curl}
C.~B\"ar, 
The curl operator on odd-dimensional manifolds,
{\it J. Math. Phys.} {\bf 60} (2019) 031501.
DOI: \href{https://doi.org/10.1063/1.5082528}{10.1063/1.5082528}.

%
%


%
%

\bibitem{birman_curl}
M.~Sh.~Birman and M.~Z.~Solomyak,
The Weyl asymptotics of the spectrum of the Maxwell operator for domains with a Lipschitz boundary, 
{\it Vestn. Leningr. Univ. Math.} {\bf 20} no.~3 (1987) 15--21.

\bibitem{birman_curl2}
M.~Sh.~Birman and M.~Z.~Solomyak,
$L^2$-Theory of the Maxwell operator in arbitrary domains, 
{\it Uspekhi Mat. Nauk} {\bf 42} no.~6 (1987) 61--76; {\it Russian Math. Surveys} {\bf 42} no.~6 (1987) 75--96.
DOI: \href{https://doi.org/10.1070/RM1987v042n06ABEH001505}{10.1070/RM1987v042n06ABEH001505}.



\bibitem{bismut}
J.-M.~Bismut and D.~S.~Freed, 
The analysis of elliptic families. II. Dirac operators, eta invariants, and the holonomy theorem, 
{\it Comm. Math. Phys.} \textbf{107} (1986) 103--163.
DOI: \href{https://doi.org/10.1007/BF01206955}{10.1007/BF01206955}.



\bibitem{branson}
T.~P.~Branson and P.~B.~Gilkey, Residues of the eta function for an operator
of Dirac type,
{\it J.~Funct.~Anal.} \textbf{108} (1992) 47--87.
DOI: \href{https://doi.org/10.1016/0022-1236(92)90146-A}{10.1016/0022-1236(92)90146-A}.

\bibitem{expansions}
L.~Brewin,
Riemann normal coordinate expansions using Cadabra.
Preprint \href{https://doi.org/10.48550/arXiv.0903.2087}{arXiv:0903.2087v3}.

%
%

\color{black}
\bibitem{bruning}
J.~Br\"uning and M.~Lesch,
On the $\eta$-invariant of certain nonlocal boundary value problems,
{\it Duke Math. J.} {\bf 96} no.~2 (1999) 425--468. 
DOI: \href{https://doi.org/10.1215/S0012-7094-99-09613-8}{10.1215/S0012-7094-99-09613-8}.
\color{black}

%
\bibitem{diagonalisation}
M.~Capoferri,
Diagonalization of elliptic systems via pseudodifferential projections,
{\it J. Differential Equations} {\bf 313} (2022) 157--187.
DOI: \href{https://doi.org/10.1016/j.jde.2021.12.032}{10.1016/j.jde.2021.12.032}.

%

\bibitem{wave}
M.~Capoferri, M.~Levitin and D.~Vassiliev,
Geometric wave propagator on Riemannian manifolds,
\emph{Comm. Anal. Geom.} {\bf 30} no.~8 (2022) 1713--1777.
DOI: \href{https://dx.doi.org/10.4310/CAG.2022.v30.n8.a2}{10.4310/CAG.2022.v30.n8.a2}.

\bibitem{obstructions}
M.~Capoferri, G.~Rozenblum, N.~Saveliev and D.~Vassiliev,
Topological obstructions to the diagonalisation of pseudodifferential systems
{\it Proc. Amer. Math. Soc. (Ser. B)} {\bf 9} (2022) 472--486.
DOI: \href{https://doi.org/10.1090/bproc/147}{10.1090/bproc/147}.



\bibitem{dirac}
M.~Capoferri and D.~Vassiliev,
Global propagator for the massless Dirac operator and spectral asymptotics.
{\it Integral Equations Operator Theory} {\bf 94} (2022) 30.
DOI: \href{https://doi.org/10.1007/s00020-022-02708-1}{10.1007/s00020-022-02708-1}.

\bibitem{part1}
M.~Capoferri and D.~Vassiliev,
Invariant subspaces of elliptic systems I: pseudodifferential projections.
{\it J.~Funct.~Anal.} {\bf 282} no.~8 (2022) 109402.
DOI: \href{https://doi.org/10.1016/j.jfa.2022.109402}{10.1016/j.jfa.2022.109402}.

\bibitem{part2}
M.~Capoferri and D.~Vassiliev,
Invariant subspaces of elliptic systems II: spectral theory,
{\it J.  Spectr. Theory} {\bf 12} no.~1 (2022) 301--338.
DOI: \href{https://doi.org/10.4171/JST/402}{10.4171/JST/402}.

\bibitem{conjectures}
M. Capoferri and D. Vassiliev,
A microlocal pathway to spectral asymmetry: curl and the eta invariant.
Preprint \href{https://doi.org/10.48550/arXiv.2502.18307}{arXiv:2502.18307} (2025).


%

\color{black}
\bibitem{colbois06}
B.~Colbois and A.~El Soufi, 
Eigenvalues of the Laplacian acting on p-forms and metric conformal
deformations, {\it Proc. Amer. Math. Soc.} {\bf 134} (2006) 715--721.
\href{http://dx.doi.org/10.1090/S0002-9939-05-08005-6}{10.1090/S0002-9939-05-08005-6}
\color{black}

%
%

\bibitem{filonov_curl1}
M.N.~Demchenko and N.~Filonov, 
Spectral asymptotics of the Maxwell operator on Lipschitz manifolds with boundary, in: \emph{Spectral Theory of Differential Operators: M.Sh.~Birman 80th Anniversary Collection} (Amer. Math. Soc., Providence, RI, 2008), 73--90.
DOI: \href{https://doi.org/10.1090/trans2/225/05}{10.1090/trans2/225/05}.

\bibitem{DLMF}
\emph{NIST Digital Library of Mathematical Functions}. \url{https://dlmf.nist.gov/}, Release 1.1.9 of 2023-03-15. F.W.J.~Olver, A.B.~Olde Daalhuis, D.W.~Lozier, B.I.~Schneider, R.F.~Boisvert, C.W.~Clark, B.R.~Miller, B.V.~Saunders, H.S.~Cohl, and M.A.~McClain, eds.

\bibitem{dowker66}
J.S.~Dowker and Yih P.~Dowker, Interactions of Massless Particles of Arbitrary Spin, 
{\it Proc. Roy. Soc. A} {\bf 294} no.~1437 (1966) 175--94.
DOI: \href{https://doi.org/10.1098/rspa.1966.0202}{10.1098/rspa.1966.0202}.


\bibitem{dowker84}
J.S.~Dowker, {\it Vacuum energy in a squashed Einstein universe},
in: Quantum Theory of Gravity,
Essays in honor of the 60th birthday of Bryce S.~DeWitt, Steven M.~Christensen (Ed.), Hilger, Bristol (1984) 103--124.


%
\bibitem{DuHo}
J.~J.~Duistermaat and L.~H\"ormander,
Fourier integral operators. II.,
{\it Acta Math.} \textbf{128} no.~3--4 (1972) 183--269.
DOI: \href{https://doi.org/10.1007/BF02392165}{10.1007/BF02392165}.

\bibitem{peralta1}
A.~Enciso, W.~Gerner and D.~Peralta-Salas,
Optimal convex domains for the first curl eigenvalue in dimension three,
{\it Trans. Amer. Math. Soc.}
\textbf{377} (2024) 4519--4540.
DOI: \href{https://doi.org/10.1090/tran/8914}{10.1090/tran/8914}.

\color{black}
\bibitem{peralta3}
A.~Enciso, W.~Gerner and D.~Peralta-Salas,
Optimal metrics for the first curl eigenvalue on 3-manifolds,
{\it Calc. Var.} {\bf 64} (2025) 146.
DOI: 
\href{https://doi.org/10.1007/s00526-025-02995-7}{10.1007/s00526-025-02995-7}
\color{black}

\bibitem{peralta2}
A.~Enciso and D.~Peralta-Salas,
Non-existence of axisymmetric optimal domains with smooth boundary for the first curl eigenvalue, {\it Ann. Sc. Norm. Super. Pisa Cl. Sci.} {\bf XXIV} (2023) 311--327.
DOI: \href{https://doi.org/10.2422\%2F2036-2145.202010\_008}{10.2422\%2F2036-2145.202010\_008}.

\color{black}
\bibitem{etnyre}
J.~Etnyre and R.~Ghrist,
Contact topology and hydrodynamics I: Beltrami fields and the Seifert Conjecture, {\it Nonlinearity} \textbf{13} (2000) 441--458.
DOI: \href{https://doi.org/10.1088/0951-7715/13/2/306}{10.1088/0951-7715/13/2/306}
\color{black}

\bibitem{sphere}
Y.-L.~Fang, M.~Levitin and D.~Vassiliev,
Spectral analysis of the Dirac operator on a 3-sphere,
{\it Operators and Matrices} \textbf{12} (2018) 501--527.
DOI: \href{https://doi.org/10.7153/oam-2018-12-31}{10.7153/oam-2018-12-31}.



\bibitem{filonov_curl2}
N.~Filonov, 
Weyl asymptotics of the spectrum of the Maxwell operator in Lipschitz domains of arbitrary dimension, 
{\it Algebra i Anal.} {\bf 25} no.~1 (2013) 170--215.
DOI: \href{https://doi.org/10.1090/S1061-0022-2013-01282-9}{10.1090/S1061-0022-2013-01282-9}.


\color{black}
\bibitem{gerner}
W.~Gerner,
{\it Minimisation problems in ideal magnetohydrodynamics}, PhD thesis, RWTH Aachen University (2020).
DOI: \href{https://doi.org/10.18154/RWTH-2020-12414}{0.18154/RWTH-2020-12414}.
\color{black}

\bibitem{gibbons}
G.W.~Gibbons,
Spectral asymmetry and quantum field theory in curved spacetime,
{\it Ann. Physics} \textbf{25} no.~1 (1980) 98--116.
DOI: \href{https://doi.org/10.1016/0003-4916(80)90120-7}{10.1016/0003-4916(80)90120-7}.


%

\bibitem{gureev}
T.E.~Gureev, 
Exact asymptotics of the spectrum of the Maxwell operator in a solid resonator, {\it Funct. Anal. Appl.} {\bf 24} no.~3 (1990) 235--237.
DOI: \href{https://doi.org/10.1007/BF01077970}{10.1007/BF01077970}.

\bibitem{hintz}
P.~Hintz, 
Resonance expansions for tensor-valued waves on asymptotically Kerr--de Sitter spaces,
{\it J. Spectr. Theory} {\bf 7} no.~2 (2017) 519--557.
DOI: \href{https://doi.org/10.4171/jst/171}{10.4171/jst/171}.

\bibitem{hitchin}
N.~Hitchin,
Harmonic spinors,
{\it Adv. Math.} {\bf 14} no.~1 (1974) 1--55.
DOI: \href{https://doi.org/10.1016/0001-8708(74)90021-8}{10.1016/0001-8708(74)90021-8}.

\bibitem{Hor 1968}
L.~H\"ormander,
The spectral function of an elliptic operator,
{\it Acta Math.} \textbf{121} (1968) 193--218.
DOI: \href{https://doi.org/10.1007/BF02391913}{10.1007/BF02391913}.

%
%
%

\bibitem{Ivr82}
V.~Ivrii,
Accurate spectral asymptotics for elliptic operators that act in vector bundles
{\it Funct. Anal. Appl.} \textbf{16} (1982) 101--108.
DOI: \href{https://doi.org/10.1007/BF01081624}{10.1007/BF01081624}.

%
%
%

\bibitem{JS}
D.~Jakobson and A.~Strohmaier,
High energy limits of Laplace-type and Dirac-type eigenfunctions and frame flows,
{\it Comm. Math. Phys.} {\bf 270} (2007) 813--833.
DOI: \href{https://doi.org/10.1007/s00220-006-0176-0}{10.1007/s00220-006-0176-0}.

\color{black}
\bibitem{jammes}
P.~Jammes, 
Minoration conforme du spectre du laplacien de Hodge-de Rham, {\it Manuscripta
Math.} {\bf 123} (2007) 15--23.
DOI: \href{https://doi.org/10.1007/s00229-007-0080-8}{10.1007/s00229-007-0080-8}
\color{black}

\bibitem{jost}
J.~Jost, 
{\it Riemannian Geometry and Geometric Analysis},
Springer-Verlag,  2011.
DOI: \href{https://doi.org/10.1007/978-3-319-61860-9}{10.1007/978-3-319-61860-9}.

\color{black}
\bibitem{kepplinger}
W.~Kepplinger,
On spectral simplicity of the Hodge Laplacian and Curl Operator along paths of metrics,
\emph{Trans. Amer. Math. Soc.} {\bf 377} no.~11 (2024) 7829--7845.
DOI: \href{https://doi.org/10.1090/tran/9221}{10.1090/tran/9221}.
\color{black}

\bibitem{milnor2}
M.A.~Kervaire and J.W.~Milnor,
Groups of homotopy spheres: I,
\emph{Ann. of Math.} \textbf{77} no.~3 (1963) 504--537.
DOI: \href{https://doi.org/10.2307/1970128}{10.2307/1970128}.


\bibitem{lauret}
E.A.~Lauret, 
The smallest Laplace eigenvalue of homogeneous 3-spheres,
{\it Bull.~LMS} {\bf 51} no.~1 (2019) 49--69.
DOI: \href{https://doi.org/10.1112/blms.12213}{10.1112/blms.12213}.



\bibitem{lerner}
N.~Lerner and F.~Vigneron,
On some properties of the curl operator and their consequences for the Navier-Stokes system, {\it Comm. Math. Research} {\bf 38} (2022) 449--497.
DOI: \href{https://doi.org/10.4208/cmr.2021-0106}{10.4208/cmr.2021-0106}.

\color{black}
\bibitem{lin1}
F.~Lin,
Monopole Floer homology and the
spectral geometry of three-manifolds,
{\it Comm. Anal. Geom.}
{\bf 28} no.~5 (2020) 1211--1219.
DOI: \href{https://dx.doi.org/10.4310/CAG.2020.v28.n5.a2}{10.4310/CAG.2020.v28.n5.a2}

\bibitem{lin2}
F.~Lin and M.~Lipnowski,
The Seiberg-Witten equations and the length spectrum of hyperbolic three-manifolds,
{\it J. Amer. Math. Soc.} {\bf 35} (2022) 233--293.
DOI: \href{https://doi.org/10.1090/jams/982}{10.1090/jams/982}
\color{black}

\bibitem{lipkin}
D.M.~Lipkin,
Existence of a new conservation law in electromagnetic theory,
{\it J. Math. Phys.} {\bf 5} (1964) 696--700.
DOI: \href{https://doi.org/10.1063/1.1704165}{10.1063/1.1704165}.

%
%
%

\bibitem{lotay}
J.~Lotay,
Stability of coassociative conical singularities,
\emph{Comm. Anal. Geom.} {\bf 20} no.~4 (2012) 803--867.
DOI: \href{https://dx.doi.org/10.4310/CAG.2012.v20.n4.a5}{10.4310/CAG.2012.v20.n4.a5}.

\color{black}
\bibitem{loya}
P.~Loya,
Residue Traces and Kernel Expansions of Pseudodifferential Operators, 
{\it Modern trends in geometry and topology},  Cluj Univ.\ Press, Cluj-Napoca (2006) 251--264.
\color{black}


\bibitem{millson}
J.J.~Millson, 
\emph{Chern-Simons invariants of constant curvature manifolds}. Ph.D. thesis, University of California, Berkeley, 1973.

\bibitem{milnor1}
J.W.~Milnor,
On manifolds homeomorphic to the 7-sphere,
\emph{Ann. of Math.} \textbf{64} no.~2 (1956) 399--405.
DOI: \href{https://doi.org/10.1142/9789812836878\_0001}{10.1142/9789812836878\_0001}.

\bibitem{minak}
S.~Minakshisundaram and {{\AA}}.~Pleijel,
Some properties of the eigenfunctions of the Laplace-operator on Riemannian manifolds,
{\it Canadian J. Math.} \textbf{1} (1949) 242--256.
DOI: \href{https://doi.org/10.4153/CJM-1949-021-5}{10.4153/CJM-1949-021-5}.

\bibitem{Uwe Muller}
U.~Muller, C.~Schubert and A.E.M.~van de Ven,
A closed formula for the Riemann normal coordinate expansion,
{\it Gen. Relativ. Gravit.} \textbf{31} (1999) 1759--1768.
DOI: \href{https://doi.org/10.1023/A:1026718301634}{10.1023/A:1026718301634}.

%
%
%
%

\bibitem{safarov_curl}
Yu.~Safarov,
Asymptotic behavior of the spectrum of the Maxwell operator, 
{\it J. Sov. Math.} {\bf 27} (1984) 2655--2661.
DOI: \href{https://doi.org/10.1007/BF01103726}{10.1007/BF01103726}.

%
%
%

\bibitem{schmudgen}
K.~Schm\"udgen,
{\it Unbounded self-adjoint operators on Hilbert space},
Graduate Texts in Mathematics {\bf 265}, Springer Netherlands, 2012.
DOI: \href{https://doi.org/10.1007/978-94-007-4753-1}{10.1007/978-94-007-4753-1}.

\bibitem{Schoen and Yau}
R.~Schoen and S.-T.~Yau,
{\it Lectures on differential geometry},
International Press, 2010.

%

\bibitem{shubin}
M.A.~Shubin,
{\it Pseudodifferential operators and spectral theory},
Springer, 2001.
DOI: \href{https://doi.org/10.1007/978-3-642-56579-3}{10.1007/978-3-642-56579-3}.

\bibitem{singerICM}
I.M.~Singer, 
Eigenvalues of the Laplacian and invariants of manifolds,
{\it Proceedings of the International Congress of Mathematicians} {\bf 1} (1974) 187--200. 

\bibitem{tang and cohen}
Y.~Tang and A.E.~Cohen,
Optical chirality and its interaction with matter,
\emph{Phys. Rev. Lett.} {\bf 104} (2010) 163901.
DOI: \href{https://doi.org/10.1103/PhysRevLett.104.163901}{10.1103/PhysRevLett.104.163901}.

%

\bibitem{tanno}
S.~Tanno,
The first eigenvalue of the Laplacian on spheres,
\emph{Tohoku Math. J.} {\bf 31} (1979) 179--185. DOI: \href{https://doi.org/10.2748/tmj/1178229837}{10.2748/tmj/1178229837}.

%
%
%
\bibitem{warner}
F.W.~Warner,
{\it Foundations of differentiable manifolds and Lie groups},
Springer, 1983.
DOI: \href{https://doi.org/10.1007/978-1-4757-1799-0}{10.1007/978-1-4757-1799-0}.

\color{black}
\bibitem{wodzicki}
M.~Wodzicki,
Local invariants of spectral asymmetry,
{\it Invent. Math.} {\bf 75} (1984) 143--177.
DOI: \href{https://doi.org/10.1007/BF01403095}{10.1007/BF01403095}.
\color{black}

%

\end{thebibliography}
\end{document}